\documentclass[12pt,a4paper]{scrartcl}
\usepackage{multicol,multirow,setspace}
\usepackage[english]{babel}			
\usepackage{hyperref} 				
\usepackage{paralist}				
\usepackage{graphicx}				
\usepackage{soul}				
\usepackage{mdframed}				
\usepackage{amsfonts} 				
\usepackage{amsmath} 				
\usepackage{mathabx}				
\usepackage[arrow, matrix, curve]{xypic} 		
\usepackage{amsopn}				
\usepackage{amsthm} 				
\usepackage{extarrows}				
\usepackage{mathtools}				
	\DeclareMathOperator{\Lip}{Lip}
	\DeclareMathOperator{\id}{id}
	\DeclareMathOperator{\ind}{ind}

	\DeclareMathOperator{\im}{im}

	\DeclareMathOperator{\interi}{int}

	\DeclareMathOperator{\GL}{GL}

	\DeclareMathOperator{\pr}{pr}

	\DeclareMathOperator{\Hol}{Hol}
	\DeclareMathOperator{\Diff}{Diff}

	\newcommand{\C}{\mathbb{C}}

	\newcommand{\G}{\mathcal{G}}

	\newcommand{\K}{\mathbb{K}}
	
	\newcommand{\N}{\mathbb{N}}

	\newcommand{\R}{\mathbb{R}}
	
	\newcommand{\U}{\mathcal{U}}

	\newcommand{\ep}{\varepsilon}
	\renewcommand{\a}{\alpha}	
	\newcommand{\g}{\gamma}
	\renewcommand{\gg}{\Gamma}	
	\renewcommand{\d}{\delta}	

	\renewcommand{\i}{\iota}	
	\renewcommand{\l}{\lambda}	
	
	\newcommand{\m}{\mu}

	\newcommand{\p}{\pi}
	
	\newcommand{\ph}{\varphi}
	\newcommand{\gph}{\Phi}
	
	\newcommand{\ps}{\psi}
	\newcommand{\gps}{\Psi}
	\renewcommand{\o}{\omega}		

	\renewcommand{\le}{\left}
	\newcommand{\ri}{\right}
	\newcommand{\set}[1]{\le\{#1\ri\}}

	\newcommand{\ra}{\rightarrow}
	
	\newcommand{\hra}{\hookrightarrow}

	\newcommand{\subs}{\subseteq}
	\newcommand{\sups}{\supseteq}
	\newcommand{\ol}[1]{\overline{#1}}
	\newcommand{\ti}{\times}

	\newcommand{\what}{\widehat}
	\newcommand{\til}{\tilde}
	\newcommand{\wtil}{\widetilde}

	\newcommand{\fr}[2]{\frac{#1}{#2}}
	\newcommand{\ms}{\mapsto}
	
	\newcommand{\str}{\mathrm{st}}
	
	\newcommand{\ci}{\circ}
	\newcommand{\co}{\colon}
	\renewcommand{\-}{\item}
	\newcommand{\inv}{^{-1}}

	\newcommand{\8}{\infty}

	\newcommand{\tx}[1]{\text{ #1 }}
	\newcommand{\bl}{{\scriptscriptstyle \bullet}}

\theoremstyle{plain}

\newtheorem{satz}{Satz}[section]		
\newtheorem{theorem}[satz]{Theorem}
\newtheorem{remark}[satz]{Remark}
\newtheorem{lemma}[satz]{Lemma}
\newtheorem{satz/definition}[satz]{Satz/Definition}
\newtheorem{theorem/definition}[satz]{Theorem/Definition}
\newtheorem{definition/theorem}[satz]{Definition/Theorem}
\newtheorem{definition/lemma}[satz]{Definition/Lemma}

\newtheorem{lemma/definition}[satz]{Lemma/Definition}

\newtheorem{theoo}{Theorem}

\theoremstyle{definition}

\newtheorem{theoreme}[theoo]{Theorem}
\newtheorem{definition}[satz]{Definition}
\newtheorem{convention}[satz]{Convention}

\title{A Lie group structure on the group of real analytic diffeomorphisms of a compact real analytic manifold with corners}
\date{}
\setdefaultitem{\textbullet}{-}{}{}		
\setdefaultenum{(a)}{(i)}{}{}			

\author{Jan Milan Eyni}
\date{}
\usepackage{tikz}
\renewcommand*\circled[1]{\tikz[baseline=(char.base)]{
    \node[shape=circle,draw,inner sep=0.4pt] (char) {#1};}}

\begin{document}

\maketitle

\begin{abstract}
We construct a smooth Lie group structure on the group of real analytic diffeomorphisms of a compact  analytic manifold with corners. This generalises the known analogous results in the situation where the real analytic manifold has no corners. Additionally our approach uses a different construction. 
\end{abstract}
{\footnotesize
{\bf Jan Milan Eyni},
Universit\"{a}t Paderborn,
Institut f\"{u}r Mathematik,
Warburger Str.\ 100,
33098 Paderborn, Germany;
{\tt janme@math.upb.de}\\[2mm]
}
\section*{Introduction}

The prime example of an infinite-dimensional Lie group is the diffeomorphism group $\Diff(M)$  of a finite-dimensional manifold $M$. First we categorise different approaches how to construct a Lie group structure on $\Diff(M)$. Then we recall the exact conditions for the existence of a Lie group structure on $\Diff(M)$.
It is well known that if $M$ is a smooth compact manifold (with $\dim(M)\geq 1$), then there exists no structure of a Banach manifold on $\Diff(M)$ (\cite[p. 457]{Kriegel u. Michor}, \cite{Omori}). Hence one has to model $\Diff(M)$ over a more general topological vector space. It turns out that locally convex spaces are the right choice. Actually it is possible to model $\Diff(M)$ over the vector fields of $M$ (in our notation $\gg(TM)$) as an infinite-dimensional Lie group. But now the question arises what differential calculus, on a locally convex space, one should use to obtain a differentiable structure on $\Diff(M)$. There are several approaches to differential calculus on locally convex spaces (for details we recommend \cite{Keller}). Among the  most popular approaches is the convenient setting, invented by Fr{\"o}licher, Kriegl and Michor (see \cite{Kriegel u. Michor}). A map is called smooth in the convenient setting if it is smooth along smooth curves (see \cite[Definition 3.11]{Kriegel u. Michor}). Of course this differential calculus is inspired by the Boman-theorem (see \cite[Theorem 3.4]{Kriegel u. Michor} and \cite{Boman}). The second popular approach is the differential calculus Known as Keller's $C^k_c$-theory  (obviously the name is inspired by \cite{Keller}). In this approach a continuous map is called continuously differentiable if all directional derivatives $df(x,v)$ exist and the map $(x,v) \ms df(x,v)$ is continuous. For details to this approach we recommend \cite{Milnor}, \cite{Hamilton} and  \cite[Chapter 1]{GloecknerNeeb}. Because Milnor used this differential calculus to turn the diffeomorphism group into a Lie group (see \cite{Milnor}), Lie groups constructed with the Keller's $C^k_c$-theory differential calculus are sometimes called Milnor-Lie groups. One can show that the convenient differential calculus and Keller's $C^k_c$-theory are equivalent on Fr{\'e}chet spaces (see \cite[p. 270]{GloecknerBertramNeeb} and \cite[Theorem 4.11]{Kriegel u. Michor}) but beyond the Fr{\'e}chet case this is false. For example a map that is smooth in the convenient sense need not be continuous (see \cite[p. 1]{Gloeckner2}).

Having chosen a  differential calculus, one has to choose a strategy how to turn $\Diff(M)$ into a Lie group. There are basically two different strategies. The first one (and most common one), is to turn the space of ($C^k$ respectively smooth respectively analytic) mappings from $M$ to $M$ (in our notation $C^k(M;M)$ with $k \in \set{r,\8,\o}$) into an infinite-dimensional manifold.  
To this end one chooses a Riemannian metric on $M$ and obtains a Riemannian exponential function $\exp$. For small $\eta \in \gg(TM)$ one can define the map $\gps_\eta:=\exp\ci \eta$. Now it turns out that in many cases it is possible to obtain a manifold structure on  the mapping space $C^k(M;M)$ by charts similar to $\gps\co \eta \ms \gps_\eta$. The second step in this strategy is to show that $\Diff(M)$ is an open submanifold of $C^k(M;M)$ and that the group operations have the required differential property (e.g. $C^r$, smooth or real analytic). In the following we call this strategy the ``global approach''(in the table further down we cite  articles that used this approach).

The second approach leads to the same Lie group structure on $\Diff(M)$ but its construction is very different. Again one chooses a Riemannian metric on $M$. With the help of the map $\gps \co \eta \ms \gps_\eta$ one obtains a manifold structure on a subset of $\Diff(M)$ that contains the identity $\id_M$. Now one uses the theorem of local description of Lie groups to extend the manifold structure to $\Diff(M)$ and to turn it into a Lie group. In the following we call this strategy the ``local approach''. This approach first was used in \cite{Gloeckner2}.

In the following table we cite  different articles that constructed Lie group structures on diffeomorphism groups. We emphasize that this list is not comprehensive. The list just contains the cases that are of interest for this paper.\\

\begin{tabular}{|l|l|l|l|}
\hline 
\multirow{3}{*}{{M}}  & \multirow{2}{*}{{Global,}} & \multirow{2}{*}{{Global,}} & \multirow{2}{*}{{Local,}}\\
&&&\\
&{Convenient}&{Keller-$C^k_c$}&{Keller-$C^k_c$}\\
\hline
$C^\8$, compact, no corners& & \cite{Milnor}  &\\
\hline
$C^\8$, non compact, no corners&\cite{Kriegel u. Michor} & & \cite{Gloeckner2}\\
\hline
$C^\8$, non compact, with corners& & \cite{Michor}&\\
\hline
orbifold, compact & \cite{Borze}& &\cite{BobDis}\\
\hline
orbifold, non compact & & &\cite{BobDis}\\
\hline
$C^\o_\R$, compact, no corners&\cite{Kriegel u. MichorII}  & \cite{Bob2}, \cite{Leslie2}&\\
\hline
\end{tabular}
\\
(We mention that in \cite[Remark 5.22]{BobDis} it was stated, that the proof of \cite{BobDis} circumvents some problems which remained in \cite{Borze}. Moreover in \cite[p.1]{Kriegel u. MichorII} it was stated that the proof of \cite{Leslie} has a gap.)
 
The aim of our paper is to turn the diffeomorphism group $\Diff(M)$ of a finite-dimensional compact real analytic manifold $M$ with \textbf{corners} into a smooth Milnor Lie group. This generalises \cite{Kriegel u. MichorII} and \cite{Bob2}. But as \cite{Kriegel u. MichorII} and \cite{Bob2} use the global approach while we use the local approach, our paper gives also an alternative construction for \cite{Kriegel u. MichorII} and \cite{Bob2}. In the following we want to describe our strategy in details.

Given a manifold with boundary one can use the double of the manifold to embed it into a manifold without boundary (see, for example, \cite[Example 9.32]{Lee}). But this does not work in the case of a manifold with corners, because the boundary of a manifold with corners is not a manifold. If one works with a smooth manifold with corners one can use a partition of unity to construct a ``strictly inner vector field'' (\cite[p. 21]{Michor}). With the help of this vector field one obtains the analogous result (see \cite[p. 21]{Michor} and \cite[Proposition 3.1]{Douady}). Obviously this approach does not work if one considers a real analytic manifold with corners. For technical reasons we want to show the following theorem in Section \ref{SecEnveloping Manifold}:
\begin{theoreme}
Given a compact real analytic finite-dimensional manifold with corners  $M$ we find an enveloping manifold $\til{M}$ of $M$. If there exist two enveloping manifolds $\til{M}_1$ and $\til{M}_2$ of $M$, then we find an open neighbourhood $U_1$ of $M$ in $\til{M}_1$, an  open neighbourhood $U_2$ of $M$ in $\til{M}_2$ and a real analytic diffeomorphism $\ph \co U_1\ra U_2$ with $\ph|_M=\id_M$.
\end{theoreme}
In this context an enveloping manifold $\til{M}$ of $M$ is a real analytic manifold without boundary that contains $M$ as a submanifold with corners. In \cite[Proposition 1]{Bruhat} Bruhat and Withney show that given a  real analytic paracompact manifold $M$ (without corners) there exists a complex analytic manifold $M_\C$ that contains $M$ as a real submanifold. The manifold $M_\C$ is called complexification of $M$. We can transfer their proof without difficulties to show our Theorem \ref{EnvelopTheo}. The proof of our Theorem \ref{EnvelopTheo} that is completely analogous to the one of Bruhat and Withney can be found in Appendix \ref{ProofEnv}. In addition, we show in Section \ref{SecEnveloping Manifold} some technical properties of real analytic mappings concerning extensions to enveloping manifolds. The proofs are analogous to the case of extensions of real analytic mappings to complexifications for example as in \cite[Chapter 2]{Bob}.

The main result of our paper is the following:
\begin{theoreme}\label{MainAAA}
Let $M$ be a finite-dimensional compact real analytic manifold with \textbf{corners} such that there exists a boundary respecting Riemannian metric on a real analytic enveloping manifold $\til{M}$. Then there exists a unique smooth Lie group structure on the group of real analytic diffeomorphisms $\Diff^\o(M)$ modelled over $\gg^\o_\str(TM)$ such that for one (and hence each) boundary respecting Riemannian metric on $\til{M}$ the map $\eta \ms \gps_\eta$ is a diffeomorphism from an open $0$-neighbourhood in $\gg^\o_\str(TM)$ onto an open identity neighbourhood in $\Diff^\o(M)$.
\end{theoreme}
In this context, a Riemannian metric on the enveloping manifold $\til{M}$ of a manifold with corners $M$ is called boundary respecting, if the strata $\partial^jM$ are totally geodesic submanifolds.
A map defined on an open subset of $\gg^\o(TM)$ is smooth in the convenient setting  if and only if it is smooth in the Keller's $C^\8_c$ theory (\cite[p.1]{Bob2}). But a map on $\gg^\o(TM)$ that is real analytic in the convenient sense need not be real analytic in the conventional sense as in \cite[p. 1028]{Milnor} (see also \cite[p. 2]{Bob2}).
In \cite{Kriegel u. MichorII}, Kriegl and Michor constructed a real analytic Lie group structure on $\Diff^\o(M)$ for a compact manifold $M$ without corners. But we emphasise that their Lie group structure is only real analytic in the convenient setting 
(cf. \cite[Proposition 1.9]{Bob2}). Because a real analytic structure induces a smooth structure, the Lie group structure of \cite{Kriegel u. MichorII} induces a structure of a smooth Milnor Lie group on $\Diff^\o(M)$ as mentioned in \cite[Proposoition 1.9]{Bob2}.

One might expect that there also exists a real analytic Lie group structure in the conventional sense   on $\Diff^\o(M)$. But Dahmen and Schmeding showed in \cite{Bob2} that there exists no real analytic structure on $\Diff^\o(\mathbb{S}^1)$ in the conventional sense of Milnor. Therefore we can not expect that there exists a real analytic structure in the conventional sense on $\Diff^\o(M)$ for a compact real analytic manifold $M$ with corners. Dahmen and Schmeding (\cite[Proposition 1.9]{Bob2}) respectively Kriegl and Michor (\cite{Kriegel u. MichorII}) used the global approach to turn $\Diff^\o(M)$ (with $M$ compact and without boundary) into a smooth respectively real analytic Lie group. We instead want to use the local approach for our Theorem \ref{MainAAA} ($M$ has corners). The local approach  was developed by Gl{\"o}ckner in \cite{Gloeckner2}, and we follow the line of thought of \cite{Gloeckner2}. But Gl{\"o}ckner considered smooth diffeomorphisms on a manifold without corners. Hence one obvious obstacle is that we can not use bump functions, because we work in the real analytic setting. Moreover because our manifold $M$ has corners, we will have to model our structure on the space of stratified vector fields as in \cite{Michor}: In \cite{Michor} Michor turned the group of smooth diffeomorphisms of a non-compact manifold with corners into a smooth Lie group. Michor  worked with the global approach and as mentioned above this leads to a very different construction.

We also mention \cite{Leslie2}. In this paper Leslie used the global approach to turn the group of real analytic diffeomorphisms of a compact real analytic manifold without corners into a smooth Lie group. But as pointed out in \cite[p.1]{Kriegel u. MichorII}, his proof has a gap.

That we use the local approach (\cite{Gloeckner2} and \cite{BobDis}), is resembled in the structure of the paper: In Section \ref{SecLocal manifoldstructure} we construct a manifold  on a subset $\mathcal{U}$ of $\Diff(M)$ that contains the identity $\id_M$. The next step is to show the smoothness of the group operations. To this point we elaborate some important preparatory results in Section \ref{SecPreparation for results of smoothness}. In Section \ref{SecSmoothness of composition} we show the smoothness of the multiplication on $\mathcal{U}$ and in Section \ref{SecSmoothnes of the inversion} the smoothness of the inversion.
The smoothness of the conjugation is proved in Section \ref{SecExistence and uniqueness of the Lie group structure}. Our proof of the smoothness of the conjugation map follows closely the ideas of \cite[Section 5]{Gloeckner2}: First we show that the Lie group structure on $\Diff^\o(M)$ is independent of the choice of the Riemannian metric (see \cite{Gloeckner2}). With help of this result, we then can show the smoothness of the conjugation map as in \cite[Section 5]{Gloeckner2} (see Lemma \ref{ConjugationGlatt}).

\section{Enveloping manifold}\label{SecEnveloping Manifold}

At fist we want to show some basic facts about real analytic maps on manifolds with corners and enveloping manifolds. The primary aim of this section is to show that a real analytic manifold with corners can be embedded into a real analytic manifold without corners. As mentioned in the introduction we can not use a construction like the double of a manifold, because the boundary of a manifold with corners in not a manifold. Moreover we can not use the construction from the smooth case (\cite[Proposition 3.1]{Douady} or \cite[p. 21]{Michor}), because of the lack of real analytic bump functions. Instead we use the proof of the existence and uniqueness of complexifications of real analytic manifolds (see \cite[Proposition 1]{Bruhat} respectively \cite[Chapter 2 and Chapter 3]{Bob}). With the help of our Lemma \ref{GottVerdammt}, our proof of the existence of enveloping manifolds (Theorem \ref{EnvelopTheo}) is completely analogous to  \cite[Proposition 1]{Bruhat} (see Appendix \ref{ProofEnv}). For technical reasons we work in this section with manifolds that are modelled over a quadrant $[0,\8[^m$. Of course this definition of a manifold with corners is equivalent to the one where manifolds with corners are modelled over sets of the form $[0,\8[^k \ti \R^m_{m-k}$ with $k \leq m$.

\begin{remark}
We recall some common definitions and basic facts: 
\begin{compactenum}
\- Let $U\subs \R^m$ be open and $f\co U\ra \R^n$ be a map. The map $f$ is called real analytic if we can find an open neighbourhood $V\subs \C^m$ of $U$ and a complex analytic function $f^\ast \co V\ra \C^n$ with $f^\ast|_{U}= f$.
\- Given a real analytic map $f\co \R^m \sups U\ra \R^n$ and $x\in U$ we find a $0$-neighbourhood $V\subs \R^m$ such that $x+V\subs U$ and for all $v\in V$ we get $f(x+v)= \sum_{k=0}^\8 \fr{\d_x^kf(v)}{k!}$. In this context $\d_x^kf$ is the $k$-th Gateaux differential of $f$ in $x$.
\- If $U$ is an open connected subset of $\R^m$, $f\co U\ra \R^n$ is a real analytic map and 
$x\in U$ with $\d^k_xf =0$ for all $k\in \N_0$,
 then $f=0$.
\item Let $U$ be an open subset of $[0,\8[^m$. We call a map  $f\co U \ra \R^n$ real analytic if every $x\in U$ has a neighbourhood $\til{U}\subs \R^m$ such that there exists an real analytic map $\til{f}\co \til{U}\ra  \R^n$ with $\til{f}|_U=f$.
\- A corner-atlas of an Hausdorff space $M$ is a set of homeomorphism $\ph \co U_\ph\ra V_\ph$ between open subsets $U$ of $M$ and $V$ of $[0,\8[^m$ such that the changes of charts $\ps\ci \ph\inv\co \ph\inv(U_\ps \cap U_\ph) \ra V_\ps$ are real analytic maps. The space $M$ together with a maximal corner-atlas is called real analytic manifold with corners.
\end{compactenum}
\end{remark}

\begin{lemma}\label{KonvexSchnitt}
If $C\subs \R^m$ is convex, $U\subs \R^m$ is open $\mathring{C}\neq \emptyset$ and $C\cap U \neq \emptyset$ then $\mathring{C}\cap U \neq \emptyset$.
\end{lemma}
\begin{proof}
Let $z\in C\cap U$. If $z \in \mathring{C}$ we are done. Hence we can assume $z\in \partial C$. Because $C$ is convex and $\mathring{C}\neq \emptyset$ we can calculate
\begin{align*}
\partial C = \ol{C}\setminus \mathring{C} = \ol{\mathring{C}} \setminus \mathring{C} = \partial \mathring{C}.
\end{align*}
Hence $z \in \partial \mathring{C}$. Therefore every $z$-neighbourhood intersects $\mathring{C}$. Thus $U\cap \mathring{C} \neq \emptyset$.
\end{proof}

\begin{lemma}
Let $U$ be an open subset of  $[0,\8[^m$ and $f\co U\ra \R^n$ be a real analytic map. There exists an open neighbourhood $\til{U}\subs \R^m$ of $U$ and a real analytic map $\til{f}\co \til{U}\ra \R^n$ with $\til{f}|_U= f$.
\end{lemma}
\begin{proof}
Given $x\in U$ we find $\ep_x>0$ and a real analytic map $\til{f}_x \co \til{U}_x \ra \R^n$ defined on $\til{U}_x:= B^\8_{\ep_x}(x)$ such that $\til{f}_x|_U = f$.  Let $x,y\in U$ with $\til{U}_x \cap \til{U}_y \neq \emptyset$. We choose $z\in \til{U}_x \cap \til{U}_y$. Now we define $\til{z}$ by $\til{z}_i:=z_i$ if $z_i \geq 0$ and $\til{z}_i:= -z_i$ if $z_i <0$. Obviously $\til{z}\in [0,\8[^m$. We show $\til{z}\in \til{U}_x\cap \til{U}_y$. If $z_i \geq 0$ we have $|\til{z}_i-x_i|<\ep_x$. Now let $z_i<0$. We have $x_i\geq 0$, $x_i-z_i = |x_i-z_i|<\ep_x$. We calculate $|x_i-\til{z}_i| = |x_i+z_i| \leq x_i-z_i <\ep_x$. Hence $\til{z}\in \til{U}_x$. In the analogous way one shows $\til{z}\in \til{U}_y$. We conclude $[0,\8[^m \cap \til{U}_x \cap \til{U}_y \neq \emptyset$. Using Lemma \ref{KonvexSchnitt} we conclude $C:=]0,\8[^m\cap \til{U}_x \cap \til{U}_y \neq \emptyset$. The set $C$ is connected, because it is convex and we have $\til{f}_x|_C = f|_C = \til{f}_y|_C$. Hence we get $\til{f}_x|_{U_x\cap U_y} =  \til{f}_y|_{U_x\cap U_y}$. Now we define $\til{U}:= \bigcup_{x\in U}U_x$ and $\til{f} \co U \ra \R^n, z\ms f_x(z)$ if $z\in U_x$. The map $\til{f}$ is well-defined according to the upper consideration. Moreover $\til{f}$ is real analytic and $\til{f}|_U=f$.
\end{proof}

\begin{convention}
Given a manifold $M$ we write $\mathcal{A}_x^M$ for the set of charts around a point $x \in M$. If $N$ is a further  manifold, $f\co M \ra N$ a map and $\ps \in \mathcal{A}^N_{f(x)}$, we write $f^{\ph,\ps}$ for the local representative of $f$ in the charts $\ph$ and $\ps$. 
\end{convention}

In the following lemma we recall the well known Identity Theorem for real analytic maps.
\begin{lemma}
Let $M$ and $N$ be a real analytic manifolds (without corners) and $g_1,g_2 \co M\ra N$ be real analytic maps that coincide on an open non-empty subset $V\subs M$. If $M$ is connected, then $g_1=g_2$.
\end{lemma}
\begin{proof}
We  define 
\begin{align*}
W:= \set{x\in M: g_1(x)=g_2(x), ~ (\forall \ph \in \mathcal{A}_x^M,~ \ps\in \mathcal{A}_{g_1(x)}^N, ~  k\in \N_0) ~ \d_{\ph(x)}g_1^{\ph,\ps} = \d_{\ph(x)}g_2^{\ph,\ps}}.
\end{align*}
Because $V\neq \emptyset$ we get $W\neq \emptyset$. Let $(x_n)_{n\in \N}$ be a sequence that converges against $x\in M$ such that $x_n \in W$. We have $g_1(x_n)\ra g_1(x)$, $g_2(x_n)\ra g_2(x)$ and $g_1(x_n)=g_2(x_n)$ hence $g_1(x)=g_2(x)$. Let $\ph\co U_\ph \ra V_\ph$ be a chart around $x$ and $\ps \co U_\ps \ra V_\ps$ a chart around $g_1(x)=g_2(x)$. Without lose of generality we assume $x_n \in U_\ph$ and $g_1(x_n)=g_2(x_n)\in U_\ps$ for all $n$.  We get $\d_{\ph(x_n)}^kg_1^{\ph,\ps} \ra \d_{\ph(x)}^kg_1^{\ph,\ps}$, $\d_{\ph(x_n)}^kg_2^{\ph,\ps} \ra \d_{\ph(x)}^kg_2^{\ph,\ps}$ and $\d_{\ph(x_n)}^kg_1^{\ph,\ps} = \d_{\ph(x_n)}^kg_2^{\ph,\ps}$ for all $n$. Hence $\d_{\ph(x)}^kg_1^{\ph,\ps} = \d_{\ph(x)}^kg_2^{\ph,\ps}$. It is left to show that $W$ is open.
Let $x\in W$, $\ph \in \mathcal{A}_x^M$, $\ps\in \mathcal{A}_{g_1(x)}^N$.
Because $g_1$ and $g_2$ are real analytic we find an open $0$-neighbourhood $V\subs \R^m$ such that 
\begin{align*}
g_1^{\ph,\ps}(\ph(x)+z) = \sum_{k=0}^\8 \fr{\d_{\ph(x)}^kg_1(z)}{k!} = \sum_{k=0}^\8 \fr{\d_{\ph(x)}^kg_2(z)}{k!} = g_2^{\ph,\ps}(\ph(x)+z)
\end{align*}
for all $z\in V$. Hence there exists a $x$-neighbourhood on which $g_1$ and $g_2$ coincide.
\end{proof}

\begin{definition}
Let $M$ be a real analytic manifold with corners. A real analytic manifold without corners $\til{M}$ is called {\it enveloping manifold} for $M$ if $M\subs \til{M}$, and for every $x\in M$ there exists a chart $\til{\ph}\co \til{U} \ra \til{V}$ of $\til{M}$ around $x$, such that $\til{\ph}(\til{U}\cap M) = \til{V}\cap [0,\8[^m$ and $\til{\ph}|_{\til{U}\cap M}^{\til{U}\cap [0,\8[^m}$ is a chart of $M$. In other words: $M$ is an equidimensional sub manifold of $\til{M}$ with corners such that its sub manifold structure coincides with its original manifold structure. The chart $\til{\ph}$ is called {\it enveloping chart} of $M$.
\end{definition}

\begin{lemma}\label{GlobId}
Let $M\neq \emptyset$ be a real analytic manifold with corners and $\til{M}$ an enveloping manifold of $M$. Moreover let $N$ be a real analytic manifold without corners and $g_1,g_2\co \til{M}\ra N$ be real analytic maps.  If $g_1|_M=g_2|_M$ we find an open neighbourhood $V\subs \til{M}$ of $M$ such that $g_1|_V=g_2|_V$. 
\end{lemma}
\begin{proof}
Let $x\in M$ and $\ph \in \mathcal{A}_x^{\til{M}}$ with $\ph(M\cap U_\ph) = V_\ph\cap [0,\8[^m$ such that $\ph|_{M\cap U_\ph}^{V_\ph \cap [0,\8[^m}$ is a chart of $M$. Let $\ps$ be a chart of $N$ around $g_1(x) = g_2(x)$. Without lose of generality we assume $g_1(U_\ph),g_2(U_\ph)\subs U_\ps$ and we assume $V_\ph$ to be connected. We get $g_1^{\ph,\ps}|_{V_\ph \cap [0,\8[^m} = g_2^{\ph,\ps}|_{V_\ph \cap [0,\8[^m}$. Because of Lemma \ref{KonvexSchnitt} we find an open $\ph(x)$-neighbourhood in $V_\ph$ on which $g_1^{\ph,\ps}$ and $g_2^{\ph,\ps}$ coincide. Hence $g_1$ and $g_2$ coincide on an open neighbourhood of $x$.
\end{proof}

The following Lemma comes from \cite[Lemma 2.1 (a)]{Bob}.
\begin{lemma}\label{Bob}
Let $X$ be a regular topological Hausdorff  space, $K\subs X$ be a compact subset and $(U_i)_{i\in I}$ be an open cover of $K$. Then there exists an open cover $(V_j)_{j\in J}$ of $K$ such that given $j_1,j_2\in J$ with $V_{j_1}\cap V_{j_2}\neq \emptyset$ we find $i\in I$ with $V_{j_1}\cup V_{j_2}\subs U_i$.
\end{lemma}

In the following Lemma we show an existence result for extensions of real analytic maps on real analytic manifolds with corners. The proof follows the idea of \cite[Lemma 2.2 (a)]{Bob}, where Dahmen, Gl{\"o}ckner and Schmeding showed an analogous result for extensions to complexifications.
\begin{lemma}\label{AusdehnenGlobal}
Let $M$ and $N$ be real analytic manifolds with corners, $M$ be compact and $\til{M}$ respectively $\til{N}$ be enveloping  manifold of $M$ respectively $N$. If $f\co M\ra N$ is a real analytic map, we find an open neighbourhood $U\subs \til{M}$ of $M$ and a real analytic map $g\co U\ra \til{N}$ with $g|_M=f$.
\end{lemma}
\begin{proof}
Given $x\in M$ let $\ph_1 \co U_1 \ra V_1$ be an enveloping chart of $M$ around $x$ 
and  $\ph_2\co U_2\ra V_2$ be an enveloping chart of $N$ around $f(x)$ with $f(U_1\cap M)\subs U_2\cap N$. Especially $\ph_1|_{U_1\cap M}^{V_1\cap [0,\8[^m}$ and $\ph_2|_{U_2\cap M}^{V_2\cap [0,\8[^m}$ are charts of $M$ and $N$ respectively. We find a real analytic map $\ps\co V_1\cap [0,\8[^m \ra V_2\cap [0,\8[^m$ such that the diagram
\begin{align*}
\begin{xy}\xymatrixcolsep{5pc}
\xymatrix{
U_1\cap M \ar[r]^-{f} \ar[d]^{\ph_1}& U_2\cap N\ar[d]^{\ph_2}\\
V_1\cap [0,\8[^m \ar[r]^{\ps} & V_2\cap [0,\8[^m
}
\end{xy}
\end{align*}
commutes. Let $\ps_x\co V_x \ra \R^n$ be a real analytic map defined on an open neighbourhood $V_x$ of $V_1\cap [0,\8[^m$, such that $\ps_x|_{V_1\cap [0,\8[^m}=\ps$. Without lose of generality we can assume $V_x\subs V_1$ and $\ps_x(V_x)\subs V_2$, because $\ps_x(\ph_1(x))= \ph_2(f(x)) \in V_2$. 
Now we define the open 
set $U_x:= \ph_1\inv(V_x)$ and the real analytic map $g_x\co U_x \ra U_2\subs \til{N}$ by $g_x:=\ph_2\inv \ci \ps_x \ci \ph_1$. We have $g_x|_{U_x\cap M} = f|_{U_x\cap M}$.
With \cite[Lemma 2.1 (a)]{Bob} respectively Lemma \ref{Bob} we find an open cover $(W_j)_{j\in J}$ of $M$ such that given $j_1,j_2\in J$ with $W_{j_1}\cap W_{j_2}\neq \emptyset$ we find $x\in M$ with $W_{j_1}\cup W_{j_2}\subs U_{x}$. Every point $x\in M$ is contained in a set $W_i$. We can substitute $W_i$ by the connected component of $x$ in $W_i$ and hence we can assume all $W_i$ to be connected and intersecting $M$. 
For $i\in I$ we find $x_i\in M$ with $W_i \subs U_{x_i}$. Now let $g_i\co W_i \ra \til{N}$ be given by $g_i:= g_{x_i}|_{W_i}$.
If $i,k\in I$ with $W_i\cap W_k \neq \emptyset$ we find $x\in M$ with $W_i\cup W_k \subs U_x$. We have $g_i|_{W_i\cap M}= f|_{W_i\cap M} = g_x|_{W_i\cap M}$. Because $W_i$ is connected, we get with Lemma \ref{GlobId} that $g_i= g_x|_{W_i}$. In the analogous way we get $g_k= g_x|_{W_k}$. Hence $g_i|_{W_i \cap W_k} = g_k|_{W_i \cap W_k}$. Now we define the open set $U:=\bigcup_{i\in J}U_i$ and the real analytic map $g\co U\ra \til{N}$ by $g|_{W_i}:=g_i$. According to the upper consideration $g$ is well-defined and real analytic.
\end{proof}

In the following Lemma we show that a real analytic diffeomorphism of manifolds with corners has a diffeomorphic real analytic extension to open subsets of enveloping manifolds. Our proof is  analogous to \cite[Lemma 2.2 (e)]{Bob}.
\begin{lemma}\label{Inverse}
Let $f\co M\ra N$ be a real analytic diffeomorphism between real analytic manifolds with corners. Moreover let $\til{M}$ and $\til{N}$ be enveloping manifolds of $M$ and $N$ respectively and $U\subs \til{M}$ be an open neighbourhood of $M$. Furthermore let $V\subs \til{N}$ be an open neighbourhood of $N$, $\til{f}\co U \ra \til{N}$ be an extension of $f$ and $\til{g}\co V\co V\ra \til{M}$ be an extension of $g:=f\inv$. We find a neighbourhood $W\subs \til{M}$ of $M$ such that $\til{f}|_W$ is an diffeomorphism onto its open image with inverse $\til{g}|_{\til{f}(W)}$.
\end{lemma}
\begin{proof}
Let $X$ be the union of all connected components of $\til{f}\inv(V)$ that intersect $M$. Thus $\til{g}\ci \til{f}|_X \co X \ra \til{M}$ is real analytic and $\til{g}\ci \til{f}|_{X\cap M}= \id_{X\cap M}$. Hence $\til{g}\ci \til{f}|_X= \id_X$.

Let $Y$ be the union of all connected components of $\til{g}\inv(X)$ that intersect $N$. As above $\til{f}|_W \ci \til{g}|_Y \co Y\ra \til{M}$ is real analytic and $\til{f}|_X \ci \til{g}|_Y = \id_Y$. Now we show $Y=\til{f}(X)$. The inclusion ``$\subs$'' follows from $\til{f}|_X \ci \til{g}|_Y = \id_Y$. It is left to show $\til{f}(X)\subs Y$. With $\til{g}\ci \til{f}|_X= \id_X$ we get $\til{f}(X)\subs \til{g}\inv(X)$. If $C$ is a connected component of $\til{f}\inv(V)$ that intersect $M$, then $\til{f}(C)$ is a connected subset of $\til{g}\inv$ and intersects $N$. Hence $\til{f}(C)\subs Y$.
\end{proof}

\begin{lemma}\label{GottVerdammt}
Let $U\subs [0,\8[^m$ be open such that the closure of $U$ in $[0,\8[^m$ coincides with the closure of $U$ in $\R^m$. Given an open neighbourhood $O$ of $\ol{U}$ in $\R^m$ we find an open neighbourhood $\til{U}$ of $U$ in $\R^m$ such that $\ol{\til{U}}\subs O$, $\til{U}\cap [0,\8[ = U$ and $\ol{\til{U}}\cap [0,\8[^m = \ol{U}$.
\end{lemma}
\begin{proof}
Given $z\in \R^m$ we define $z_+\in [0,\8[^m$ by $(z_+)_i:= |z_i|$. Hence the map $\l \co \R^m \ra [0,\8[^m, z\ms z_+$ is continuous and $
\l\inv(U)$ is open in $\R^m$. Given $x\in U$ we find $\ep_x>0$ with $B_{\ep_x}(x)\cap [0,\8[^m \subs U$ and $B_{\ep_x}(x)\subs \l\inv(U)$. The set $\til{U}_1:= \bigcup_{x\in U} B_{\ep_x}(x)$ is open in $\R^m$. We also have $\til{U}_1\cap [0,\8[^m = U$, because $U\subs \til{U}_1$ and $\til{U}_1\cap [0,\8[^m = \bigcup_{x\in U} (B_{\ep_x}(x)\cap [0,\8[^m) \subs U$. Hence $\ol{U} = \ol{\til{U}_1 \cap [0,\8[^m} \subs \ol{\til{U}_1}\cap [0,\8[^m$. Now let $z\in \ol{\til{U}_1}\cap [0,\8[^m$. We find a sequence $(z_n)_{n\in \N}$ in $\til{U}_1$ with $\lim_{n\ra \8}z_n = z$. Hence $z= \l(z) = \lim_{n\ra \8}\l(z_n)$. But with $\til{U}\subs \l\inv(U)$ we get $\l(z_n)\in U$. Thus $z\in \ol{U}$. We conclude $\ol{\til{U}_1} \cap [0,\8[ = \ol{U}$. Now let $W$ be a neighbourhood of $\ol{U}$ in $\R^m$ with $\ol{U}\subs W \subs \ol{W}\subs O$. We define $\til{U}:= \til{U}_1\cap W$ and get $\ol{\til{U}} \subs \ol{W} \subs O$ and $\til{U}\cap [0,\8[^m = U\cap W =U$. Hence $\ol{U} = \ol{\til{U}\cap [0,\8[^m}\subs \ol{\til{U}}  \cap [0,\8[^m$. Moreover
\begin{align*}
\ol{\til{U}} \cap [0,\8[^m = \ol{\til{U}\cap W} \cap [0,\8[^m \subs \ol{\til{U}} \cap \ol{W} \cap [0,\8[^m = \ol{W} \cap \ol{U} =\ol{U}.
\end{align*}
Hence $\ol{\til{U}}\cap [0,\8[^m = \ol{U}$.
\end{proof}

With the help of Lemma \ref{GottVerdammt} it is possible to transfer the proof of \cite[Proposition 1]{Bruhat} respectively \cite[Proposition 3.1]{Bob} (existence of complexifications of real analytic manifolds) to our situation. Using Lemma \ref{GottVerdammt} our proof is complete analogous to the one of \cite[Proposition 1]{Bruhat} respectively \cite[Proposition 3.1]{Bob}.
\begin{theorem}\label{EnvelopTheo}
Given a compact real analytic finite-dimensional manifold with corners  $M$ we find a enveloping manifold $\til{M}$ of $M$. If there exist two enveloping manifolds $\til{M}_1$ and $\til{M}_2$ of $M$, then we find a neighbourhood $U_1$ of $M$ in $\til{M}_1$, a neighbourhood $U_2$ of $M$ in $\til{M}_2$ and a real analytic diffeomorphism $\ph \co U_1\ra U_2$ with $\ph|_M=\id_M$.
\end{theorem}
\begin{proof}
See Appendix \ref{ProofEnv}.
\end{proof}

\section{Local manifold structure}\label{SecLocal manifoldstructure}
As mentioned in the introduction we use the ``local approach'' developed in \cite{Gloeckner2} and transfer its line of thought to the case of a compact real analytic manifold $M$ with corners. In this section we construct an open subset $\mathbb{V}$ of $\gg^\o_\str(TM)$ that is small enough such that $\mathcal{U}:= \set{\exp\ci \eta: \eta \in \mathcal{V}}$ is a subset of $\Diff^\o(M)$. As in \cite{Gloeckner2} we control simultaneously the uniform norm of the vector fields of $\mathcal{V}$ on a compact set and the norms of the first derivative. For the rest of the paper  the manifolds with corners are modelled over spaces of the form $\R^m_k:= [0,\8[^k \ti \R^{m-k}$ instead of quadrants $[0,\8[^m$ (see Section \ref{SecEnveloping Manifold}). Of course both definitions are equivalent. As in \cite{Michor} we model the group of diffeomorphisms on stratified vector fields ($\gg_\str(TM)$). In Appendix \ref{SecBasic definitions and results for manifolds with corners} we fix some notation and recall basic definitions and facts concerning manifolds with corners that are used in this paper.

\subsection{Conventions and notations}

For the following definition we referee to \cite[Chapter 4; Definition 12 and Proposition 13]{ONeil}.
\begin{definition}
Let $M$ be a Riemannian manifold without boundary and $N \subs M$ be a Riemannian submanifold without boundary. We call $N$ {\it totally geodetic} if all geodetics of $N$ are also geodetics of $M$.
\end{definition}

\begin{convention}
Let $M$ be a $m$-dimensional compact real analytic manifold with corners and $\til{M}$ be an enveloping manifold of $M$. Moreover let $M_\C$ be a complexification of $\til{M}$. \textbf{ We assume that there exists a Riemannian metric $g$ on $\til{M}$ such that 
the submanifolds $\partial^jM$ are totally geodetic for all $j\in \set{1,\dots,m}$}. We call such a metric \it{boundary respecting}. In this context $\til{\Omega} \subs T\til{M}$ is a neighbourhood of the $0$-section.
\end{convention}

\begin{lemma}\label{Existenz von Karten}
There exist finite many enveloping-charts $\til{\ph}_i \co \til{U}_{i,5}\ra B_5(0)$ with $i=1,\dots,n$ and induced $M$-charts $\ph_i\co U_{i,5} \ra B_5^{k_i}(0)$ such that $M\subs \bigcup_{i=1}^n \ph_i\inv(B_{1}^{k_i}(0))$. There exist an open subset ${U}^\ast_i$ of $M_\C$, an open subset ${V}^\ast_i$ of $\C^m$ and a complex analytic diffeomorphism ${\ph}^\ast_i \co {U}^\ast_i \ra {V}^\ast_i$ such that $\til{U}_{i,5} \subs {U}^\ast_i$, $B_{i,5}(0) \subs {V}^\ast_i$ and ${\ph}^\ast_i$ is an extension of $\til{\ph}_i$.\end{lemma}

\begin{convention}
\begin{compactitem}
\item On $\R^m$ respectively on $\C^m$ we use the Euclidean norm.
\item If $f$ is a differentiable  map on an open subset $U$ of $\R^m_k$ respectively $\R^m$ respectively $\C^m$ then $f'$ always means the first derivative as a map from $U$ to $\mathcal{L}(\R^m)$ respectively $\mathcal{L}(\C^m)$. We equip $\mathcal{L}(\R^m)$ respectively $\mathcal{L}(\C^m)$ with the operator norm.
\item Let $\K \in \set{\R,\C}$. If $f\co X \ra \K^m$ is a map, we write $\|f\|_\8^U:= \sup\set{\|f(x)\|:x\in U}$ and $\|f\|_\8:=\|f\|_\8^X$ for the uniform norm.
\item We write $B_r(x)$ for balls with radius $r$ in $\R^m$ and $B_r(x)^\C$ for balls with radius $r$ in $\C^m$. Moreover we define $B^k_r(x):= B_r(x)\cap \R^m_k$.
\item We fix charts $\ph_i$, $\til{\ph}_i$ and ${\ph}^\ast_i$ like in Lemma \ref{Existenz von Karten}. Moreover we define
\begin{align*}
\til{U}_{i,r}:= \til{\ph}\inv(B_r(0)) \tx{and} {U}_{i,r}:= {\ph}\inv(B^{k_i}_r(0)).
\end{align*}
\item Let $g_i$ be the Riemannian metric on $B_5(0)$ that is induced by $g$ via $\til{\ph}_i$ and let $\exp_i \co \til{\Omega}_i  \ra B_5(0)$ be the exponential map on $B_5(0)$ that is induced by $g_i$.
\item If $\eta \in \gg^\o(TM)$ we define $\eta_{(i)}:= d\ph_i \ci \eta \ci \ph_i\inv \co B_5^{k_i}(0) \ra \R^m$. If $U \subs \til{M}$ is an open neighbourhood of $M$ in $\til{M}$ respectively in $M_\C$ and $\eta \in \gg^\o(TU)$ respectively $\eta \in \gg^\8_\C(TU)$ we define $\eta_{(i)}:= d\til{\ph}_i \ci \eta \ci \til{\ph}_i\inv \co \til{\ph}_i(U_{i,5}\cap U)\subs B_5(0) \ra \R^m$ respectively $\eta_{(i)}:= d{\ph}^\ast_i \ci \eta \ci {{\ph}^\ast_i}\inv \co {\ph}^\ast_i({U}^\ast_{i,5}\cap U) \subs {V}^\ast_{i,5} \ra \C^m$.
\item If $\eta \in \gg^\o(TM)$ we can use the identity theorem to obtain an extension that is a real analytic vector field $\til{\eta}$ of an neighbourhood of $M$ in $\til{M}$. 
Analogously we write ${\eta}^\ast$ for a complex analytic extension to  a vector fields on an open neighbourhood of $M$ in the  enveloping complex analytic manifold $M_\C:= (\til{M})_\C$.
\item If $f \in C^\o(B^{k}_R(0); \R^m)$ we write $\til{f}$ for an extension to an open neighbourhood of $B_R^k(0)$ in $\R^m$ and ${f}^\ast$ for an extension to an open neighbourhood of $B_R^k(0)$ in $\C^m$.
\item Given a compact connected set $K$ in a topological space $X$, we call a sequence of open connected relative compact subsets $(U_n)_{n\in \N}$ of $X$ a connected filtration of $K$ if ${U}_n \sups \ol{U}_{n+1} \sups K$ for all $n \in \N$ and $(U_n)_{n\in\N}$ is a neighbourhood basis of $K$ in $X$.
\item If $U$ is open in $\C^m$ and $k\in \set{0,1}$ we define $\Hol^k_b(U;\C^m):= \set{f\in \Hol(U;\C^m): \|f\|_\8,\|f'\|_\8<\8}$. With $\|\bl\|^1 \co \Hol^1_b(U;\C^m) \ra [0,\8[$, $f\ms \max(\|f\|_\8, \|f'\|_\8)$ respectively $\|\bl\|^0:= \|\bl\|_\8$ the space $\Hol^k_b(U;\C^m)$ becomes a Banach space.
\item If $\mathbb{V}$ is a vector bundle over a manifold $M$ and $K$ is a compact subset of $M$ then write $\gg(\mathbb{V}|K)$ for the space of germs of vector fields along $K$. If $K$ is a compact subset of $\C^m$ we write  $\mathcal{G}(\C;\C|K)$, for the space of germs of complex analytic functions along $K$.
\item If $U\subs \C^m$ is open we define $\Hol^k_b(U;\C^m)^\R:= \set{f\in \Hol^k_b(U;\C^m): f(U \cap \R^m) \subs \R^m}$. Analogously we define $\mathcal{G}(\C;\C|K)^\R$.
\end{compactitem}
\end{convention}

\begin{definition}
\begin{compactenum}
\item\label{strat} We call a section $\eta \in \Gamma^{\o}(TM)$ stratified if 
\begin{align*}
p \in \partial^j M \Rightarrow \eta(p) \in T_p\partial^jM
\end{align*}
for all $p \in M$ and write $\Gamma_\str^\o(TM)$ for the subspace of stratified sections.
\item A map $\eta \colon \ol{B}_r^k(0)\ra \R^m$ respectively $\eta \co {B}_r^k(0) \ra \R^m$ is called stratified if 
\begin{align*}
x_j = 0 \Rightarrow \eta(x)_j=0
\end{align*}
for all $x\in B_r^k(0)$ and $j=1,\dots,k$. With respect to the canonical identification this definition coincides with the one of (\ref{strat}). We write $C^\o(B_r^k(0); \R^m)_\str$ for the subspace of stratified real analytic maps.
\item Let $\K \in \set{\R,\C}$, $U\subs \K^m$ be open and $B^k_r(0) \subs U$. A $\K$-analytic map $f\co U \ra \K^m$ is called stratified along $B^k_r(0)$ if 
\begin{align*}
x_j = 0 \Rightarrow \eta(x)_j=0
\end{align*}
for all $x\in B_r^k(0)$ and $j=1,\dots,k$. We write $C^\o(U;\R^m)_\str$ respectively $\Hol(U;\C^m)_\str$ for the subspaces of stratified maps.
\item A germ $[f] \in \mathcal{G}(\C^m;C^m|B_r^k(0))$ is called stratified if one and hence all representatives are stratified along $B^k_r(0)$. We write $\mathcal{G}(\C^m;C^m|B_r^k(0))_\str$ for the subspaces of stratified germs.
\end{compactenum}
\end{definition}

\begin{remark}\label{SilvaUVR}
A section $\eta \in \gg^\o(TM)$ is stratified if and only if for all $i \in \set{1,\dots,n}$ there exists $R \in [1,5]$ such that $\eta_{(i)} \co B^{k_i}_R(0) \ra \R^m$ is stratified.
\end{remark}

\subsection{Topological considerations}
In the following we recall the definition of a Silva space for the convenience of the reader. For more details see \cite[p. 260]{Gloeckner4} and \cite{Leslie}.
\begin{definition}
A locally convex space is called a \it{Silva space} if it is the direct limit of Banach spaces in the category of locally convex spaces such that the transition maps are compact operators.
\end{definition}

The facts from the following Lemma \ref{SilvaGleichTop} about Silva spaces are a direct consequence of \cite[Proposition 4.5]{Gloeckner4}. See also \cite{Leslie}.
\begin{lemma}\label{SilvaGleichTop}
Let $(E,\ph_i)$ be a Silva space over an inductive system $(E_i,T_{i,j})$. Then the following holds:
\begin{compactenum}
\item The space $E$ is Hausdorff.
\item\label{SilvaGleichTopB} The topology on $E$ coincides with the inductive limit topology in the category of topological spaces.
\item Let $(U_i)_{i\in\N}$ be a sequence of open sets $U_i \subs E_i$, such that $(\ph_i(U_i))$ is an ascending sequence of subsets of $E$. Then $\bigcup_{n\in \N} U_n \subs E$ is open in $E$.
Moreover a map $f\co U \ra F$ into an other Hausdorff locally convex space $F$ is smooth if $f|_{U_n} \co U_n \ra F$ is smooth for all $n\in \N$.
\end{compactenum}
\end{lemma}

\begin{remark}
With Lemma \ref{SilvaGleichTop} (\ref{SilvaGleichTopB}) we conclude that a closed subspace $F$  of $E$ is a Silva space over the inductive system $(F_i,T_{i,j}|_{F_i})$ with $F_i:= \ph_i\inv(F)$.
\end{remark}

As in the proof of \cite[A.7]{Bob2}, we use in the following lemma the Cauchy integral formula to obtain an upper bound of the derivative of a complex analytic function.
\begin{lemma}
If $U$ and $V$ are open subsets of $\C^m$ such that $V$ is relative compact and $V\subs \ol{V}\subs U$, then the map $\Hol_b^0(U;\C^m) \ra \Hol_b^1(V;\C^m)$, $f \ms f|_V$ is continuous linear.  
\end{lemma}
\begin{proof}
Let $\mu \co \C^m \ti \C^m \ra \C^m$, $(x,y) \ms x+y$ be the addition on $\C^m$. Now $\ol{V} \ti\set{0} \subs \mu\inv(U)$. With the Lemma of Wallace we find $\ep>0$ with $\ol{V} + \ol{B}_\ep^{\C}(0) \subs U$. Hence for all $p \in V$ we have $\ol{B}_\ep(p) \subs U$. If $v\in \C^m$ with $\|v\|=1$ we can use the Cauchy integral formula and get
\begin{align*}
\|f'(p)(v)\|\leq \frac{2}{\ep} \cdot \sup_{q \in \ol{B}^\C_\ep(p)} \|f(q)\| \leq \frac{2}{\ep} \cdot \|f\|_\8.
\end{align*}
Hence $\|f'\|^V_\8 \leq \frac{2}{\ep}\cdot \|f\|_\8^U$.
\end{proof}

\begin{definition}
Let $K$  be a connected compact subset of $\C^m$ and $(U_n)_{n\in \N}$ a connected filtration of $K$. As in \cite[Appendix A]{Bob2} we give $\mathcal{G}(\C^m;\C^m|K)$ the direct limit topology induced by the inductive system $\Hol^0_b(U_n ; \C^m) \ra \Hol^0_b(U_{n+1} ; \C^m)$, $f \ms f|_{U_{n+1}}$ in the category of locally convex spaces. Because of the diagram 
\begin{align*}
\begin{xy}\xymatrixcolsep{0.2pc}
\xymatrix{
&\cdots \Hol_b^0(U_n;\C^m)\ar[dr]^-{\text{res}}\ar[rr]^-{\text{res}}&  &\Hol_b^0(U_{n+1};\C^m)  \cdots\\
\cdots \Hol_b^1(U_n;\C^m)\ar@{^{(}->}[ur]\ar[rr]^-{\text{res}}&  &\Hol_b^1(U_{n+1};\C^m)\ar@{^{(}->}[ur]\cdots &                    
}
\end{xy}
\end{align*}
$\mathcal{G}(\C^m;\C^m|K)$ is the direct limit of the inductive system $\Hol^0_b(U_n ; \C^m) \ra \Hol^0_b(U_{n+1} ; \C^m)$, $f \ms f|_{U_{n+1}}$ in the category  of locally convex spaces.
\end{definition}

In  \cite[A.10]{Bob2} it was shown that $\mathcal{G}(\C^m;\C^m|K)$ becomes a Silva space as the inductive limit $\varinjlim \Hol^0_b(U_n;\C^m)$. In the following we show that $\mathcal{G}(\C^m;\C^m|K)$ also becomes an Silva space as the inductive limit $\varinjlim \Hol^1_b(U_n;\C^m)$.
\begin{lemma}
Let $U$ and $V$ be open subsets of $\C^m$ such that $V$ is relative compact and $V \subs \ol{V} \subs U$. Then the restriction $T_1\co \Hol^1_b(U;\C^m) \ra \Hol^1_b(V;\C^m)$, $\eta \ms \eta|_V$ is a compact operator. Hence if $K \subs \C^m$ is  connected and compact with a connected filtration $(U_n)_{n\in \N}$, then the space  $\mathcal{G}(\C^m;\C^m|K)$ becomes a Silva space as the inductive limit $\varinjlim \Hol^1_b(U_n;\C^m)$. Moreover $\mathcal{G}(\C^m;\C^m|\ol{B}_{r}^{k}(0))_{\str}^\R$ becomes an Silva space as the inductive limit $\varinjlim \Hol^1_b(U_n;\C^m)_\str^\R$ (see Lemma \ref{SilvaUVR}).
\end{lemma}
\begin{proof}
We choose an open relative compact subset $W$ of $\C^m$ such that $\ol{V} \subs W \subs \ol{W} \subs U$. From \cite[Theorem 3.4]{Kriegel u. MichorII} and \cite[A. 10]{Bob2} we deduce that the map $T_0 \co \Hol^0_b(U;\C^m) \ra \Hol^0_b(W;\C^m)$, $\eta \ms \eta|_W$ is a compact operator.  Now we consider the following diagram:
\begin{align*}
\begin{xy}\xymatrixcolsep{5pc}
\xymatrix{
\Hol_b^1(U;\C^m) \ar[r]^{T_1} \ar@{_{(}->}[d]& \Hol_b^1(V;\C^m)\\
\Hol_b^0(U;\C^m) \ar[r]^{T_0}& \Hol_b^0(W;\C^m)\ar[u]_{\text{res}}
}
\end{xy}
\end{align*} 
where the second vertical arrow is the restriction to $V$. Because $T_0$ is compact and both vertical arrows are continuous linear, we get that $T_1$ is compact.
\end{proof}

\begin{definition}
\begin{compactenum}
\item We topologies the germs of vector fields around a compact set in the same way as  Dahmen and Schmeding respectively Krigel and Michor (see \cite{Kriegel u. MichorII} respectively \cite{Bob2}). Hence we give  $\gg^\o(T\til{M}|M)$ the topology of a  subspace of $\gg^\o(T\til{M}|M)_\C = \gg(TM_\C|K)$. With help of the bijection $\gg^\o(TM) \ra \gg^\o(T\til{M}|M)$, $\eta\ms [\til{\eta}]$ we turn $\gg^\o(TM)$ into a locally convex space. Therefore the closed subspace $\gg_\str^\o(TM)$ becomes a locally convex space. Given $R \in [1,5[$ we use \cite[Lemma A.16]{Bob2} and see that
\begin{align*}
\gg^\o_\str(TM) \ra \prod_{i=1}^n \mathcal{G}(\C;\C|\ol{B}^{k_i}_{R}(0))_\str^\R, ~ \eta \ms [{\eta}^\ast_{(i)}]
\end{align*}
is a linear topological embedding with closed image.
\item Given $r,\ep>0$ and $k\in \N_0$ we write 
\begin{align*}
\mathcal{B}_{r,\ep}^k:= \set{\eta\in \Gamma_\str^\o(TM): (\forall i \in \set{1,..,n}) ~ \|\eta_{(i)}\|_{\ol{B}^{k_i}_r(0)}^k <\ep}
\end{align*}
\end{compactenum}
\end{definition}

\begin{lemma}\label{SimultaneFortsetzung}
If $r \in ]0,5[$, $\ep>0$ and $f \co B_5^{k}(0) \ra \R^m$ is a real analytic map with $\|f\|^1_{\ol{B}_r^{k}(0)} <\ep$, then there exist a complex analytic map ${f}^\ast \co U \ra \C^m$ with $\|{f}^\ast\|_\8^1<\ep$ on an open subset $U \subs \C^m$ with $\ol{B}_r^{k}(0) \subs U$ such that ${f}^\ast|_{\ol{B}_r^{k}(0)} = f|_{\ol{B}_r^{k}(0)}$.
\end{lemma}
\begin{proof}
We define the real analytic map
\begin{align*}
\ph \co \ol{B}_r^{k}(0) \ra B_\ep^{\C^m}(0) \ti B_\ep^{\mathcal{L}(\C^m)}(0), ~x \ms (f(x), f'(x)),
\end{align*}
where we consider $\ol{B}_r^{k}(0)$ as a real analytic manifold with corners. We find a connected open neighbourhood $U \subs \C^m$ of $\ol{B}_{r}^{k}(0)$ and an extension ${\ph}^\ast \co U \ra B_\ep^{\C^m}(0) \ti B_\ep^{\mathcal{L}(\C^m)}(0)$ of $\ph$.
If $x \in B_r^{k_i}(0)$ and $v \in \R^m$, we get with \cite[Lemma 1.6.5]{GloecknerNeeb} that ${{\ph}^\ast_1}'(x)(v) = f'(x)(v)= {\ph}_2^\ast(x)(v)$. Using the linearity over $\C$ we conclude ${{\ph}_1^\ast}'(x)  = {\ph}_2^\ast(x)$ for $x\in B_{r}^{k_i}(0)$. Hence ${{\ph}^\ast_1}' = {\ph}^\ast_2$. Therefore ${\ph}^\ast_1$ is an extension of $f$ as needed.
\end{proof}

\begin{lemma}
Given $\ep>0$, $r\in [1,5[$ and $k\in \set{0,1}$, the set $\mathcal{B}_{r,\ep}^k$ is an open  $0$-neighbourhood in $\gg_\str^\o(TM)$.
\end{lemma}
\begin{proof}
Let $(U_n^i)_{n\in \N}$ be a connected filtration of $\ol{B}^{k_i}_r(0)$ in $\C^m$. Then 
\begin{align*}
\mathcal{U} := G_\ep^k(\C^m;\C^m|\ol{B}^{k_i}_r(0)) = \bigcup_{n\in \N} \left[\Hol_\ep^k(U_{n}^i;\C^m)\right] 
\end{align*}
is open in $\mathcal{G}(\C^m;\C^m|\ol{B}^{k_i}_r(0))$, because the right-hand side is an ascending union. Hence the set 
\begin{align*}
(\dagger) :=\set{\eta \in \gg^\o_\str(TM): (\forall i) ~ [{\eta}^\ast_{(i)}] \in \mathcal{U}}
\end{align*}
is open in $\gg^\o_\str(TM)$. Now we calculate with Lemma \ref{SimultaneFortsetzung}:
\begin{align*}
&(\dagger) =  \set{\eta \in \gg^\o_\str(TM): (\forall i) ~ (\exists n\in \N) ~  [({\eta}^\ast)_{(i)}]= [({\eta_{(i)})^\ast}] \in \left[\Hol^k_\ep(U_n^i;\C^m)\right]}\\
=& \set{\eta \in \gg^\o_\str(TM): (\forall i) ~ (\exists n\in \N)~ {\eta}_{(i)} \tx{has an extension} {\eta_{(i)}^\ast} \in \Hol^k_\ep(U_n^i;\C^m)}\\
=& \set{\eta \in \gg^\o_\str(TM): (\forall i) ~ \|{\eta}_{(i)}\|_{\ol{B}_r^{k_i}(0)}^k <\ep} = \mathcal{B}^k_{r,\ep}.
\end{align*}
\end{proof}

\subsection{A local chart}

\begin{remark}\label{OB}
Obviously we have
\begin{align*}
\exp_i(T\til{\ph}_i(v)) = \til{\ph}_i(\exp(v)) 
\end{align*}
for all $v\in T\til{\ph}_i\inv(\til{\Omega}_i)$ respectively
\begin{align*}
\exp_i(x,v) = \til{\ph}_i (\exp(T\til{\ph}_i\inv(x,v)))
\end{align*}
for $(x,v)\in \til{\Omega}_i$. Moreover we have
\begin{align*}
\exp_i(x,0)=x \text{ and } d_2\exp(x,0;\bl)=\id_{\R^m}
\end{align*}
for all $x\in B_5(0)$.
\end{remark}

Boiling down \cite[Theorem 2.3]{Gloeckner3} to our situation we get the following Lemma \ref{UKFunk} that is the analogous statement to \cite[Proposition 3.1]{Gloeckner2} in the analytic case:
\begin{lemma}\label{UKFunk}
Let $\K \in \set{\R,\C}$, $P\subs \K^n$ and $U\subs \K^m$ be open and $f\co P\ti U \ra \K^m$  be a $\K$-analytic map. Moreover let $(x_0,y_0)\in P\ti U$ and $d_2f(x_0,y_0;\bl)\in \GL(\K^m)$ 
There exists a $y_0$-neighbourhood $U'\subs U$ and a $x_0$-neighbourhood $P'\subs P$  such that:
\begin{compactitem}
\item For all $x\in P'$ the map $f(x,\bl)\co U' \ra \K^m$ has open image and is a  $\K$-analytic diffeomorphism onto its image;
\item The set $W:= \bigcup_{x\in P'} \set{x}\ti f(x,U')$ is open in $\K^n\ti \K^m$ and the map $P'\ti U' \ra W,~ (x,y)\ms (x,f(x,y))$ is a $\K$-analytic diffeomorphism with inverse function $W\ra B_{r_1}(x_0)\ti B_{r_2}(y_0), (x,z) \ms (x,f(x,\bl)\inv(z))$;
\item 
There exists $\delta >0$ such that for all $x\in P'$ we have $B_\delta(f(x,y_0)) \subs f(x,U')$ and $W':= \bigcup_{x\in P'} \set{x}\ti B_\delta(f(x,y_0)) \subs W$ is open.
\end{compactitem}
\end{lemma}

The following Lemma \ref{ExpDef} is the analogous statement to \cite[3.2]{Gloeckner2} in the real analytic case.
\begin{lemma}\label{ExpDef}
There exists $\ep_{\exp} >0$ such that:
\begin{compactenum}[(a)]
\item We have $\ol{B}_{4.5}(0)\ti \ol{B}_{\ep_{\exp}}(0)\subs \til{\Omega}_i \subs B_5(0)\ti \R^m$ for all $i\in \set{1,...,n}$.
\item For all $x\in \ol{B}_{4.5}(0)$ and  $i\in \set{1,...,n}$ the map $\exp_{i,x}:=\exp_i(x,\bl) \co B_{\ep_{\exp}}(0)\ra \R^m$ has open image and is a real analytic diffeomorphism onto its image. Moreover the map $B_{4.5}(0)\ti B_{\ep_{\exp}}(0) \ra B_{4.5}(0) \ti \R^m$, $(x,y)\ms \exp_i(x,y)$ has open image and is a real analytic diffeomorphism onto its image.
\end{compactenum}
\end{lemma}
\begin{proof}
\begin{compactenum}
\item[(a)/(b):] Let $i\in \set{1,...,n}$. Given $x\in B_5(0)$ we use Lemma \ref{UKFunk} to find $r_x>0$ and $\ep_x>0$ such that $\ol{B}_{r_x}(x)\ti \ol{B}_{\ep_x}(0) \subs \til{\Omega}_i$ and $\exp(y,\bl)\co B_{\ep_x}(0)\ra B_5(0)$ has open image and is a real analytic diffeomorphism onto its image for all $y \in B_{r_x}(x)$ and $B_{r_x}(x)\ti B_{\ep_x}(0) \ra B_{r_x}(x)\ti \R^m$ is a real analytic diffeomorphism onto its image. We find finite many $x_1,\cdots, x_k \in B_5(0)$ such that $\ol{B}_{4.5}(0) \subs \bigcup_{j=1}^kB_{r_{x_j}}(x_j)$ and set $\ep^i_{\exp}:=\min_{j}\ep_{x_j}>0$. Now we set $\ep_{\exp}:=\min_{i=1,\dots,n} \ep^i_{\exp}$.
Given $i\in \set{1,...,n}$ and $y \in \ol{B}_{4.5}(0)$ we find $j$ such that $y \in B_{r_{x_j}}(x_j)$. Hence $\set{y}\ti \ol{B}_{\ep_{\exp}}(0) \subs \til{\Omega}_i$ and $\exp_i(y,\bl)\co B_{\ep_{\exp}}(0)\ra B_5(0)$ is a real analytic diffeomorphism onto its open image. Moreover the map $B_{4.5}(0)\ti B_{\ep_{\exp}}(0) \ra B_{4.5}(0) \ti \R^m$, $(x,y)\ms \exp_i(x,y)$ is injective and a local diffeomorphism. Hence it has open image and is a real analytic diffeomorphism onto its image. Therefore we find $\ep_{\exp}$ as needed. 
\end{compactenum}
\end{proof}

\begin{remark}\label{Hilfreich1}
Using Remark \ref{OB} we make the following observation: For all $i\in \set{1,\dots,n}$, $x\in B_4(0)$ and $w \in B_{\ep_{\exp}}(0)$ we have $\exp(T\til{\ph}_i\inv(x,w)) \in U_{i,5}$ and $\exp(T\til{\ph}_i\inv(x,w)) = \til{\ph}_i\inv(\exp_i(x,w))$.
\end{remark}

\begin{definition}
\begin{compactenum}
\item Let $r\in [{1},4]$. If $\eta \in \mathcal{B}_{r, \ep_{\exp} }^0$, then $\im(\eta) \subs \til{\Omega}$, because $(x,\eta_{(i)}(x)) \in \til{\Omega}_i$ for $x\in \ol{B}_r^{k_i}(0)$. In this situation we define the real analytic map
\begin{align*}
\ps_\eta \co M \ra \til{M},~ p \ms \exp(\eta(p)).
\end{align*}
\item For $i \in \set{1,\dots,n}$, $U \subs B^{k_i}_4(0)$ open and $\eta\in C^\o(U;\R^m)_\str$ with $\|\eta\|^{U}_\8<\ep_{\exp}$ we define the real analytic map
\begin{align*}
\ps^i_\eta\co U \ra B_5(0),~ x\ms \exp_i(x,\eta(x)).
\end{align*}
\end{compactenum}
\end{definition}

\begin{lemma}\label{Hilfreich2}
For $r\in [{1},4]$ and $\eta \in \mathcal{B}_{r, \ep_{\exp} }^0$ we get $\ps_\eta(U_{i,r})\subs \til{U}_{i,5}$ and
\begin{align*}
\ps_\eta|_{U_{i,r}} = \til{\ph}_i\inv\ci \ps^i_{\eta_{(i)}} \ci \ph_i|_{U_{i,r}}.
\end{align*}
\end{lemma}
\begin{proof}
Given $p \in U_{i,r}$ we  use Remark \ref{Hilfreich1} and calculate
\begin{align*}
&\ps_\eta(p) = \ps_\eta(\ph_i\inv(\ph_i(p))) = \exp(T\ph_i\inv \ci T\ph_i \ci \eta \ci \ph_i\inv \ci \ph_i (p))) \\
=& \exp ( T\ph_i\inv (\ph_i(p),\eta_{(i)} (\ph_i(p))) = \ph_i\inv ( \exp_i(\eta_{(i)} (\ph_i(p)))) = \ph_i\inv\ci \ps^i_{\eta_{(i)}} \ci \ph_i (p)
\end{align*}
\end{proof}

\begin{remark}\label{idA}
If $A\in \mathcal{L}(\R^m)$, then $(\id,A)\co \R^m \ra \R^m\ti\R^m$ is linear with $\|(\id,A)\|_{op} \leq 1+\|A\|_{op}$.
\end{remark}

The following lemma is a stronger version of \cite[Lemma 3.7]{Gloeckner2} in the case of open sets with corners and with a variable radius an control of the norms.
\begin{lemma}\label{Nahe1}
Given $R\in ]0,5]$, $l\in ]0,R[$ and $r\in ]0,1[$, we find $\ep\in ]0,\ep_{\exp}]$ such that  for all $\eta\in C^\o(B^{k_i}_R(0);\R^m)_\str$ with $\|\eta\|^1_{l}<\ep$ and $i\in \set{1,\dots,n}$ the following assertions hold:
\begin{compactenum}
\item $\|{\ps^{i}_\eta}'(x)-\id_{\R^m}\|_{op} <r$ for all $x\in B_l^k(0)$;
\item $\|\ps^i_\eta(x)-x\| <r$ for all $x\in B_l^k(0)$.
\end{compactenum}
\end{lemma}
\begin{proof}
Obviously it is enough to show the Lemma for a fixed $i\in \set{1,\dots,n}$. Hence let $i\in \set{1,\dots,n}$ be fixed for the rest of the proof. As in \cite[Lemma 3.7]{Gloeckner2} we define $H\co B_5(0)\ti B_{\ep_{\exp}}(0)\ra \R^m$, $(x,y)\ms \exp_i(x,y)-x-y$ and $h\co B_5(0)\ti B_{\ep_{\exp}}(0)\ra [0,\8[$, $(x,y)\ms \|H'(x,y)\|_{op}$.  For all $x\in B_5(0)$ we get $d_1H(x,0;\bl)=0$ and $d_2H(x,0;\bl)=0$ and so $H'(x,0) = dH(x,0;\bl) =0$ in $\mathcal{L}(\R^m\ti\R^m;\R^m)$.  Hence $\ol{B}_l(0)\ti \set{0} \subs h\inv ([0,\frac{r}{r+10}[)$ and with the Lemma of Wallace we find $\ep\in ]0,\min(\ep_{\exp},\frac{r}{2})[$ such that $\|H'(x,y)\|_{op} <\frac{r}{r+10}$ for all $x\in {B}_l(0)$ and $y \in B_{\ep}(0)$. Now let $\eta \in C^\o(B_R^k(0),\R^m)_\str$ with $\|\eta\|_l^1<\ep$.
\begin{compactenum}
\item \label{kl1} We have 
\begin{align}\label{H}
\ps^i_\eta(x)=H(x,\eta(x)) +x + \eta(x) 
\end{align}
for all $x\in B_l^k(0)$. Hence  
\begin{align*}
{\ps^i_\eta}'(x)= H'(x,\eta(x);\bl) \ci (\id_{\R^m},\eta'(x)) + \id_{\R^m} + \eta'(x)
\end{align*}
for all $x\in B_l^k(0)$. With Remark \ref{idA} we calculate
\begin{align}\label{nahe1.1}
&\|{\ps^i_\eta}'(x)-\id_{\R^m}\|_{op} \leq \|H'(x,\eta(x))\|_{op} \cdot \|(\id_{\R^m},\eta'(x))\|_{op} + \|\eta'(x)\|_{op} \nonumber\\
<& \frac{r}{r+10} \cdot (1+\ep)+ \ep \leq \frac{r}{r+2} \cdot \left(1+\frac{r}{2}\right)+ \frac{r}{2} =r. 
\end{align}
\item \label{kl12} Let $x\in B_l(0)$ and $y \in B_{\ep}(0)$. Then $\|(x,y)\| \leq \|x\|+\|y\| < l+\ep\leq 5+\frac{r}{2}$. Hence
\begin{align*}
&\|H(x,y)\|=\|H(x,y)- H(0,0)\| = \left\|\int_0^1dH(tx,ty;x,y)dt \right\| \\
\leq &\int_0^1\|H'(tx,ty)\| \cdot \|(x,y)\|dt < \frac{r}{r+10} \cdot \|(x,y)\| \leq \frac{r}{2}.
\end{align*}
Thus given $x\in B^k_l(0)$ we can calculate with (\ref{H})
\begin{align}\label{nahe1.2}
\|\ps^i_\eta(x)-x\| \leq \|H(x,\eta(x))\| + \|\eta(x)\| <r.
\end{align}
\end{compactenum}
\end{proof}

As in \cite[Lemma 3.7]{Gloeckner2} we will use the following well known fact:
\begin{remark}\label{PraktischPraktisch}
Let $U \subs \C^m$ be open and convex and $f\co U \ra \C^m$ complex analytic with $\|df(x)-\id_{\C^m}\|<1$ for all $x\in U$. In this situation $f$ is injective and hence $f$ has open image and is a diffeomorphism onto its image: Let $x\neq y \in U$. Because $U$ is convex we can define  
\begin{align*}
\tau \co [0,1]\ra \C^m,~ t \ms df((1-t)x+ty;y-x)-(y-x).
\end{align*}
We get $\|\tau(t)\|< \|y-x\|$ for $t \in [0,1]$ and so $\int_0^1 \|\tau(t)\| dt < \|y-x\|$. With $f(y)-f(x)= y-x + \int_0^1\tau(t)dt$ we get $f(y)\neq f(x)$.
\end{remark}

\begin{lemma}\label{InjektivitaetPsiEta}
There exists $\ep \in ]0,\ep_{\exp}[$, such that for all $\eta \in \mathcal{B}_{1,\ep}^1$ the map $\ps_\eta \co M \ra \til{M}$ is injective.
\end{lemma}
\begin{proof}
Because of Lemma \ref{Nahe1}, Remark \ref{PraktischPraktisch} and Lemma \ref{Hilfreich1}, we can find $\ep_1\in ]0,\ep_{\exp}[$ such that $\ps_{\eta^{k_0}} \co U_{l,1} \ra M$ is injective for all $\eta \in \mathcal{B}_{1,{\ep_1}}^1$ and $l \in \set{1,\dots,n}$.
Similar to \cite[4.10]{Gloeckner2} one can show that given $i,j \in \set{1,\dots,n}$ we find $\ep_{i,j}\in ]0,\ep_1[$ such that 
$\ps_\eta$ is injective on $U_{i,1}\cup U_{j,1}$ for all $\eta \in \mathcal{B}^1_{1,\ep_{i,j}}$: 
Suppose the opposite. Then we find sequences $(\eta^k)_{k\in \N}$ in $\gg^\o_\str(TM)$, $(p_k)_{k\in\N}$ in $U_{i,1}$ and $(q_k)_{k\in\N}$ in $U_{j,1}$ such that for all $k \in \N$ we have $\|\eta^{k}\|_{\ol{B}^{k_l}_1}^1<\frac{1}{k}$, $p_k \neq q_k$ and $\ps_{\eta^k}(p_k) = \ps_{\eta^k}(q_k)$. Because $\ol{U_{i,1}}$ and $\ol{U_{j,1}}$ are compact, we can assume without loose of generality that there exist $p \in \ol{U_{i,1}}$ and $q \in \ol{U_{j,1}}$ such that $p_k \ra p$ and $q_k \ra q$. Hence $\eta^k(p_k) \ra 0_p$ and $\eta^k(q_k) \ra 0_q$ in $TM$. Therefore $\ps_{\eta^k}(p_k) \ra p$ and $\ps_{\eta^k}(q_k) \ra q$. Thus $p=q$. There exists $l \in \set{1,\dots,n}$ such that $p=q \in U_{l,1}$. Hence there exist $k_0$ and $j_0$ such that $p_{k_0},p_{j_0} \in U_{l,1}$. The map $\ps_{\eta^{k_0}} \co U_{l,1} \ra M$ is injective. But $p_{k_0} \neq q_{k_0}$ and $\ps_{\eta^{k_0}}(p_{k_0}) = \ps_{\eta^{k_0}}(q_{k_0})$. This is a contradiction. Now $\ep:= \min\{\ep_{i,j}: i,j\in \set{1,\dots,n}\}$ is as needed.
\end{proof}

\begin{definition}
For $j \in \set{1,\dots,m}$ and each connected component $C$ of $\partial^jM$ we fix a point $p_C^j\in \partial^jM$. The submanifold $\partial^jM$ is totaly geodetic in $\til{M}$. Hence given a point $p_C^j\in \partial^jM$ there exists an open neighbourhood $U$ of $0_{p_C^j}$ in $T_{p_C^j}\partial^jM$ such that $\exp_{p_C^j} (U) \subs C$ and $\exp_{p_C^j} \co U\ra C$ is continuous. Thus we find $\ep_{C}^j\in ]0,\ep_{\exp}[$ such that for all $\eta\in \mathcal{B}^0_{1,\ep_C^j}$ we have $\ps_\eta(p_C^j) \in C$.
\end{definition}

\begin{remark}\label{PC2Q}
Let $j \in \set{1,\dots,m}$ and $v \in T_p\partial^jM$ with $[0,1]v \subs \Omega_{\til{M}}$ and $\exp_{\til{M}}([0,1]v) \subs \partial^j M$. Then $v \in \Omega_{\partial^j M}$: We consider the curve $\gamma \co [0,1] \ra \partial^jM$, $t\ms \exp_{\til{M}}(tv)$. Because $\partial^jM$ is totally geodetic we see that $\gamma$ is also a geodetic for $\partial^jM$.
\end{remark}

\begin{lemma}\label{GehtnachM}
There exists $\ep \in ]0,\ep_{\exp}[$ such that for all $\eta\in \mathcal{B}_{1,\ep}^1$ the map $\ps_\eta$ is injective and $\ps_\eta(\partial^jM)=\partial^jM$. Moreover if $C$ is a connected component of $\partial^jM$ then $\ps_\eta(C)=C$.
\end{lemma}
\begin{proof}
We use Lemma \ref{InjektivitaetPsiEta} and Definition \ref{PC2Q} to choose $\ep >0$ such that for all $\eta\in \mathcal{B}_{1,\ep}^1$ the map $\ps_\eta \co M \ra \til{M}$ is injective and $\ps_\eta(p_C^j) \in C$ for all $j \in \set{1,\dots,m}$ and all connected components $C \subs \partial^jM$. Note that the strata of $M$ have only finite many connected components, because $M$ is compact. Now we show by induction over $j$ from $m$ to $0$ that $\ps_\eta(\partial^j M) = \partial^j M$. The case $j=m$ is clear. For the induction step we chose let $C$ be a connected component of $\partial^jM$ and $Z:=\set{x\in C: (\forall t \in [0,1])\ps_{t\eta}(x) \in C}$. Because $p_C^j \in Z$ we get $Z \neq \emptyset$. Now let $p \in C$. We have $[0,1] \eta|_C(p) \subs \Omega_{\til{M}}$ and $\exp_{\til{M}} (t\eta|C(p)) \in C$ for all $t \in [0,1]$. We conclude $\eta|_C(p) \in \Omega_{\partial^jM}$ and $\exp_{\partial^jM}(\eta|C(p)) = \exp_{\til{M}}(\eta|_C(p)) \in C$. Hence there exists a $p$-neighbourhood $V \subs C$ such that $\exp_{\partial^jM}\ci \eta|_C(V) \subs C$. We conclude that $Z$ is open in $C$. Now let $p \in C\setminus Z$. We get $\ps_\eta(p)\notin C$. First suppose $\ps_\eta(p) \in \ol{C}\setminus C = \bigcup_{j<i} \partial^iM$. Because $\ps_\eta$ is injective and $\ps_\eta(\partial^iM) = \partial^iM$ for all $j<i$ we conclude $p \in \partial^iM$ for $i>j$. But this is a contradiction. Now suppose $\ps_\eta(p) \in \til{M}\setminus \ol{C}$. Then there exists a $p$-neighbourhood $V$ in $M$ such that $\ps_\eta(V) \subs \til{M}\setminus \ol{C}$. But $C \cap V$ is a $p$-neighbourhood in $C$. Hence $Z$ is closed. Therefore $Z=C$. We conclude $\ps_\eta(C) \subs C$ and obtain a continuous injective map $\ps_\eta|_C^C \co C \ra C$. From $\ps_\eta(\ol{C}\cap C) \cap C = \emptyset$ we conclude $\ps_\eta(C)=\ps_\eta(\ol{C}) \cap C$. And because $\ol{C}$ is compact we see that $\ps_\eta(C)$ is closed in $C$. But $\ps_\eta$ is also an open map, because it is injective and continuous (invariance of domain). We conclude $\ps_\eta(C)=C$.
\end{proof}

The following Lemma \ref{PreUKF} is a direct consequence of \cite[Lemma 2.2.3]{Margalef}
\begin{lemma}\label{PreUKF}
Let $\til{f}\colon \til{U}\ra \til{V}$ be a homeomorphism between open subsets of $\R^m$ and $\til{V}$ be convex, such that 
\begin{align*}
\til{f}(\til{U}\cap \partial \R^m_k) \subs \til{V}\cap \partial \R^m_k \text{ and } \til{f}(\til{U}\cap \partial^0\R^m_k) \cap \partial^0\R^m_k \neq \emptyset
\end{align*}
then 
\begin{align*}
\til{f}(\til{U}\setminus \R^m_k) \subs \til{V}\setminus \R^m_k.
\end{align*}
\end{lemma}

\begin{lemma}
Let $U\subs \R^m_k$ be open, $f\colon U \ra \R^m_k$ be a real analytic map and $x_0 \in U$  such that
\begin{align*}
f(U\cap \partial \R^m_k) \subs \partial \R^m_k \text{ and } f'(x_0) \in \text{GL}(\R^m).
\end{align*}
We can find an open $x_0$-neighbourhood $U'\subs U$ and an open $f(x_0)$-neighbourhood $V\subs \R^m_k$ such that $f|_{U'}^V\colon U'\ra V$ is a real analytic diffeomorphism.
\end{lemma}
\begin{proof}
Without loose of generality we can assume $x_0 \in \partial U:=U\cap \partial \R^m_k$, because otherwise we can use the standard inverse function theorem. Now let $\til{f}\colon \til{U}\ra \R^m$ be a real analytic extension of $f$. Without loose of generality  we can assume $\til{U}\cap \R^m_k=U$. We have $\til{f}|_U=f$ and $\til{f}'(x_0)= f'(x_0) \in \GL(\R^m)$. Let $\til{U}'\subs \til{U}$ be a $x_0$-neighbourhood, $\til{V}\subs\R^m$ a $f(x_0)$-neighbourhood such that $\til{f}|_{\til{U}'} \co \til{U}'\ra \til{V}$ is a real analytic diffeomorphism between open sets of $\C^m$. Without loose of generality we can assume that $\til{V}$ is convex. 
We have
\begin{align*}
\til{f}(\til{U}'\cap \partial\R^m_k) = f(\til{U}' \cap \partial\R^m_k) \subs \til{V}\cap \partial \R^m_k.
\end{align*}
On the other hand $\til{f}(\til{U}' \cap \partial^0\R^m_k) \subs \R^m$ is open in $\R^m$ and not empty. Therefore $\til{f}(\til{U}' \cap \partial^0\R^m_k)$ is not contained in $\partial \R^m_k$.
With Lemma \ref{PreUKF} we get $\til{f}(\til{U}'\setminus \R^m_k) \subs \til{V}\setminus \R^m_k$. Using $\til{f}(\til{U}'\cap \partial\R^m_k) \subs \til{V}\cap \partial \R^m_k$ we get 
\begin{align*}
\til{f} (\til{U}'\cap \R^m_k) = \til{V}\cap \R^m_k.
\end{align*}
Now we define $U':=\til{U}' \cap \R^m_k \subs \til{U}'$ and $V:= \til{V}\cap \R^m_k \subs \til{V}$. The map $f|_{U'}^V\colon U' \ra V$ is bijective real analytic and also $(f|_{U'}^V)\inv = \til{f}\inv|_{V}$ is real analytic.
\end{proof}

\begin{theorem}\label{Zentral}
There exists $\ep_{diff}\in ]0,\ep_{\exp^\ast}[$ such that, if $\eta \in \mathcal{B}_{1,\ep_{diff}}^1$, then $\ps_\eta \co M \ra M,~ p \ms \ps_\eta(p)$ is a diffeomorphism.
\end{theorem}
\begin{proof}
We simply use the $\ep$ defined in Lemma \ref{GehtnachM}.
Then $\ps_\eta$ is a real analytic diffeomorphism, because it is bijective and a local real analytic diffeomorphism.
\end{proof}

The idea of the following Lemma \ref{HP} bases manly on \cite[4.12]{Gloeckner2}. But because our manifold is compact we can find a single $\ep$.
\begin{lemma}\label{HP}
There exists $\ep_{inj}\in ]0,\ep_{\exp}[$ such that  for all $p \in M$ the map $\exp_p \co \til{\Omega}_p \subs T_p\til{M}\ra \til{M}$ is injective on
\begin{align*}
W_pM:= \bigcup_{i=1}^n T{\ph}_i\inv (\set{\ph_i(p)} \ti B_{\ep_{inj}} (0)) \subs T_pM.
\end{align*}
\end{lemma}
\begin{proof}
Let $\ep_{inj}\in ]0,\ep_{\exp}[$ such that 
\begin{align*}
T(\ph_i\ci \ph_j\inv) \left(\set{\ph_j(p)} \ti B_{\ep_{inj}}(0)\right) \subs \set{\ph_i(p)} \ti B_{\ep_{\exp}}(0)
\end{align*}
for all $p\in \ol{U}_{4,i} \cap \ol{U}_{4,j}$ and $i,j\in \set{1,\dots,n}$. For $i \in \set{1,\dots,n}$ let $A_i':= T\ph_i\inv(\set{\ph_i(p)} \ti B_{\ep_{inj}}(0)) \subs T_pM$ and $A_i:= T\ph_i\inv(\set{\ph_i(p)} \ti B_{\ep_{\exp}}(0)) \subs T_pM$. Now  let $v,w\in \bigcup_{i}A_i'$ eg. $v \in A_i'$ and $w\in A_j'$ with $\exp(v) = \exp(w)$. We know that $\exp$ is injective on $A_j$. But obviously $v,w\in A_j$ by the choose of $\ep_{inj}$.
\end{proof}

As we use the local approach the following definition corresponds to \cite[4.13]{Gloeckner2}.
\begin{definition}
We define 
\begin{align*}
\ep_{\mathcal{U}}:= \min(\ep_{diff},\ep_{inj}).
\end{align*}
Moreover we define $\mathcal{V}:= \mathcal{B}^1_{1,\ep_{\mathcal{U}}} \subs \Gamma^\o_\str(TM)$,  $\mathcal{U}:= \set{\ps_\eta: \eta \in \mathcal{V}} \subs \Diff(M)$ and 
\begin{align*}
\Psi \co \mathcal{V}\ra \mathcal{U},~ \eta \ms \ps_\eta.  
\end{align*}
Given $\alpha \in \mathcal{U}$ we find $\eta \in \mathcal{V}$ with $\alpha = \ps_\eta$. We get $\eta_{(i)}(x) \in B_{\ep_{inj}}(0)$ for all $x \in \ol{B}^{k_i}_1(0)$. Hence $\eta(p)\in W_pM$ for all $p \in U_{i,1}$ and all $i=1,..,n$. Thus $\eta(p) \in W_pM$ for all $p \in M$. Therefore $\alpha (p) =\ps_\eta (p)= \exp|_{W_pM}(\eta(p)) \in \exp|_{W_pM}(W_pM)$ and $\eta(p) = \exp|_{W_pM}\inv(\alpha(p))$. Hence the map
\begin{align*}
\Phi \co \mathcal{U}\ra \mathcal{V},~ \alpha \ms \Phi(\alpha)
\end{align*}
with $\Phi(\alpha)(p)= \exp|_{W_pM}\inv(\alpha(p))$ makes sense and is inverse to $\Psi$.
\end{definition}

\section{Preparation for results of smoothness}\label{SecPreparation for results of smoothness}
To show the smoothness of the group actions we need some further definitions and results. Especially we need results concerning extensions of real analytic maps on $B^k_\ep(0)$ to open sets of $\C^m$. In this section we elaborate these foundations.

The following Lemma is the standard quantitative inverse function theorem for Lipschitz continuous maps (see \cite[Theorem 5.3]{Gloeckner5} and \cite{Wells}).
\begin{lemma}\label{StandQuand}
Let $A\co \R^m \ra \R^m$ be a linear isomorphism, $x_0 \in \R^m$, $r>0$ and $g\co B_r(x_0) \ra \R^m$ Lipschitz continuous with $\Lip(g)<\frac{1}{\|A\inv\|_{op}}$. If we define $a:= \frac{1}{\|A\inv\|_{op}} - \Lip(g)$, $b:= \|A\|_{op} + \Lip(g)$ and $f\co B_r(x_0) \ra \R^m$, $x\ms Ax+g(x)$ we have
\begin{align*}
B_{as}(f(x)) \subs f(B_s(x)) \subs B_{bs}(f(x))
\end{align*}
for all $x\in B_r(x_0)$ and $s\in ]0,r-\|x-x_0\|]$. Moreover $f$ has open image and is a homeomorphism onto its image.
\end{lemma}

In \cite{Gorni} Gorny showed a \textbf{qualitative} inverse function theorem for Lipschitz continuous maps on open sets of half-spaces  in  Banach spaces (``open sets with \textbf{boundary}'', that means the local  ``boundary''-case). In \cite[Remark on p 47]{Gorni} she states the open problem whether there is also a \textbf{qualitative} inverse function theorem for Lipschitz continuous functions on ``open sets with corners'' (that means the local "corner"-case).

Our Lemma \ref{QanUKFEcke} is a \textbf{quantitative} inverse function theorem for Lipschitz continuous maps on open sets with corners in $\R^m$, but the proof can be transferred to the Banach case verbatim by substituting $\R^m$ with a Banach space.

\begin{lemma}\label{LipAusd}
Let $x_0 \in \R^m$ and  $g \co B_r^k (x_0) \ra \R^m$ Lipschitz continuous with $\Lip(g)=:L$. Let $\|\bl\|_k\co B_r^k\ra \R^m_k$, $x\ms \|x\|_k$ with 
\begin{align*}
(\|x\|_k)_i= 
\begin{cases}
|x_i| &: i\leq k \\
x_i &: \text{otherwise}
\end{cases}
\end{align*}
and $\til{g} \co B_r(x_0) \ra \R^m$, $x \ms g(\|x\|_k)$. In this situation $\til{g}|_{B^k_r(x_0)} =g$ and $\Lip(\til{g}) =\Lip(g)$.
\end{lemma}
\begin{proof}
The map $\|\bl\|_k\co B_r^k\ra \R^m_k$ is Lipschitz continuous with Lipschitz constant $1$. The assertion now follows from $\Lip(\til{g}) = \Lip(g\ci \|\bl\|_k) \leq \Lip(g)$.
\end{proof}

\begin{lemma}\label{QanUKFEcke}
Let $A\co \R^m \ra \R^m$ be a linear isomorphism, $x_0 \in \R^m$, $r>0$ and $g\co B^k_r(x_0) \ra \R^m$ be Lipschitz continuous with $\Lip(g)<\frac{1}{\|A\inv\|_{op}}$. We define $a:= \frac{1}{\|A\inv\|_{op}} - \Lip(g)$, $b:= \|A\|_{op} + \Lip(g)$ and $f\co B^k_r(x_0) \ra \R^m$, $x\ms Ax+g(x)$. If $f(\partial B_r^k(x_0)) \subs \partial \R^m_k$ and $f(\partial^0B_r^k(x_0)) \subs \partial^0\R^m_k$, then we have
\begin{align*}
B_{as}^k(f(x)) \subs f(B_s^k(x)) \subs B_{bs}^k(f(x))
\end{align*}
for all $x\in B_r^k(x_0)$ and $s\in ]0,r-\|x-x_0\|]$. Moreover $f(B^k_r(x_0))$ is open in $\R^m_k$ and $f$ is a homeomorphism onto its image.
\end{lemma}
\begin{proof}
Without loose of generality we can assume $B^k_r(x_0)\neq \emptyset$. Hence $\partial^0B^k_r(x_0)\neq \emptyset$.  
We define $\til{g} \co B_r(x_0) \ra \R^m$ as in Lemma \ref{LipAusd}. Then $\Lip(\til{g}) = \Lip(g) <\frac{1}{\|A\inv\|_{op}}$. Let $\til{f} \co B_r(x_0) \ra \R^m$, $x\ms Ax + \til{g}(x)$ and $a:= \frac{1}{\|A\inv\|_{op}} - \Lip(\til{g}) = \frac{1}{\|A\inv\|_{op}} - \Lip({g})$. Using Lemma \ref{StandQuand} we get:
\begin{compactenum}
\item The map $\til{f} \co B_r(x_0) \ra \R^m$ has open image and is a homeomorphism onto its image; 
\item\label{Qan1} For all $x \in B_r(x_0)$ and $s \in ]0,r-\|x-x_0\|[$ we have $B_{as}(\til{f}(x)) \subs \til{f}(B_s(x)) \subs B_{bs}(\til{f}(x))$;
\item\label{Qan2} $\til{f}|_{B_r^k(x_0)} = f$;
\item $\til{f}(\partial B_r^k(x_0)) = f(\partial B^k_r(x_0)) \subs \partial \R^m_k$ and $\til{f}(\partial^0 B_r^k(x_0)) \cap \partial^0 \R^m_k \neq \emptyset$.
\end{compactenum}
Let  $x \in B_r^k(x_0)$ and $s \in ]0,r-\|x-x_0\|[$.  We have to show 
\begin{align*}
B^k_{as}({f}(x)) \subs {f}(B^k_s(x)).  
\end{align*}
The inclusion $f(B_s^k(x)) \subs B^k_{bs}(b(x))$ follows directly from (\ref{Qan1}). With (\ref{Qan1}) and (\ref{Qan2}) we get $B_{as}^k(f(x))\subs B_{as}(\til{f}(x)) \subs \til{f}(B_s(x))$. Now let $y \in B_{as}^k(f(x))$. We find $z \in B_s(x)$ such that $y = \til{f}(z)$. It is left to show $z\in \R^m_k$. Let $\til{U}:= \til{f}\inv(B_{as}(f(x)))$ and $\til{V}:= B_{as}(f(x))$. Then $\til{U}$ is an open subset of $B_r(x_0)$ and $\til{f} \co \til{U} \ra \til{V} = B_{as}(f(x))$ is a homeomorphism. Because $\til{f}(z) =y \in B_{as}(f(x))$ and $\til{f}(x) = f(x) \in B_{as}(f(x))$ we get
\begin{align*}
z \in \til{U} \tx{and} x \in \til{U}.
\end{align*}
Hence $\til{U} \cap \R^m_k \neq \emptyset$ and so $\til{U} \cap \partial^0 \R^m_k \neq \emptyset$. Therefore 
\begin{align*}
\emptyset \neq \til{f}(\til{U}\cap \partial^0 \R^m_k) = f(\til{U} \cap \partial^0 \R^m_k) \subs \partial^0 \R^m_k.
\end{align*}
And so 
\begin{align}\label{QanW1}
\til{f}(\til{U} \cap \partial^0 \R^m_k) \cap \partial^0\R^m_k \neq \emptyset.
\end{align}
Moreover we get 
\begin{align}\label{QanW2}
\til{f}(\til{U} \cap \partial \R^m_k) = f(\til{U} \cap \partial\R^m_k) \subs \partial \R^m_k.
\end{align}
Using (\ref{QanW1}), (\ref{QanW2}) and Lemma \ref{PreUKF} we get 
\begin{align*}
\til{f}(\til{U}\setminus \R^m_k) \subs \til{V}\setminus \R^m_k.
\end{align*}
Suppose $z \notin \R^m_k$, then $\til{f}(z)=y \notin \R^m_k$. But this would be a contradiction. Hence $z \in \R^m_k$. Now we show that $f \co B_r^k(x_0) \ra \R^m_k$ has open image and is a homeomorphism onto its image. Let $\til{X}:=B_r(x_0)$, $X:=B_r^k(x_0)$, $\til{Y}:= \til{f}(\til{X}) \subs \R^m$ and $Y := f(X) \subs \R^m_k$. Now the assertion follows from the fact, that $\til{f} \co \til{X} \ra \til{Y}$ is a homeomorphism and $\til{f}|^Y_{X}  = f\co X \ra Y$.
\end{proof}

\begin{definition}
Given $r\in ]0,1[$ we choose $r^{op} \in ]0,1[$ such that for all $A\in \mathcal{L}(\C^m)$ we  have
\begin{align*}
\|A-\id\|_{op}<r^{op} \Rightarrow \|A\inv - \id\|_{op}<r.
\end{align*}
\end{definition}

\begin{remark}\label{Klar}
Let $\K\in \set{\R,\C}$, $U \subs \K^n$ and $f\co U \ra \K^n$ be an injective map. Then 
\begin{align*}
\|(f|^{f(U)})\inv -\id_{f(U)} \|_\8 = \|f- \id_U\|_\8.
\end{align*}
\end{remark}
\begin{proof}
Let $g:= f-\id_\U \co U \ra \K^m$. Then $f(x)=x+g(x)$ for all $x\in U$. Hence $f\inv(y) = y - g(f\inv(y))$ for all $y \in f(U)$. And so $f\inv = \id_{f(U)} - g\ci f\inv$. Therefore
\begin{align*}
\|f\inv- \id_{f(U)}\|_\8 = \|g\ci f\inv\|_\8 = \|g\|_\8 =\|f- \id_U\|_\8.
\end{align*}
\end{proof}

In the following Lemma the points  (\ref{OhMann1}) and (\ref{OhMann2}) are in some sense a stronger version of \cite[Lemma D.4 (b), (c)]{BobDis} (smooth case without boundary).
\begin{lemma}\label{WennNahe1}
Let $l\in ]0,\8[$ and $r\in ]0,1[$ such that $l':= (1-r)l-r>0$. For all $f\in C^\o(B_l^k(0); \R^m)$ with $\|f-\id\|^1_\8 <r^{op}$, $f(\partial(B_l^{k}(0))) \subs \partial \R^m_{k}$ and $f(\partial^0(B_l^{k}(0))) \subs \partial^0 \R^m_{k}$ we get:
\begin{compactenum}
\item $f(B_l^k(0)) \subs B^k_{l+r}(0)$;
\item $f \co B^k_l(0) \ra \R^m_k$ has open image and is a real analytic diffeomorphism onto its image; \label{OhMann1}
\item $B^k_{l'}(0)\subs f(B_l^k(0))$ and the map $f\inv \co B_{l'}(0)\ra B_l^k(0)$ is a real analytic diffeomorphism;\label{OhMann2}
\item $\|f\inv(x)-x\| <r$ for all $x\in B_{l'}^k(0)$;
\item $\|(f\inv)'(x)-\id_{\R^m}\|_{op} <r$ for all $x\in B_{l'}^k(0)$.
\end{compactenum}
\end{lemma}
\begin{proof}
\begin{compactenum}
\item Given $x\in B^k_l(0)$ we calculate
\begin{align*}
\|f(x)\| =\|f(x)-x +x\| \leq r+l.
\end{align*}
Hence $f(B_l^k(0)) \subs B_{l+r}^k(0)$.
\item We define $g\co B^k_l(0)\ra \R^m$, $g:=f-\id_{\R^m}$. Given $x,y \in B_l^k(0)$ we calculate
\begin{align*}
&\|g(x)-g(y)\| = \left\|\int_0^1g'((1-t) x+ ty;x-y)dt\right\| \\
\leq&\int_0^1\|f'((1-t)x+ty;x-y)-\id_{\R^m}(x-y)\|dt \leq r \|x-y\|.
\end{align*}
Hence $g$ is Lipschitz continuous with $\Lip(g)<1$. Moreover $f(\partial^0 B^k_l(0)) \subs \partial^0 \R^m_k$ and $f(\partial B^k_l(0)) \subs \partial \R^m_k$. Therefore we can apply Lemma \ref{QanUKFEcke}. Hence $f(B^k_l(0))$ is open in $\R^m_k$ and $f\co B^k_l(0) \ra f(B^k_l(0))$ is a bijection. Because of $r^{op}<1$ the map $f$ is a local real analytic diffeomorphism and so $f \co B^k_l(0) \ra f(B^k_l(0))$ is a real analytic diffeomorphism.
\item Using Lemma \ref{QanUKFEcke} we get $B^k_{(1-r)l}(f (0)) \subs f (B_l^k(0))$. It is left to show $B^k_{l'}(0)\subs B^k_{(1-r)l}( f(0) )$. Given $\|x\| <l'$ we calculate
\begin{align*}
\|f (0)-x\| \leq \| f(0)-0\| + \|x\| <r+ (1-r)l-r= (1-r)l.
\end{align*}
\item This follows directly from Remark \ref{Klar}.
\item We have $\|f'(x)-\id_{\R^m}\|_{op} <r^{op}$ for all $x\in B_l^k(0)$. Hence $\|(f'(x))\inv-\id_{\R^m}\|_{op} < r$ for all $x\in B_l^k(0)$. Now let $y \in B_{l'}^k(0)$ then 
\begin{align*}
\|(f\inv)'(y) - \id_{\R^m}\|_{op} = \|(f'(f\inv(y)))\inv - \id_{\R^m}\|_{op} <r.
\end{align*}
\end{compactenum}
\end{proof}

\begin{definition}
Given $l,r$ like in Lemma \ref{WennNahe1} we write $\ep_{l,r}$ for the $\ep$ constructed in Lemma \ref{Nahe1} with $R=5$ and $r^{op}$ instate of $r$. 
\end{definition}

The following Lemma \ref{del} is the analogous statement to \cite[3.3]{Gloeckner2} in the real analytic case.
\begin{lemma}\label{del}
Given $\ep\in ]0,\ep_{\exp}]$ there exists $\delta(\ep) \in ]0,1[$ such that:
\begin{compactenum}
\item For $x\in \ol{B}_4(0)$ and $i\in \set{1,...,n}$, we have $B_{\delta(\ep)}(x) \subs \exp_i(x,B_{\ep}(0))$;
\item The set $D_\ep:= \bigcup_{x\in B_4(0)}\set{x}\ti B_{\delta(\ep)}(x)$ is open in $B_5(0)\ti \R^m$ and 
\begin{align*}
D_\ep \ra B_{\ep}(0),~ (x,z)\ms \exp_i(x,\bl)\inv(z) 
\end{align*}
is real analytic for all $i\in \set{1,...,n}$.
\end{compactenum}
\end{lemma}
\begin{proof}
\begin{compactenum}
\item Now again let $i\in \set{1,\dots, n}$ and consider the map $\exp_i \co B_{4,5}(0)\ti B_{\ep}(0) \ra B_5(0)$. Now we use Lemma \ref{UKFunk}. Given $x\in B_{4,5}(0)$, we find $r_x>0$ and $\delta_x>0$ such that for all $y \in B_{r_x}(x)$, we have $B_{\delta_x}(y) = B_{\delta_x}(\exp_i(y,0))\subs \exp_i(y,B_{\ep}(0))$. We find finite many $x_1,\dots,x_k \in B_{4,5}(0)$ with $\ol{B}_{4}(0)\subs \bigcup_{j=1}^kB_{r_{x_j}}(x)$ and set $\delta^i(\ep):= \min_{j}\delta_{x_j}$. We define $\delta(\ep):=\min_i\delta^i(\ep)$. Given $i\in \set{1,...,n}$ and $x \in \ol{B}_{4}(0)$ we find $j$ such that $x \in B_{r_{x_j}}(x_j)$. Hence $B_{\delta(\ep)}(x) \subs \exp_i(x,B_{\ep}(0))$.
\item For $(x_0,y_0) \in D_\ep$ we get $\|y_0-x_0\|<\delta(\ep)$. Let $\tau := \min\left(4-\|x_0\|, \fr{\delta(\ep)-\|y_0-x_0\|}{2}\right)$. Then $B_\tau(x_0)\ti B_\tau(y_0) \subs D_\ep$. Hence $D_\ep$ is open. Because $D_\ep$ is open and contained in the image of $B_4(0)\ti B_{\ep}(0) \ra B_4(0) \ti \R^m$, $(x,y)\ms \exp_i(x,y)$ we conclude that $D_\ep \ra B_{\ep}(0)$, $(x,z)\ms \exp_i(x,\bl)\inv(z)$ makes sense and is a real analytic diffeomorphism.
\end{compactenum}
\end{proof}

\begin{definition}
Let $k\in \set{1,\dots, n}$, $R>0$, $r>0$ and $x\in \R^m_k$. We define
\begin{align*}
B_{R,r}^{k,\C}(x):=B_R^k(x) + B^\C_r(0) \subs \C^m \tx{and} \ol{B}_{R,r}^{k,\C}(x) := \ol{B}_R^k(x) + \ol{B}^\C_r(0) = \ol{B_{R,r}^{k,\C}(x)}.
\end{align*}
Obviously $B_{R,r}^{k,\C}(x)$ is an open neighbourhood  of $\ol{B}_R^k(x)$ in $\C^m$. We also write $B_{R,r}^\C(x) := B_R(x) + B^\C_r(0)$ and $\ol{B}_{R,r}^\C(x) := \ol{B}_R(x) + \ol{B}^\C_r(0)$ for $x\in \R^m$.
\end{definition}

\begin{remark}\label{UmgebungenIII}
If $U\subs \C^m$ is open, $k\in \set{1,\dots,n}$, $x\in \R^m_k$  and $\ol{B}_R^k(x) \subs U$ then $\ol{B}_R^k(x)\ti \set{0} \subs \m\inv(U)$ for $\m\co \C^m \ti \C^m \ra \C^m$, $(x,y)\ms x+y$. With the Lemma of Wallace we find $r>0$ such that $\ol{B_{R,r}^{k,\C}(x)}\subs \ol{B}_R^{k}(x) + \ol{B}^\C_r(0)= \ol{B}_{R,r}^{k,\C}(x) \subs U$.
\end{remark}

\begin{definition}
For $i \in \set{1,\dots,n}$ let $\exp^\ast_i \co \Omega^\ast_i \ra B_5^{\C}(0)$ a complex analytic extension of $\exp_i \co \Omega_i \ra B_5(0)$. Because $\ol{B}_{4.7}(0)\ti \set{0} \subs \Omega^\ast_i$ we can use Remark \ref{UmgebungenIII} and find $r^{\ast}_1>0$ and $\ep^{\ast}_1>0$ such that $\ol{B}_{4.7,r^{\ast}_1}^{\C}(0)\ti B^\C_{\ep^{\ast}_1}(0) \subs \Omega^\ast_i$ for all $i\in \set{1, \dots, n}$.
\end{definition}

\begin{remark}
\begin{compactenum}
\item For all $i \in \set{1,\dots,n}$ the map $\exp_i^\ast(\bl,0) \co B^{\C}_{4.7,r^{\ast}_1}(0) \ra \C^m$, $x\ms \exp_i^\ast(x,0)$ is an extension of $\exp_i(\bl,0)\co B_{4.7}(0) \ra \R^m$, $x\ms x$. Hence we have 
\begin{align*}
\exp_i^\ast(x,0)= x,
\end{align*} 
for all $x\in B^{\C}_{4.7,r^{\ast}_1}(0)$, because $B^{\C}_{4.7,r^{\ast}_1}(0)$ is connected.
\item For all $i \in \set{1,\dots,n}$ the map $d_2\exp_i^\ast(\bl,0) \co B^{\C}_{4.7,r^{\ast}_1}(0) \ra \mathcal{L}(\C^m)$, $x\ms d_2\exp_i^\ast(x,0)$ is a complex analytic extension of $d_2\exp_i(\bl,0)\co B_{4.7}(0) \ra \mathcal{L}(\R^m)$, $x\ms d_2\exp_i(x,0) = \id_{\R^m}$. Hence we have 
\begin{align*}
d_2\exp_i^\ast(x,0)= \id_{\C^m},
\end{align*} 
for all $x\in B^{\C}_{4.7,r^{\ast}_1}(0)$, because $B^{\C}_{4.7,r^{\ast}_1}(0)$ is connected.
\end{compactenum}
\end{remark}

\begin{lemma}\label{ExpDefII}
We define $O_1:=\ol{B}_{R,r^{\ast}_1}^{k,\C}(0)\ti B^\C_{\ep^{\ast}_1}(0)$. There exists $\ep_{\exp^\ast} >0$ and $r_{\exp^\ast} > 0$ such that:
\begin{compactenum}[(a)]
\item We have 
\begin{align*}
\ol{B}^\C_{4.5,r_{\exp^\ast}}(0) \ti \ol{B}^\C_{\ep_{\exp^\ast}}(0)\subs O_1 \subs B^\C_5(0)\ti \C^m
\end{align*}
for all $i\in \set{1,...,n}$.
\item For all $x\in \ol{B}^\C_{4.5,r_{\exp^\ast}}(0)$ and  $i\in \set{1,...,n}$ the map $\exp^\ast_{i,x}:=\exp^\ast_i(x,\bl) \co B^\C_{\ep_{\exp^\ast}}(0)\ra \C^m$ has open image and is a complex analytic diffeomorphism onto its image. Moreover the map $B^\C_{4.5,r_{\exp^\ast}}(0)\ti B^\C_{\ep_{\exp^\ast}}(0) \ra B^\C_{4.5,r_{\exp^\ast}}(0) \ti \C^m$, $(x,y)\ms (x,\exp^\ast_i(x,y))$ has open image and is a complex analytic diffeomorphism onto its image.
\end{compactenum}
In order to shorten the notation we define $U_{\exp^\ast}:= B^\C_{4.5,r_{\exp^\ast}}(0)$.
\end{lemma}
\begin{proof}
\begin{compactenum}
For every $x\in \ol{B}^\C_{4.7,r^{\ast}}(0)$ we find $r_x>0$ and $\ep_x>0$ such that:
\begin{compactitem}
\item
\item $\ol{B}_{r_x}^\C(x) \ti \ol{B}^\C_{\ep_x}(0) \subs O_1$.
\item for all $y \in B_{r_x}^\C(x)$ and $i\in \set{1,\dots,n}$ the map
\begin{align*}
\exp_i^\ast(y,\bl)\co B_{\ep_x}^\C(0) \ra \C^m
\end{align*}
has open image and is diffeomorphism onto its image.
\item For all $i\in \set{1,\dots,n}$ the map $\ol{B}_{r_x}^\C(x) \ti \ol{B}^\C_{\ep_x}(0) \ra \ol{B}_{r_x}^\C(x) \ti \C^m$, $(x,y) \ms \exp_i^\ast(x,y)$ 
has open image and is diffeomorphism onto its image.
\end{compactitem}
Let $r_{\exp^\ast}:=r^\ast/2$. We find finite many $x_1,\dots , x_n$ such that $\ol{B}^\C_{4.5,r_{\exp^\ast}}(0) \subs \bigcup_{j=1}^n B_{r_{x_j}}(x_{j})$. Let $\ep_{\exp^\ast}:=\min_j \ep_{x_j}$. Obviously we have
\begin{align*}
\ol{B}^\C_{4.5,r_{\exp^\ast}}(0) \ti \ol{B}^\C_{\ep_{\exp^\ast}}(0) \subs O_1
\end{align*}
Moreover the map $B_{4.5,r_{\exp^\ast}}^\C(0)\ti B^\C_{\ep_{\exp^\ast}}(0) \ra B_{4.5,r_{\exp^\ast}}^\C(0) \ti \C^m$, $(x,y)\ms (x,\exp_i(x,y))$ is injective and a local diffeomorphism. Hence it has open image and is a complex analytic diffeomorphism onto its image. Therefore we find $\ep_{\exp^\ast}$ as needed. 
\end{compactenum}
\end{proof}

\begin{lemma}\label{DeltaAst}
Let $r_{\log^\ast}:= r_{\exp^\ast}/2$. For all $\ep\in ]0,\ep_{\exp^\ast}]$ there exists $\delta^\C(\ep) \in ]0,1[$ such that:
\begin{compactenum}
\item For all $x\in \ol{B}_{4.7,r_{\exp^\ast}}^\C(0)$ and $i \in \set{1,\dots,n}$ we have $B_{\delta^\C(\ep)}^\C (x) \subs \exp_i^\ast(x,B_\ep^\C(0))$.
\item The set
\begin{align*}
D_\ep^\C:= \bigcup_{x\in B_{4,r_{\log^\ast}}^\C(0)} \set{x} \ti B_{\delta^\C(\ep)}^\C (x) \subs \C^m \ti C^m
\end{align*}
is open and $D_\ep^\C \ra B_\ep^\C(0)$, $(x,z) \ms \exp_i^\ast(x,\bl)\inv(z)$ is complex analytic.
\end{compactenum}
\end{lemma}
\begin{proof}
\begin{compactenum}
\item We consider the map $\exp_i^\ast \co B_{4.5,r_{\exp^\ast}}^\C(0) \ti B_{\ep_{\exp^\ast}}^\C \ra \C^m$ for $i =1, \dots ,n$. Given $x\in B_{4.5,r_{\log^\ast}}^\C(0)$ we use Lemma \ref{UKFunk} and find $r_x>0$ and $\delta_x>0$ such that for all $y \in B_{r_x}^\C(x)$ and $i\in \set{1,\dots,n}$
\begin{align*}
B_{\delta_x}^\C(y) = B_{\delta_x}^\C(\exp_i^\ast(y,0)) \subs \exp_i^\ast(y,B^\C_\ep(0)).
\end{align*}
We choose finite many $x_1,\dots, x_n\in B_{4.5,r_{\exp^\ast}}^\C(0)$ such that 
\begin{align*}
\ol{B}_{4,r_{\log^\ast}}^\C(0)\subs \bigcup_{j=1}^n B_{r_{x_j}}(x_j).
\end{align*}
We set $\delta^\C(\ep):= \min_j \delta_{x_j}$. If $i \in \set{1,\dots,n}$ and $x\in \ol{B}_{4,r_{\log^\ast}}^\C(0)$,
then $B_{\delta^\C(\ep)}^\C(x) \subs \exp_i^\ast(x,B_\ep^\C(0))$.
\item Given $(x_0,y_0) \in D_\ep^\C$ we find $\sigma>0$ with $B_\sigma^\C(x_0) \subs {B}_{4,r_{\log^\ast}}^\C(0)$. Defining $\tau:= \min \left(\sigma, \frac{\delta^\C(\ep)- \|y_0-x_0\|}{2}\right)$, we get $B_\tau(x_0)\ti B_\tau(y_0) \subs D_\ep^\C$.
\end{compactenum}
\end{proof}

\begin{definition}
Let $N_0 \in \N$ with $\frac{1}{N_0} <\min( r_{\exp^\ast},r_{\log^\ast} )$. For $R\in [0,4]$, $k\in \set{1,\dots,m}$ and $n \in \N$ we define 
\begin{align*}
U^k_{R,n} := B_{R,\frac{1}{n+N_0}}^{k,\C}(0) = B_R^{k}(0) + B^\C_{\frac{1}{n+N_0}}(0)
\end{align*}
and get a connected fundamental sequence of open neighbourhoods of $\ol{B}_R^{k}(0)$ in $\C^m$ with $U_n^k \subs  B^\C_{4.5,r_{\exp^\ast}}(0)$.
Moreover we define 
\begin{align*}
\ol{U}^k_{R,n} := \ol{B}_R^{k}(0) + \ol{B}^\C_{\frac{1}{n+N_0}}(0) = \ol{{U}^k_{R,n}}.
\end{align*}
\end{definition}

\begin{lemma}\label{Nahe1II}
For $r_0 \in ]0,1[$ there exists $\ep^\C_{r_0} \in ]0,\ep_{\exp^\ast}]$ such that for all $n\in \N$, $i\in \set{1,\dots,n}$ and $\eta \in \Hol^1_b(U_{4,n}^{k_i}; \C^m)_\str^\R$ with $\|\eta\|_\8^1<\ep^\C_{r_0}$ we have $\|\ps_\eta^{k_i} - \id\|^1_{U_{4,n}^{k_i}} <r_0$.
\end{lemma}
\begin{proof}
Obviously it is enough to show the Lemma for a fixed $i\in \set{1,\dots,n}$. Hence let $i\in \set{1,\dots,n}$ be fixed for the rest of the proof. We define $H\co U_{\exp^\ast} \ti B^\C_{\ep_{\exp^\ast}}(0)\ra \C^m$, $(x,y)\ms \exp^\ast_i(x,y)-x-y$ and $h\co U_{\exp^\ast} \ti B^\C_{\ep_{\exp^\ast}} (0) \ra [0,\8[$, $(x,y)\ms \|H'(x,y)\|_{op}$. For all $x\in U_{\exp^\ast}$ we get $d_1H(x,0;\bl)=0$ and $d_2H(x,0;\bl)=0$ and so $H'(x,0) = dH(x,0;\bl) =0$ in $\mathcal{L}(\C^m\ti\C^m;\C^m)$. Hence $\ol{ U_{4,1}^{k_i} } \ti \set{0} \subs h\inv ([0,\frac{r_0}{r_0+10}[)$ and with the Lemma of Wallace we find $\ep^\C_{r_0} \in ]0,\min(\ep_{\exp},\frac{r_0}{2})[$ such that $\|H'(x,y)\|_{op} <\frac{r_0}{r_0+10}$ for all $x\in \ol{ U_{4,1}^{k_i} }$ and $y \in B^\C_{\ep^\C_{r_0}}(0)$. Now let $\eta \in \Hol(U_{4,n}^{k_i},\C^m)^\R_\str$ with $\|\eta\|_\8^1<\ep^\C_{r_0}$.
\begin{compactitem}
\item  We have 
\begin{align}\label{HII}
\ps^i_\eta(x)=H(x,\eta(x)) +x + \eta(x) 
\end{align}
for all $x\in U_{4,n}^{k_i}$. Hence  
\begin{align*}
{\ps^i_\eta}'(x)= H'(x,\eta(x);\bl) \ci (\id_{\R^m},\eta'(x)) + \id_{\R^m} + \eta'(x)
\end{align*}
for all $x\in U_{4,n}^{k_i}$. With Remark \ref{idA} we calculate
\begin{align}\label{nahe1.1II}
&\|{\ps^i_\eta}'(x)-\id_{\C^m}\|_{op} \leq \|H'(x,\eta(x))\|_{op} \cdot \|(\id_{\C^m},\eta'(x))\|_{op} + \|\eta'(x)\|_{op} \nonumber\\
<& \frac{r}{r+10} \cdot (1+\ep)+ \ep \leq \frac{r}{r+2} \cdot \left(1+\frac{r}{2}\right)+ \frac{r}{2} =r. 
\end{align}
\item \label{kl12II} Let $x\in U_{4,n}^{k_i}$ and $y \in B^\C_{\ep^\C_{r_0}}(0)$. Then $\|(x,y)\| \leq \|x\|+\|y\| < 5+\frac{r}{2}$. Hence
\begin{align*}
&\|H(x,y)\|=\|H(x,y)- H(0,0)\| = \left\|\int_0^1dH(tx,ty;x,y)dt \right\| \\
\leq &\int_0^1\|H'(tx,ty)\| \cdot \|(x,y)\|dt < \frac{r}{r+10} \cdot \|(x,y)\| \leq \frac{r}{2}.
\end{align*}
Thus given $x\in U_{4,n}^{k_i}$ we calculate with (\ref{HII})
\begin{align}\label{nahe1.2II}
\|\ps^i_\eta(x)-x\| \leq \|H(x,\eta(x))\| + \|\eta(x)\| <r.
\end{align}
\end{compactitem}
\end{proof}

\begin{lemma}\label{WennNahe1II}
Let $R \in ]0,\8[$, $r\in ]0,\8[$, $r_0 \in ]0,1[$, $R':=(1-r_0)R -r_0>0$, $r':=(1-r_0)r$, $R'':=R+r_0$ and $r'':=(1+r_0)r$. For all $k\in \set{1,\dots,m}$ and $f\in \Hol_b^1(B^{k,\C}_{R,r}(0);\C^m)$ such that $\|f-\id\|_{\8}^1<r_0^{op}$, $f(\partial(B_l^{k}(0))) \subs \partial \R^m_{k}$ and $f(\partial^0(B_l^{k}(0))) \subs \partial^0 \R^m_{k}$, we get the following assertions:
\begin{compactenum}
\item $f \co B^{k,\C}_{R,r}(0) \ra \C^m$ has open image and is a real analytic diffeomorphism onto its image;
\item $B^{k,\C}_{R',r'}(0) \subs f(B^{k,\C}_{R,r}(0))$ and the map $f\inv \co B^{k,\C}_{R',r'}(0) \ra B^{k,\C}_{R,r}(0)$ is a real analytic diffeomorphism;
\item $f(B^{k,\C}_{R,r}(0)) \subs B^{k,\C}_{R,r+r_0}(0)$ and $f(B^{k,\C}_{R,r}(0)) \subs B^{k,\C}_{R'',r''}(0)$;
\item $\|f\inv(x)-x\| <r_0$ for all $x\in B^{k,\C}_{R',r'}(0)$;
\item $\|(f\inv)'(x)-\id_{\C^m}\|_{op} <r_0$ for all $x\in B^{k,\C}_{R',r'}(0)$.
\end{compactenum}
\end{lemma}
\begin{proof}
\begin{compactenum}
\item From Remark \ref{PraktischPraktisch} and $r_0<1$ we get that $f\co B^{k,\C}_{R,r}(0) \ra \C^m$ has open image and is a real analytic  diffeomorphism onto its image. 
\item Let $g:= f-\id_{\C^m} \co B^{k,\C}_{R,r}(0) \ra \C^m$. Then as in Lemma \ref{WennNahe1} one gets $\|g(x)-g(y)\| <r_0 \|x-y\|$ for all $x,y \in B^{k,\C}_{R,r}(0)$. Hence $\Lip(g)\leq r_0<1$. Now let $x\in B^k_R(0)$. Then $B^\C_r(x) = x+B^\C_r(0) \subs B^{k,\C}_{R,r}(0)$. We consider the map $f|_{B_r^\C(x)}$. Thus using the identification $\C^m\cong \R^{2m}$ in combination with Lemma \ref{StandQuand} we get $B_{(1-r_0)r}^\C(f(x)) \subs f(B_r^\C(x)) \subs f(B^{k,\C}_{R,r}(0))$. Hence $B^\C_{(1-r_0)r}(0)+f(x) \subs f(B^{k,\C}_{R,r}(0))$ for all $x \in B^{k}_{R}(0)$. Therefore $B^\C_{(1-r_0)r} (0) + f(B_R^k(0)) \subs f(B^{k,\C}_{R,r}(0))$. With Lemma \ref{WennNahe1} we get $B^k_{R'}(0) \subs f(B^k_R(0))$, because $f(\partial(B_l^{k}(0))) \subs \partial \R^m_{k}$ and $f(\partial^0(B_l^{k}(0))) \subs \partial^0 \R^m_{k}$. Therefore
\begin{align*}
B_{R',r'}^{k,\C}(0) = B_{R'}^k(0) + B^\C_{r'}(0) \subs f(B^{k,\C}_{R,r}(0)).
\end{align*}
Hence the map $f\inv \co B^{k,\C}_{R',r'}(0) \ra B^{k,\C}_{R,r}(0)$ makes sense and is a real analytic diffeomorphism. 
\item Let $x\in B^{k,\C}_{R,r}(0)$ then $f(x) =x + (f(x)-x)  \in B^{k,\C}_{R,r}(0) + B^\C_{r_0}(0) =B^{k,\C}_{R,r+r_0}(0)$. Now let $x_0 \in B_R^k(0)$. Using Lemma \ref{WennNahe1} we get $f(x_0) \in B^{k}_{R''}(0)$. Again we consider the map $f|_{B_r^\C(x_0)}$. Lemma \ref{StandQuand} yields
\begin{align*}
f(B_r^{\C}(x_0)) \subs B^\C_{(1+r_0)r}(f(x_0))= B^\C_{r''}(0)+f(x_0)\subs B^\C_{r''}(0)+B^k_{R''}(0)= B_{R'',r''}^{k,\C}(0).
\end{align*}
\item Because $\|f-\id\|^0_\8 <r_0$ we can use Remark \ref{Klar}  and get $\|f\inv(x)-x\| <r_0$ for all $x\in B^{k,\C}_{R',r'}(0)$.
\item Because $\|f'(x)-\id\|_{op} <r_0^{op}$ for all $x \in B^{k,\C}_{R,r}(0)$ we can argue as in Lemma \ref{Nahe1} and get $\|(f\inv)'(x)-\id_{\C^m}\|_{op} <r_0$ for all $x\in B^{k,\C}_{R',r'}(0)$.
\end{compactenum}
\end{proof}

\begin{lemma}\label{invEcken}
There exists $\ep_\partial >0$ such that for all $i \in \set{1,\dots,n}$, $j \in \set{1,\dots, m}$, all connected components $C \subs \partial^jB_5^{k_i}(0)$, $x,y \in C$ with $\|x-y\|<\frac{\ep_\partial}{2}$, $v\in B_{\ep_\partial}(0)$ with $y = \exp^i(x,v)$ we get $v\in T_x \partial^jB_5^{k_i}(0)$.
\end{lemma}
\begin{proof}
There exists $\ep_\partial \in ]0,\ep_{\exp}]$ such that $\|d\exp^i_x(y)-\id_{\R^m}\|<\frac{1}{2}$ for all $i \in \set{1,\dots,n}$, $x \in B_4(0)$ and $y \in B_{\ep_\partial}(0)$. Given $x\in \partial^jB_5^{k}(0)$ we define the subset $J_x \subs \set{1,\dots,m}$ such that $x_l=0$ if and only if $l \in J_x$. Moreover we write $I_x:= \set{1,\dots,m}\setminus J_x$ and $\R^m_{I_x}:=\langle e_i:i \in I_x \rangle$. Obviously we have $T_x\partial^j B_5^k(0) = \R^m_{I_x}$. Let $v \in B_{\ep_{\exp}}(0) \cap T_x\partial^jB_5^{k_i} = B_{\ep_{\exp}}(0) \cap \R^m_{I_x} = B^{\R^m_{I_x}}_{\ep_{\exp}}(0)$. Because $\partial^jM$ is totally geodetic there exist $t_0 \in ]0,1[$ such that  $d_x\ph_i\inv (tv) \in \Omega_{\partial^jM}$ for all $t \in [0,t_0]$. Hence for all $t \in [0,t_0]$ we have $\exp^i_x(tv)\in \partial^jB_5^{k_i}(0)$. Therefore $\pr_l\ci \exp_x^i(tv) =0$ for all $t\in [0,t_0]$ and $l \in J_x$ (note that $t\ms \exp^i_x(tv)$ stays in the connected component of $x$). Because $\exp_x^i$ is real analytic we can use the Identity Theorem and obtain $\pr_l\ci \exp^i_x(v)=0$. Therefore $\exp^i_x (B^{\R^m_{I_x}}_{\ep_{\exp}}(0)) \subs \R^m_{I_x}$. Especially we can consider the continuous map 
$\exp_x^i \co B^{\R^m_{I_x}}_{\ep_{\partial}}(0) \ra \R^m_{I_x}$. We write $C$ for the connected component of $x$ in $\partial^jB^{k_i}_5(0)$. And because $\|d\exp^i_x(y)-\id\|<\frac{1}{2}$, we get 
\begin{align*}
C\cap \partial^j B_{\frac{1}{2}\ep_\partial}^{k_i}(x)\subs B^{\R^m_{I_x}}_{\frac{1}{2}\ep_{\partial}}(x) = B^{\R^m_{I_x}}_{\frac{1}{2}\ep_{\partial}}(\exp_x(0)) \subs \exp_x(B^{\R^m_{I_x}}_{\ep_{\partial}}(0))
\end{align*}
with the quantitative inverse function theorem. Now the assertion follows from the injectivity of $\exp_x^i \co B_{\ep_{\exp}} (0) \ra \R^m$.
\end{proof}

\begin{definition}
We use the notation of Lemma \ref{del} and define $\delta_\mathcal{U}:=\delta_{\min(\ep_{\mathcal{U}},\ep_\partial)}$, $D_\mathcal{U}:=D_{\min(\ep_{\mathcal{U}},\ep_\partial)}$ and
\begin{align*}
\log_i\co D_{\mathcal{U}} \ra B_{\min(\ep_{\mathcal{U}},\ep_{\partial})}(0), (x,z)\ms \exp_i(x,\bl)\inv(z)
\end{align*}
for all $i\in \set{1,\dots,n}$.
\end{definition}

\begin{remark}\label{OBdAConnected}
For all $j \in \set{1,\dots,m}$, $i \in \set{1,\dots,n}$ and all connected components $C$ in $\partial^jB_1^{k_i}(0)$ we choose $p_C^{i,j}$. After shrinking $\ep_{\mathcal{U}}$ we assume without loose of generality $\ps_{\eta_{(i)}}^i(p_C^{i,j}) \in C$ for all $i,j,C$ as above and $\eta \in \mathcal{B}^0_{1,\ep_\mathcal{U}}$.
\end{remark}

\begin{definition}\label{AllesKlein}
Let $i\in \set{1,\dots,n}$. We define the real analytic map 
\begin{align*}
\a_i \co B_4(0)\ti B_{\delta_{\mathcal{U}}}(0) \ra B_{\ep_\mathcal{U}}(0),~ (x,y)\ms \log_i(x,x+y) 
\end{align*}
and get $\a_i(x,0)=0$ for all $x\in B_4(0)$. Hence $d_1\a_i(x,0)=d\a(\bl,0)(x) =0$ in $\mathcal{L}(\R^m)$ for all $x\in B_4(0)$. Therefore $\ol{B}_3(0)\ti \set{0} \subs (d_1\a_i)\inv(\set{0})$.  We find $\nu_\a>0$ such that $\|d_1\a_i(x,y)\|_{op}<\frac{\ep_\mathcal{U}}{2}$ for all $x\in \ol{B}_3(0)$, $y \in B_{\nu_\a}(0)$ and $i\in \set{1,\dots,n}$. Let $K_\a\geq \sup \set{\|d_2\a_i(x,y)\|_{op} : x\in \ol{B}_3(0), y \in \ol{B}_{\nu_\a}(0)}$ for all $i\in \set{1,\dots,n}$.
\end{definition}

\section{Smoothness of composition}\label{SecSmoothness of composition}
In this section we show the smoothness of the composition in the local chart $\gph$.

\begin{definition}
Let $r_\diamond = \min (\frac{\delta_\mathcal{U}}{2}, \ep_\mathcal{U}, \frac{\ep_\mathcal{U}}{8K_\a}, \frac{\nu_\a}{2} , \frac{\ep_\partial}{4})$, $r_\diamond^\C:= \frac{\delta^\C(\ep_{\exp^\ast})}{2}$ and $\ep_\diamond:=\min( \ep_{\mathcal{U}}, \ep_{4,r_\diamond} \ep^\C_{r^\C_\diamond})$.
\end{definition}

The corresponding  result to the following lemma, in the case of a non compact smooth manifold without corners, is \cite[4.17]{Gloeckner2}.
\begin{lemma}
Let $\eta, \zeta \in \mathcal{B}_{\ep_\diamond}$
\begin{compactenum}
\item The map $\ps^i_{\eta_{(i)}} \ci \ps^i_{\zeta_{(i)}} \co B_3^{k_i}(0)\ra \R^m_{k_i}$ makes sense and the map
\begin{align*}
\eta_{(i)} \diamond \zeta_{(i)} \co B_3^{k_i}(0) \ra \R^m,~ x\ms \log_i (x,\ps_{\eta_{(i)}}^i \ci \ps^i_{\zeta_{(i)}} (x))
\end{align*}
makes sense and is a real analytic function with $\|\eta_{(i)} \diamond \zeta_{(i)}(x)\|<\min(\ep_{\mathcal{U}},\ep_{\partial})$ for all $x\in B_3^{k_i}(0)$. Moreover we have 
\begin{align}\label{split}
\ps^i_{\eta_{(i)} \diamond \zeta_{(i)}} = \ps^i_{\eta_{(i)}} \ci \ps^i_{\zeta_{(i)}} 
\end{align}
on $B_3^{k_i}(0)$ and $\eta_{(i)}\diamond \zeta_{(i)}$ is stratified.
\item For the map $\ps_\eta \ci \ps_\zeta \co M \ra M$ we have $\ps_\eta \ci \ps_\zeta(U_{i,3}) \subs U_{i,4}$ and 
\begin{align}\label{diamondsplitt2}
\ph_i \ci \ps_\eta \ci \ps_\zeta \ci \ph_i\inv =  \ps^i_{\eta_{(i)}} \ci \ps^i_{\zeta_{(i)}}.
\end{align}
\item The vector field 
\begin{align*}
\eta \diamond \zeta \co M \ra TM,~ p \ms \exp|_{W_pM}\inv (\ps_\eta \ci \ps_\zeta (p))
\end{align*}
makes sense and is real analytic and stratified. Moreover we have
\begin{align*}
(\eta\diamond \zeta)_{(i)}(x) = \eta_{(i)} \diamond \zeta_{(i)}(x)
\end{align*}
for all $x\in B_3^{k_i}(0)$ and $p = \ph_i\inv(x)$ and
\begin{align*}
\ps_{\eta\diamond \zeta} = \ps_\eta \ci \ps_\zeta.
\end{align*} 
\item $\eta \diamond \zeta$ is in $\mathcal{V}$.
\end{compactenum}
\end{lemma}
\begin{proof}
\begin{compactenum}
\item Using Lemma \ref{Nahe1} and Lemma \ref{WennNahe1} we get $\ps_{\zeta_{(i)}}(B_3^{k_i}(0)) \subs B_4^{k_i}(0)$. Hence $\ps^i_{\eta_{(i)}} \ci \ps^i_{\zeta_{(i)}} \co B_3^{k_i}(0)\ra \R^m_{k_i}$ makes sense. We have $r_\diamond\leq \frac{\delta_\mathcal{U}}{2}$ and so $\|\ps^i_{\eta_{(i)}} -\id\|< \frac{\delta_\mathcal{U}}{2}$ and $\|\ps^i_{\zeta_{(i)}} -\id\|< \frac{\delta_\mathcal{U}}{2}$ on $B_4^{k_i}(0)$. Hence $\ps^i_{\eta_{(i)}} (\ps^i_{\zeta_{(i)}} (x)) -x \in B_{\delta_{\mathcal{U}}}(0)$. Therefore $\eta_{(i)} \diamond \zeta_{(i)} \co B_3^{k_i}(0) \ra \R^m$ makes sense and $\|\eta_{(i)} \diamond \zeta_{(i)} (x)\| < \min(\ep_{\mathcal{U}},\ep_{\partial})$ for all $x\in B_3^{k_i}(0)$ and so 
\begin{align*}
\ps^i_{\eta_{(i)} \diamond \zeta_{(i)}} = \ps^i_{\eta_{(i)}} \ci \ps^i_{\zeta_{(i)}}
\end{align*}
on $B_3^{k_i}(0)$ (note $\ep_\mathcal{U} < \ep_{\exp}$).
Now we show that $\eta_{(i)}\diamond \zeta_{(i)}$ is stratified. With $r_\diamond \leq \min(\nu_\a,\ep_\partial)$ we calculate for $x \in B_3^{k_i}(0)$:
\begin{align}\label{EckPartundNuAlph}
&\|\ps^i_{\eta_{(i)}} \ci \ps_{\zeta_{(i)}}(x) - x\| \leq  \|\ps^i_{\eta_{(i)}} (\ps_{\zeta_{(i)}}(x)) - \ps_{\zeta_{(i)}}(x)\| + \|\ps_{\zeta_{(i)}}(x) -x\| \nonumber\\
&<\min\left(\nu_\a,\frac{\ep_\partial}{2}\right).
\end{align}
Now let $x \in \partial^j B_3^{k_i}(0)$ and $C$ the connected component of $x$ in $\partial^jB_5^{k_i}(0)$. We define $y:= \ps^i_{\eta_{(i)}}\ci \ps^i_{\zeta_{(i)}}(x)$ and $v:= \eta_{(i)} \diamond \zeta_{(i)}(x) \in B_{\ep_{\partial}}(0)$. With Lemma  Remark \ref{OBdAConnected} we deduce $x,y \in C$ and $y =\exp^i_x(v)$. With (\ref{EckPartundNuAlph}) and Lemma \ref{invEcken} we deduce $v \in T_x\partial^jB_5^{k_i}(0)$. Hence $\eta_{(i)}\diamond \zeta_{(i)}$ is stratified.
\item The maps $\ps_\eta$ and $\ps_\zeta$ make sense, because $r_\diamond <\ep_{\exp}$. With Lemma \ref{Hilfreich2} we calculate
\begin{align*}
&\ps_\eta(\ps_\zeta(U_{i,3})) = \ps_\eta (\ps_\zeta (\ph_i\inv(B_3^{k_i}(0)))) = \ps_\eta(\ph_i\inv(\ps^i_{\zeta_{(i)}}(B_3^{k_i}(0) ) ))\\
=& \ph_i\inv (\ps_{\eta_{(i)}}^i \ci \ps_{\zeta_{(i)}}^i (B_3^{k_i}(0))).
\end{align*}
This calculation also shows (\ref{diamondsplitt2}).
\item Let $p \in M$ and $i \in \set{1,\dots,n}$ with $p \in U_{i,1}$. Let $x:= \ph_i(p) \in B_1^{k_i}(0)$ and $v:= T\ph_i\inv(x,\eta_{(i)}\diamond \zeta_{(i)} (x)) \in T_pM$. Then $v \in \Omega$ and because $\ep_\mathcal{U}\leq \ep_{inj}$ we get $v \in W_pM$. Now we calculate 
\begin{align*}
\exp(v) = \ph_i\inv (\exp_i(x, \eta_{(i)} \diamond \zeta_{(i)} (x))) = \ph_i\inv \ci \ps_{\eta_{(i)}}^i \ci \ps_{\zeta_{(i)}}^i \ci \ph_i (p) = \ps_\eta \ci \ps_\zeta (p).
\end{align*}
Hence the vector field $\eta\diamond \zeta$ makes sense. Now let $x\in B_3^{k_i}(0)$ and $p := \ph_i(x)$. We calculate
\begin{align*}
&(\eta\diamond \zeta)_{(i)} (x) = d\ph_i \ci \exp|_{W_pM}\inv \ci \ps_\eta \ci \ps_\zeta (p)\\
=& d_p \ph_i \ci \ci \exp|_{W_pM}\inv \ci \ph_i\inv \ci \ps_{\eta_{(i)}}^i \ci \ps_{\zeta_{(i)}}^i (\ph_i(p))\\
=& (\ph_i \ci \exp|_{W_pM}\inv \ci d_p \ph_i\inv)\inv \ci \ps_{\eta_{(i)}}^i \ci \ps_{\zeta_{(i)}}^i (x) \\
=& \exp_i(x,\bl)\inv (\ps_{\eta_{(i)}}^i \ci \ps_{\zeta_{(i)}}^i (x))= \eta_{(i)} \diamond \zeta_{(i)}(x).
\end{align*}
Obviously we have $\ps_{\eta\diamond \zeta} = \ps_{\eta} \ci \ps_\zeta$. The vector field $\eta\diamond \zeta$ is stratified, because  its local representation is stratified.
\item First we show
\begin{align*}
\|(\eta\diamond \zeta)_{(i)}\|_{\ol{B}_1^{k_i}(0)}^1 = \|\eta_{(i)} \diamond \zeta_{(i)}\|_{\ol{B}_1^{k_i}(0)}^1 <\ep_\mathcal{U}
\end{align*}
for all $i \in \set{1,\dots, n}$. From $\ps^i_{\eta_{(i)}} \ci \ps_{\zeta_{(i)}}^i (B_3^{k_i}(0)) \subs B_{\delta_\mathcal{U}}(0)$ we get $\|\eta_{(i)} \diamond \zeta_{(i)}\|_{\ol{B}_1^{k_i}(0)}^0 <\ep_\mathcal{U}$. Now we show $\|(\eta_{(i)} \diamond \zeta_{(i)})'(x)\|_{op} <\ep_\mathcal{U}$ for all $x\in B_3^{k_i}(0)$. We define the auxiliary function $h:= \ps^i_{\eta_{(i)}} \ci \ps_{\zeta_{(i)}} - \id$ on $B_3^{k_i}(0)$. Let $x \in B_3^{k_i}(0)$. Using (\ref{EckPartundNuAlph}) we see $\|h(x)\|<\nu_\a$. With $r_\diamond \leq \min(\frac{\ep_\mathcal{U}}{2K_\a} \cdot \frac{1}{4},1)$ we calculate
\begin{align*}
&\|h'(x)\|_{op} = \|{\ps^i_{\eta_{(i)}}}'(\ps^i_{\zeta_{(i)}}(x)) \ci {\ps^i_{\zeta_{(i)}}}'(x) - \id\|_{op}\\
\leq & \|{\ps^i_{\eta_{(i)}}}'(\ps^i_{\zeta_{(i)}}(x)) \ci {\ps^i_{\zeta_{(i)}}}'(x) - \id\ci {\ps^i_{\zeta_{(i)}}}'(x)\|_{op} + \|{\ps^i_{\zeta_{(i)}}}'(x) - \id\|_{op} \\
\leq& \|{\ps^i_{\eta_{(i)}}}'(\ps^i_{\zeta_{(i)}}(x)) - \id\|_{op} \cdot \|{\ps^i_{\zeta_{(i)}}}'(x)\|_{op} + \|{\ps^i_{\zeta_{(i)}}}'(x) - \id\|_{op}\\
<& \frac{\ep_\mathcal{U}}{2K_\a} \cdot \frac{1}{4} \cdot  2 +  \frac{\ep_\mathcal{U}}{2K_\a} \cdot \frac{1}{2} = \frac{\ep_\mathcal{U}}{2K_\a}.
\end{align*}
For all $x\in B_3^{k_i}(0)$ we have 
\begin{align*}
\eta_{(i)}\diamond\zeta_{(i)}(x) = \a_i\left(x, (\ps_{\eta_{(i)}}^i \ci \ps_{\zeta_{(i)}}^i -\id)(x)\right)  = \a_i(x,h(x)).
\end{align*}
Let $x\in B_3^{k_i}(0)$ and $v\in \R^m$. We show $\|d(\eta_{(i)}\diamond\zeta_{(i)})(x;v)\| \leq \ep_\mathcal{U} \|v\|$:  
\begin{align*}
&\|d(\eta_{(i)}\diamond\zeta_{(i)})(x;v)\| = \|d\a_i(x,h(x); v,dh(x,v))\| \\
\leq & \|d_1\a_i(x,h(x); v) \| + \|d_2\a_i(x,h(x); dh(x,v))\|\\
\leq & \left(\|d_1\a_i(x,h(x); \bl) \|_{op} + \|d_2\a_i(x,h(x);\bl) \|_{op} \cdot \|dh(x,\bl))\|_{op} \right) \cdot \|v\|\\
< & \frac{\ep_\mathcal{U}}{2} + K_{\a} \cdot \frac{\ep_\mathcal{U}}{2K_\a} = \ep_\mathcal{U}.
\end{align*}
\end{compactenum}
\end{proof}

\begin{lemma}
For all $m \in \N$ we define $m':= 2m+N_0$. If $\eta \in \Hol_{\ep_\diamond}^1(U_{4,n}; \C^m)_\str^\R$, $m \in \N$ with $m\geq n$ and $R\in ]1,4[$, then
\begin{align*}
\ps_\eta^i(\ol{U}^{k_i}_{R-1,m'}) \subs U^{k_i}_{R,m} \tx{and} \|\ps_\eta^i - \id\|^0_{\ol{U}^{k_i}_{3,n'}} < r_\diamond^\C =\frac{\delta^\C(\ep_{\exp^\ast})}{2}.
\end{align*}
\end{lemma}
\begin{proof}
This follows from Lemma \ref{Nahe1II} and Lemma \ref{WennNahe1II} with $r_0<1$.
\end{proof}

Bowling down \cite[Proposition 4.23]{Gloeckner1} to our situation we obtain the following Lemma.
\begin{lemma}\label{EssentialI}
Let $U,Z,U_g,V_g\subs \C^m$ be open subsets, such that $\ol{Y}\subs U\subs U_g$ and $\ol{Y}$ is compact (hear the closure of $Y$ is taken in $\C^m$). If $g\co U_g\ti V_g \ra \C^m$ is a complex analytic map, then 
\begin{align*}
\Hol(U;\C^m)\cap \lfloor\ol{Y},V_g \rfloor \ra \Hol(Y;\C^m),~ \gamma \ms g(\bl,\gamma(\bl))|_Y
\end{align*}
is a complex analytic map.
\end{lemma}

By applying \cite[Proposition 11.4]{Gloeckner1} to our situation we obtain the following Lemma.  
\begin{lemma}\label{EssentialII}
If $U,Z,Y \subs \C^m$ are open subsets such that $\ol{Y}$ is compact and contained in $Z$, then
\begin{align*}
\Hol(U;\C^m)\ti (\Hol(Z;\C^m)\cap \lfloor \ol{Y};U\rfloor) \ra \Hol(Y;\C^m),~ (\eta,\tau)\ms \eta\ci \tau|_{Y}
\end{align*}
is complex analytic.
\end{lemma}

\begin{lemma}
The map
\begin{align*}
\diamond\co \mathcal{B}_{4,\ep_\diamond} \ti \mathcal{B}_{4,\ep_\diamond} \ra \mathcal{B}_{1,\ep_\mathcal{U}},~ (\eta, \zeta)  \ms \eta\diamond \zeta 
\end{align*}
is smooth.
\end{lemma}
\begin{proof}
In this proof, all vector spaces except of $\C^m$ are considered as $\R$ vector spaces. If $k\in \N$ we write $k':=2k+N_0$.
After identification we have $\mathcal{B}_{1,\ep_\mathcal{U}} \subs \prod_{i=1}^n \mathcal{G}(\C^m;\C^m|\ol{B}_1^{k_i}(0))^\R_\str$ and it is left to show the smoothness of
\begin{align*}
\mathcal{B}_{4,\ep_\diamond} \ti\mathcal{B}_{4,\ep_\diamond} \ra \mathcal{G}(\C^m;\C^m|\ol{B}_1^{k_i}(0)), ~ (\eta,\zeta) \ms \left[ {{(\eta\diamond \zeta)}^\ast}_{(i)} \right]_{\ol{B}^{k_i}_1(0)} = \left[ {{(\eta_{(i)} \diamond \zeta_{(i)})}^\ast} \right]_{\ol{B}^{k_i}_1(0)}
\end{align*}
We consider the diagram
\begin{align*}
\begin{xy}
\xymatrixcolsep{8pc}\xymatrix{
\mathcal{B}_{4,\ep_\diamond} \ti \mathcal{B}_{4,\ep_\diamond} \ar@{_{(}->}[d] \ar[r]^-{\diamond} & \mathcal{G}(\C^m;\C^m|\ol{B}_1^{k_i}(0))^\R_\str\\
\left(\prod_{i=1}^n G_{\ep_\diamond}^1(\C^m;\C^m| \ol{B}_4^{k_i}(0))_\str^\R \ar[r]^-{\pr_i} \right)^2
\ar[r]^-{\pr_i} & \left(G_{\ep_\diamond}^1(\C^m;\C^m| \ol{B}_4^{k_i}(0))_\str^\R\right)^2 \ar[u]_-{\Phi}
}
\end{xy} 
\end{align*}
with $\Phi([f], [g]) = \left[\left(f|_{\ol{B}_4^{k_i}(0)}\diamond g|_{\ol{B}_4^{k_i}(0)}\right)^\ast \right]$. Because $\mathcal{G}( \C^m ; \C^m | \ol{B}_4^{k_i}(0) )_\str^\R = \varinjlim \Hol_b^1(U^{k_i}_{4,n};\C^m)_\str^\R$ we have to show the smoothness of 
\begin{align*}
\left(\Hol_{\ep_\diamond}^1 (U_{4,n}^{k_i};\C^m)_\str^\R\right)^2 \ra \mathcal{G}(\C^m ; \C^m |\ol{B}_1^{k_i}(0))_\str^\R,~
(\eta, \zeta) \ms \left[ {\eta|_{\ol{B}_4^{k_i}(0)}\diamond \zeta|_{\ol{B}_4^{k_i}(0)}}^\ast \right] = \left[ \eta\diamond \zeta\right]
\end{align*}
with $\eta\diamond \zeta \co U_{1,n'''}^{k_i} \ra \C^m$, $x\ms \exp_i^\ast(x,\bl)\inv ({\ps^i_\eta}\inv(x) - x)$.
Therefore we show the smoothness of $(\Hol_{\ep_\diamond}^1 (U_{4,n}^{k_i};\C^m)_\str^\R)^2 \ra \Hol (U_{1,n'''}^{k_i};\C^m)_\str^\R$, $(\eta,\zeta) \ms \eta\diamond \zeta$.
We write this map as the following composition:
\begin{align*}
&\left(\Hol_{\ep_\diamond}^1(U_{4,n}^{k_i};\C^m)_\str^{\R}\right)^2  \xrightarrow{\circled{1}} \Hol(U_{4,n} ; \C^m) \cap \lfloor \ol{U}_{3,n'}; B^\C_{r^\C_\diamond}(0)\rfloor +\id\\
&\ti \Hol(U_{4,n} ; \C^m) \cap \lfloor \ol{U}_{3,n'}; B^\C_{r^\C_\diamond}(0)\rfloor +\id  \cap \lfloor \ol{U}_{2,n''}; U_{3,n'}\rfloor\\
&\xrightarrow{\circled{2}} \Hol(U_{2,n''};\C^m) \cap \lfloor \ol{U}_{1,n'''}^{k_i} ; B_{\delta^\C(\ep_{\exp^\ast})}^{\C}(0) \rfloor+\id
\xrightarrow{\circled{3}} \Hol(U_{1,n'''} ;\C^m)\\
&\text{with } \circled{1}(\eta,\zeta)= (\ps^i_\eta,\ps_\zeta^i) , ~\circled{2}(f,g) = f\ci g \tx{and}  \circled{3}(g) = x\ms \exp_i^\ast(x,\bl)\inv(g(x)).
\end{align*} 
Now the smoothness of $\circled{1}$ and $\circled{3}$ follows from Lemma \ref{EssentialI} and the smoothness of $\circled{2}$ from Lemma \ref{EssentialII}.
\end{proof}

\section{Smoothness of the inversion}\label{SecSmoothnes of the inversion}

In this section we proof the smoothness of the inversion in the local chart $\gph$.

\begin{definition}\label{AstaaaD}
We use the constants of Lemma \ref{AllesKlein}. Let $r_{\star} := \min\left(\frac{1}{4},\delta_\mathcal{U}, \frac{\ep_\mathcal{U}}{2\cdot K_{\a}}, \nu_\a , \frac{\ep_\partial}{2} \right)$ and $\ep_\ast:= \min (\ep_\mathcal{U},\ep_{4,r_\star})$.
\end{definition}

In the case of an smooth manifold without corners one can use smooth bump functions to show the smoothness of the inversion \cite[Lemma 3.8]{Gloeckner2}. Obviously this is not possible in the real analytic case. As in \cite[Lemma D.4]{BobDis} we use a quantitative inverse function theorem to show the smoothness of the inversion. Therefore it was necessary to show the quantitative inverse function theorem for open sets with corners (Lemma \ref{QanUKFEcke}) to apply it to show Lemma \ref{WennNahe1}.
\begin{lemma}\label{llAstaaa}
Let $\eta \in \mathcal{B}_{4,\ep_\star}$ and $i\in \set{1,\cdots,n}$ 
\begin{compactenum}
\item\label{Astaaa} The map ${\ps_{\eta_{(i)}}^i}\inv\co B_2^{k_i}(0)\ra \R^m_{k_i}$ makes sense. Moreover the map  
\begin{align*}
(\eta_{(i)})^\star \co B_2^{k_i}(0)\ra \R^m,~  x\ms \log_i\left(x,{\ps_{\eta_{(i)}}^i}\inv(x)\right)
\end{align*}
makes sense and is a stratified real analytic function with $\|(\eta_{(i)})^\star(x)\|< \min(\ep_\mathcal{U}, \ep_\partial)$ for all $x\in B_2^{k_i}(0)$ and $\ps^i_{(\eta_{(i)})^\star} = (\ps^i_{\eta_{(i)}})\inv$ on $B^{k_i}_2(0)$.
\item The map $\ps_\eta\inv\co M\ra M$ makes sense and $\ps_\eta\inv(U_{i,2}) \subs U_{i,4}$. Moreover $\ps_\eta\inv|_{U_{i,2}} = \ph_i\inv \ci (\ps^i_{\eta_{(i)}})\inv \ci \ph_i|_{U_{i,2}}$.
\item The vector field 
\begin{align*}
\eta^\star \co M\ra TM, ~ p \ms \exp|_{W_pM}\inv(\ps_\eta\inv(p))
\end{align*}
makes sense, is real analytic and stratified. Moreover we have $\ps_{\eta^\star} = \ps_\eta\inv$ and 
\begin{align}\label{lokalAst}
(\eta^\star)_{(i)}|_{B_2^{k_i}(0)} = (\eta_{(i)})^\star.
\end{align}
\item We have $\eta^\star \in \mathcal{V}= \mathcal{B}^1_{1,\ep_\mathcal{U}}$.
\end{compactenum}
\end{lemma}
\begin{proof}
\begin{compactenum}
\item Because $\ep_\star \leq \ep_{4,r_\star}$ and $r_\star\leq \min(\delta_U,\frac{1}{4}, \frac{\ep_\partial}{2})$ we can use Lemma \ref{Nahe1} and see that ${\ps_{\eta_{(i)}}^i}\inv\co B_2^{k_i}(0)\ra \R^m_{k_i}$ makes sense and that $\|{\ps_{\eta_{(i)}}^i}\inv(x)-x\|< \min(\delta_\mathcal{U}, \frac{\ep_\partial}{2})$ for all $x\in B_2^{k_i}(0)$. Hence the map $(\eta_{(i)})^\star \co B^{k_i}_2(0)\ra \R^m$, $x\ms \log_i(x,{\ps_{\eta_{(i)}}^i}\inv(x))$ makes sense and $\|(\eta_{(i)})^\star (x)\| < \min(\ep_\mathcal{U}, \ep_{\partial}) \leq \ep_{\exp}$ for all $x\in B_2^{k_i}(0)$. Thus
\begin{align}\label{starst}
\exp_i(x,\eta_{(i)}^\star(x)) = {\ps_{\eta_{(i)}}^i}\inv(x) \text{ for all }x\in B_2^{k_i}(0).
\end{align}
Now let $x\in \partial^j B_2^{k_i}(0)$. Obviously $y_0:=(\ps_{\eta_{(i)}}^i)\inv(x)\in \partial^j B_4^{k_i}(0)$. Let $C$ be th connected component of $x$ in $\partial^j B_5^{k_i}$. With Remark \ref{OBdAConnected} we see $y \in C$. We have $v:= \eta^\star_{(i)}(x) \in B_{\ep_\partial}(0)$ and $\exp_x^i(v)=y$ and $\|y-x\|< \frac{\ep_{\partial}}{2}$. With Lemma \ref{invEcken} we get $v \in T_x\partial^jB_5^{k_i}(0)$
and so $(\eta_{(i)})^\star \in C^\o(B^{k_i}_2(0); \R^m)_\str$.
\item Because $\ep_\star \leq \ep_\mathcal{U}$ the map $\ps_\eta \co M\ra M$ is a real analytic diffeomorphism. Using Lemma \ref{Hilfreich2}, Lemma \ref{Nahe1} and $r_{\mathcal{U}}\leq \frac{1}{4}$ we calculate
\begin{align*}
U_{i,2} = \ph_i\inv(B_2^{k_i}(0)) \subs \ph_i\inv (\ps^i_{\eta_{(i)}}(B_4^{k_i}(0))) = \ph_i\inv\ci \ps_{\eta_{(i)}}^i \ci \ph_i(U_{i,4}) = \ps_\eta(U_{i,4}).
\end{align*}
Let $p \in U_{i,2}$. To see $\ps_\eta\inv(p) = \ph_i\inv \ci (\ps_{\eta_{(i)}}^i)\inv \ci \ph_i(p)$ we use Lemma \ref{Hilfreich2} and calculate
$\ps_\eta\left(\ph_i\inv \ci (\ps_{\eta_{(i)}}^i)\inv \ci \ph_i(p)\right) = p$.
\item We show that the map $\eta^\star\co M\ra TM$, $p \ms \exp|_{W_pM}\inv(\ps_\eta\inv(p))$ makes sense. Let $p \in M$ and $i\in \set{1,\dots, n}$ such that $p \in U_{i,1}$. We define $x:= \ph_i(p)$ and $v:= T\ph_i\inv(x,(\eta_{(i)})^\star(x))$. Because $\ep_{\mathcal{U}}\leq \ep_{inj}$ we get $v\in W_pM$. Now we calculate
\begin{align*}
\exp(v) = \ph_i\inv\ci \exp_i(x,\eta^\star_{(i)}(x)) = \ph_i\inv \ci \ps_{\eta_{(i)}}^i \ci \ph_i(p) = \ps_\eta\inv(p).
\end{align*}
Hence $\eta^\star$ makes sense. Obviously $\ps_{\eta^\star} = \ps_\eta\inv$. Now let $x\in B_2^{k_i}(0)$ and $p:=\ph_i\inv(x)$.  We calculate
\begin{align*}
&(\eta^\star)_{(i)}(x) = d_p\ph_i \ci \exp|_{W_pM}\inv \ci \ps_\eta\inv(p) = d_p\ph_i \ci \exp|_{W_pM}\inv \ci \ph_i\inv \ci (\ps^i_{\eta_{(i)}})\inv(\ph_i(p))\\
=& \left(\ph_i \ci  \exp|_{W_pM}\ci d_p\ph_i\inv  \right)\inv \ci (\ps^i_{\eta_{(i)}})\inv(x)= \exp_{i}(x,\bl)\inv \ci (\ps^i_{\eta_{(i)}})\inv(x)= (\eta_{(i)})^\star(x).
\end{align*}
This shows (\ref{lokalAst}). We conclude with this local representation and (\ref{Astaaa}) that $\eta^\star$ is real analytic and stratified.
\item It is enough to show $\|d\eta^\ast_{(i)}(x,v)\| \leq \ep_\mathcal{U}\cdot \|v\|$ for all $x\in B_2^{k_i}(0)$ and $v\in \R^m$.
Because $r_\star<\delta_{U}$ we have $(\ps_{\eta_{(i)}}^i)\inv (x)-x \in B_{\delta_\mathcal{U}}(0)$ (Lemma \ref{Nahe1}) and so 
\begin{align*}
\eta_{(i)}^\star(x)=a_i\left(x,(\ps_{\eta_{(i)}}^i)\inv (x)-x\right)
\end{align*} 
for all $x\in B_2^{k_i}(0)$. Now let $x\in B_2^{k_i}(0)$ and $v\in \R^m$. We get
\begin{align}\label{detastar}
d&\eta_{(i)}^\star(x,v)= d_1a_i \left( x , (\ps_{\eta_{(i)}}^i)\inv(x)-x );v\right) \nonumber\\
&+ d_2 a_i \left(x , (\ps_{\eta_{(i)}}^i)\inv(x)-x ; d(\ps_{\eta_{(i)}}^i)\inv(x;v)-v\right).
\end{align}
Using $r_\star \leq \nu_\a$ and Lemma \ref{Nahe1} we see $(\ps_{\eta_{(i)}}^i)\inv(x)-x ) \in B_{\nu_\a}(0)$ and so $\|d_1a_i ( x , (\ps_{\eta_{(i)}}^i)\inv(x)-x ;v))\| \leq \frac{\ep_\mathcal{U}}{2} \|v\|$. Analogously  we get 
\begin{align*}
&\|d_2 a_i (x , (\ps_{\eta_{(i)}}^i)\inv (x)-x ; d(\ps_{\eta_{(i)}}^i)\inv(x;v)-v )\| \leq K_\a \cdot \|d(\ps_{\eta_{(i)}}^i)\inv(x;\bl)-\id\|_{op} \cdot \|v\|\\
&\leq \frac{\ep_{\mathcal{U}}}{2} \|v\|.
\end{align*}
Now the assertion follows from (\ref{detastar}).
\end{compactenum}
\end{proof}

\begin{lemma}\label{SIgmaAA}
Let $n \in \N$. As before we define $n':=2n+N_0$. There exists $\sigma>0$ such that for all $x\in U^{k_i}_{3,n'}$ and $y \in B_{\sigma}^\C(0)$ we have
\begin{compactenum}[(i)]
\item $\exp^\ast_i(x,y) \in U_{4,n}^{k_i}$;
\item $\|d_2\exp_i^\ast(x,y;\bl) - \id\|_{op} <\frac{1}{2}$.
\end{compactenum}
Especially we have $\|d_2\exp^\ast_i(x,y)\|_{op} \leq \frac{3}{2}$.  
\end{lemma}
\begin{proof}
We have $\exp^\ast_i(x,0) \in U_{4,n}^{k_i}$ for all $x \in \ol{U}_{3,n'}^{k_i}$ and  so $\ol{U}_{3,n'}^{k_i} \ti \set{0} \subs \exp\inv(U_{4,n}^{k_i})$. Moreover $d_2\exp_i^\ast(x,0;\bl)- \id_{\C^m} =0$ for all $x \in \ol{U}^{k_i}_{3,n'}$. Let $\ph \co U_{4,n}^{k_i} \ti B_{\ep_{\exp^\ast}}^\C (0)\ra  [0,\8[$, $(x,y) \ms \|d_2\exp^\ast_i(x,y) - \id_{\C^m}\|_{op}$. Then $\ol{U}^{k_i}_{3,n'} \ti \set{0} \subs \ph\inv([0,\frac{1}{2}[)$. The rest follows from the Lemma of Wallace.
\end{proof}

\begin{definition}
With the notation from Lemma \ref{DeltaAst} and Lemma \ref{SIgmaAA} we define $\delta_{log^\ast}:=\min( \delta^\C(\sigma),(\frac{1}{4})^{op})$. Without loose of generality we can assume $\ep_\star < \ep^\C_{\delta_{\log^\ast}}$.
\end{definition}

\begin{lemma}\label{CUmkehrallesbls}
\begin{compactenum}
\item If $\eta \in \Hol_{\ep_\star}^1(U^{k_i}_{4,n};\C^m)_\str^\R$ then $\ps_\eta^i\co U_{4,n}^i \ra \C^m$ is stratified and $\|\ps_\eta^i - \id_{\C^m}\|^1 < \delta_{\log^\ast} \leq  (\frac{1}{4})^{op}$.
\item\label{bStar} As before we define $n':=2n+N_0$. We have $U^i_{2,n'} \subs \ps_\eta^i(U_{3,n}^i)$ and so $(\ps_\eta^i)\inv\co U_{2,n'}^i \ra \C^m$ makes sense and $\|(\ps_\eta^i)\inv - \id\|^0 < \delta_{\log^\ast} \leq \delta^\C(\sigma)$.
\item The map
\begin{align*}
\eta^\star \co U_{2,n'}^{k_i} \ra B_\sigma^\C(0), ~ x\ms \exp_i^\ast(x,\bl)\inv({\ps^i_\eta}\inv(x))
\end{align*}
makes sense and is a complex analytic extension of $\eta|_{B_4^{k_i}(0)}^\star \co B_2^{k_i}(0) \ra \R^m$.
\end{compactenum}
\end{lemma}
\begin{proof}
It is enough to show (\ref{bStar}). With $R=3, r_0\leq (\frac{1}{4})^{op}, r= \frac{1}{n+N_0}$ we use Lemma \ref{WennNahe1II} and get $R'=2$ and $r'>\frac{1}{(2n+N_0)+N_0}$.
\end{proof}

The following lemma is a consequence of \cite{BobHamza}.
\begin{lemma}\label{ExponentialGesetz}
Let $U \subs \C^m$ be open and $f\co \R \ti U \ra \C^m$ a map that is complex analytic in the second argument. In this situation the map $f$ considered as a map between finite dimensional $\R$-vector spaces is smooth if and only if the map 
$\check{f} \co \R \ra (\Hol(U;\C^m))_\R$ considered as a map between locally convex $\R$-vector spaces is smooth.
\end{lemma}

We follow the idea of \cite[Lemma 3.8]{Gloeckner2} and use an argument from the convenient setting to show the following Lemma.
\begin{lemma}\label{EndeInv}
The map $\gph \co \Hol^1_{\ep_\star} (U^{k_i}_{4,n} ; \C^m)_\str^\R \ra \Hol(U_{2,n'};\C^m)$, $\eta \ms \eta^\star$ is smooth (hear we consider the vector spaces as real vector spaces).
\end{lemma}
\begin{proof}
Let $c \co \R \ra \Hol^1_{\ep_\star} (U^{k_i}_{4,n} ; \C^m)_\str^\R$, $t \ms c_t$ be smooth.  We have to show that $\what{\gph \ci c} \co \R \ti U_{2,n'} \ra \C^m \cong \R^{2m}$, $(t,x) \ms c_t^\star (x)$ is smooth. From Lemma \ref{CUmkehrallesbls} we get $c_t^\star \in \Hol_\sigma^0(U_{2,n'};\C^m)$. Hence $\ps^i_{c_t^\star}(U_{2,n'}^{k_i}) \subs U_{4,n}^{k_i}$. Therefore $\ps^i_{c_t} \ci \ps_{c_t^\star}$ makes sense and
\begin{align}\label{ddddddddddd}
&\ps^i_{c_t} \ci \ps^i_{c_t^\star}(x) =x \tx{for all} x \in {U_{2,n'}^{k_i}}\\
\Rightarrow& \exp_i^\ast(\exp_i^\ast(x, c_t^\star(x)), c_t(\exp^i(x,c_t^\ast(x))) ) -x =0 \tx{for all} x\in  {U_{2,n'}^{k_i}}
\end{align}
We define the smooth function
\begin{align*}
&\Lambda \co \R \ti U^{k_i}_{2,n'} \ti B_\sigma^\C(0) \ra \C^m\\
& (t,x,y) \ms \exp^\ast_i(\exp_i^\ast(x,y) , c_t(\exp_i^\ast(x,y))) - x = \ps_{c_t}(\exp^\ast_i(x,y)) - x
\end{align*}
For $(t,x,y) \in \R \ti U^{k_i}_{2,n'} \ti B_\sigma^\C(0)$ we calculate 
\begin{align*}
&\|d_3\Lambda (t,x,y;\bl) - \id\|_{op}  = \|d\ps_{c_t}(\exp^\ast_i(x,y); d_2\exp^\ast_i(x,y;\bl)) - \id \|_{op} \\
=& \|d\ps_{c_t}(\exp^\ast_i(x,y); \bl) \ci  d_2\exp^\ast_i(x,y;\bl) - \id \|_{op}\\
\leq & \|d\ps_{c_t}(\exp^\ast_i(x,y); \bl) \ci  d_2\exp^\ast_i(x,y;\bl) - \id  \ci d_2\exp^\ast_i(x,y;\bl) \|_{op}
+ \| \id  \ci d_2\exp^\ast_i(x,y;\bl) - \id\|_{op}\\
\leq& \|d\ps_{c_t}(\exp^\ast_i(x,y); \bl)\|_{op} \cdot  \|d_2\exp^\ast_i(x,y;\bl)\|_{op} + \|d_2\exp^\ast_i(x,y;\bl) - \id\|_{op}\\
<& \frac{1}{4} \cdot \frac{3}{2} + \frac{1}{2} = \frac{7}{8}.  
\end{align*}
Hence $d_3\Lambda (t,x,y;\bl) \in GL(\C^m)$. Obviously $\Lambda (t,x,\bl) \co B_\sigma^\C(0) \ra \C^m$ is injective. With the implicit function theorem and (\ref{ddddddddddd}) we see that $\gph$ is smooth.
\end{proof}

\begin{lemma}
The map
\begin{align*}
i_M\co \mathcal{B}_{4,\ep_\star} \ra \mathcal{B}_{1,\ep_\mathcal{U}},~ \eta \ms \eta^\star 
\end{align*}
is smooth.
\end{lemma}
\begin{proof}
After identification we have $\mathcal{B}_{1,\ep_\mathcal{U}} \subs \prod_{i=1}^n \G(\C^m;\C^m|\ol{B}_1^{k_i}(0))$ and it is left to show the smoothness of
\begin{align*}
\mathcal{B}_{4,\ep_\star} \ra \mathcal{G}(\C^m;\C^m|\ol{B}_1^{k_i}(0)), ~ \eta \ms \left[ {{(\eta^\star)}^\ast}_{(i)} \right]_{\ol{B}^{k_i}_1(0)} = \left[ {{(\eta^\star_{(i)})}^\ast} \right]_{\ol{B}^{k_i}_1(0)}
\end{align*}
Consider the diagram
\begin{align*}
\begin{xy}
\xymatrixcolsep{5pc}\xymatrix{
\mathcal{B}_{4,\ep_\star} \ar@{_{(}->}[d] \ar[r]^-{i_M} & \mathcal{G}(\C^m;\C^m|\ol{B}_1^{k_i}(0))\\
\prod_{i=1}^n G_{\ep_\star}^1(\C^m;\C^m| \ol{B}_4^{k_i}(0))_\str^\R \ar[r]^-{\pr_i} & G_{\ep_\star}^1(\C^m;\C^m| \ol{B}_4^{k_i}(0))_\str^\R \ar[u]_-{\Phi}
}
\end{xy} 
\end{align*}
with $\Phi([f]) = \left[{f|_{\ol{B}_4^{k_i}(0)}^\star}^\ast \right]$. Because $\mathcal{G}( \C^m ; \C^m | \ol{B}_4^{k_i}(0) )_\str^\R = \varinjlim \Hol_b^1(U^{k_i}_{4,n};\C^m)_\str^\R$ we have to show the smoothness of 
\begin{align*}
\Hol_{\ep_\star}^1 (U_{4,n}^{k_i};\C^m)_\str^\R \ra \mathcal{G}(\C^m; \C^m |\ol{B}_1^{k_i}(0))_\str^\R,~
\eta \ms \left[ {\eta|^\star_{\ol{B}_4^{k_i}(0)}}^\ast \right] = \left[ \eta^\star \right]
\end{align*}
with $\eta^\star \co U_{1,n''}^{k_i} \ra \C^m$, $x\ms \exp_i^\ast(x,\bl)\inv ({\ps^i_\eta}\inv(x))$.
Therefore we have to show the smoothness of $\Hol_{\ep_\star}^1 (U_{4,n}^{k_i};\C^m)_\str^\R \ra \Hol (U_{1,n''}^{k_i};\C^m)_\str^\R$, $\eta \ms \eta^\star$. But this follows from Lemma \ref{EndeInv}.
\end{proof}

\section{Existence and uniqueness of the Lie group structure}\label{SecExistence and uniqueness of the Lie group structure}
In this section we follow the strategy of \cite[Section 5]{Gloeckner2}: First we use the theorem about the local description of Lie groups to obtain a Lie group structure on a subgroup $\Diff(M)_0$ of $\Diff(M)$. Then we show that this structure dose not depend on the choice of the Riemannian metric (Lemma \ref{UnabhaengigkeitMetrik}). With the help of this result we show the smoothness of the conjugation map (Lemma \ref{ConjugationGlatt}).

The following Lemma comes from \cite[Proposition 1.20]{Gloeckner2}:
\begin{lemma}[Theorem about the local description of Lie groups]\label{LokaleBeschreibung}
Given a group $G$ and a subset $U \subs G$ that is a smooth manifold, such that there exists a symmetric subset $V \subs U$ that contains the identity and fulfils  $V \cdot V \subs U$, we get from \cite[Proposition 1.20]{Gloeckner2}  the following results: If the restriction of the inversion and the multiplication on $V$ are smooth, then there exists a unique manifold structure on $\langle V \rangle$ such that 
\begin{compactitem}
\item $\langle V \rangle$ is a Lie group; 
\item $V$ is open in $\langle V\rangle$;
\item $U$ and $\langle V\rangle$ induce the same manifold structure on $V$.
\end{compactitem}
Moreover if $\langle V \rangle$ is a normal subgroup of $G$ and for all $g \in G$ the conjugation $\interi_g \co \langle V \rangle \ra \langle V \rangle $, $h \ms ghg\inv$ is smooth, then there exists a unique manifold structure on $G$ such that 
\begin{compactitem}
\item $G$ becomes a Lie group;
\item $V$ is an open submanifold of $G$.
\end{compactitem}
\end{lemma}

The following Lemma is analogous to \cite[6.2 (b)-(c)]{Gloeckner2}.
\begin{lemma}\label{UnabhaengigkeitMetrik}
\begin{compactenum}
\item There exists $\ep_0\in ]0,\ep_\star[$ such that $\mathcal{B}_{4.1,\ep_0}^1 \subs i_M\inv  (\mathcal{B}_{4,\ep_\star}^1)$.
\item The set $\mathcal{U}_0^1 := \gps(\mathcal{B}_{4.1,\ep_0}^1)$ is an open connected $1$-neighbourhood. Thus the set $\mathcal{U}_0:= \mathcal{U}_0^1 \cup (\mathcal{U}_0^1)\inv \subs \mathcal{U}_\star$ is an open connected symmetric $1$-neighbourhood. We define $\mathcal{V}_0 := \gph\inv(\mathcal{U}_0)$ and $\mathcal{V}^1_0 := \mathcal{B}_{4.1,\ep_0}^1$.
\item \label{LieGroupStructure}Analogous to \cite[6.2]{Gloeckner2} we use Lemma \ref{LokaleBeschreibung} to find a unique Lie group structure on $\Diff(M)_0:=\langle \mathcal{U}_0\rangle $.
\item The Lie group structure in (\ref{LieGroupStructure}) is independent of the choice of the atlas $\ph_i \co U_{i,5} \ra B_5^{k_i}(0)$.
\item The Lie group structure in (\ref{LieGroupStructure}) is independent of the choice of the Riemannian metric $g$.
\end{compactenum}
\end{lemma}
\begin{proof}
\begin{compactenum}
\item Reconsidering the proof of Lemma \ref{del} we can achieve that $\log_i$ is also defined and real analytic on the bigger set $\bigcup_{x\in B_{4.2}(0)}\set{x} \ti B_{\delta_{\mathcal{U}}} (x)$. Hence we consider the maps 
\begin{align*}
&\log_i \co \bigcup_{x\in B_{4.2}(0)} \set{x} \ti B_{\delta_{\mathcal{U}}} (x) \ra B_{\min(\ep_\mathcal{U},\ep_\partial)}(0),~ (x,y)\ms \exp(x,\bl)\inv(y) \tx{and}\\
&\a_i \co B_{4.2}(0)\ti B_{\delta_\mathcal{U}} \ra B_{\ep_\mathcal{U}}(0),~(x,y)\ms \log_i(x,x+y).
\end{align*}
Reconsidering the Definition \ref{AllesKlein} we can achieve 
\begin{align*}
\|d_1\a_i(x,y)\|_{op} <\frac{\ep_\mathcal{U}}{2}  
\end{align*}
for all $x\in \ol{B}_{4.1}(0)$, $y \in \ol{B}_{\nu_\a}(0)$, $i \in \set{1,\dots,n}$ and
\begin{align*}
K_\a\geq \sup \set{\|d_2\a_i(x,y)\|_{op} : x\in \ol{B}_{4.1}(0), y \in \ol{B}_{\nu_\a}(0)}
\end{align*}
for all $i\in \set{1,\dots,n}$. We find $r>0$ such that $4 \leq (1-r)4.1 -r$. Now the assertion follows from Definition \ref{AstaaaD} and Lemma \ref{llAstaaa} with $r$ instead of $4$.
\item As in \cite{Gloeckner2} we have $\iota\inv(\mathcal{U}_0^1) = \iota(\mathcal{U}_0^1)$.
\item Clear.
\item Let $\ph_i' \co U_{i,5}' \ra B_5^{k_i'}(0)$ with $i \in \set{1,\dots,n'}$ be an other atlas with the same properties. 
We find an analogous open $0$-neighbourhood $V_0'\subs \gg^\o_\str(TM)$ such that $\mathcal{U}_0':=\gph(V_0')$ generates a Lie group $\Diff(M)'_0 \subs \Diff(M)$. We define the open $0$-neighbourhood  $W:= \mathcal{V}_0\cap \mathcal{V}_0'$. Then $X:= \gph(W) \subs \Diff(M)$. Obviously we get $X \subs \Diff(M)_0\cap \Diff(M)_0'$. Moreover $X$ is open in $\Diff(M)_0$ and in $\Diff(M)_0'$. Because both Lie groups are connected we get $\Diff(M)_0 = \langle X \rangle  = \Diff(M)_0'$ in the sense of sets and in the sense of Lie groups.
\item Let $g'$ be an other Riemannian metric with the same properties as $g$. In the following all objects induced by $g'$ are written with an extra ``$'$''. We choose $\ep>0$ with
\begin{align*}
\ep< \min\left(\ep_{4.1,\delta_\mathcal{U}}',
\ep^{\C}_{4.1,\delta^\C(\frac{1}{2})}, {\ep'}_{\delta_{\log^\ast}}^{\C}\right)
\end{align*}
If $\eta \in C^\o(B_5^{k_i}(0); \C^m)_\str$ and $\|\eta\|_1 <\ep$, then the map $\eta\dagger B_{4.1}^{k_i}(0) \ra \R^m$, $x\ms \log^i(x,\exp_i'(x,\eta(x)))$ makes sense and is real analytic. If $\eta \in \gg_\str^\o(TM)$ and $\|\eta\|<\ep$ then the map $\eta^\dagger \co M \ra TM$, $p \ms \exp|_{W_pM}\inv(\exp'(\eta(p)))$ makes sense and is a stratified, real analytic vector field  of $M$: To see this one chooses a chart $\ph_i$ around $p$ and  defines $x:= \ph_i(x)$ and $v:= \log_i(x,\exp_i'(x,\eta_{(i)}(x)))$. With $v\in B_{\min(\ep_\mathcal{U},\ep_\partial)}$ one directly gets $T\ph_i\inv(x,v) \in W_pM$. Obviously we have $(\eta^\dagger)_{(i)} = (\eta_{(i)})^\dagger$. Now again let $\eta \in C^\o(B^{k_i}_5(0);\R^m)_\str$ with $\|\eta\|_1 <\ep$. We claim $\|\eta^\dagger\|_1 <\ep_0$. We have $\eta^\dagger (x) = \a_i(x,\exp_i'(x,\eta(x)-x))$ for all $x\in B_{4.1}^{k_i}(0)$. Hence 
\begin{align*}
d\eta^\dagger(x,v) = d_1\a_i(x, {\ps'}_\eta^{i}(x)-x;v) + d_2\a_i(x,{\ps'}_\eta^{i}(x)-x; d{\ps'}_\eta^i(x,v)- v)
\end{align*}
and so 
\begin{align*}
&\|d\eta^\dagger(x,\bl)\|_{op} \leq \|d_1\a_i(x, {\ps'}_\eta^i(x)-x;\bl)\|_{op} + \|d_2\a_i(x, {\ps'}_\eta^i(x)-x;\bl)\|_{op} \\
&\cdot \|d{\ps'}_\eta^i(x)-\id\|_{op}< \frac{\ep_0}{2} + K_\a\cdot \frac{\ep_0}{K_\a\cdot 2} = \ep_0.
\end{align*}
Hence the map 
\begin{align*}
\Delta \co \mathcal{B}_{4.1,\ep}^1 \ra \mathcal{B}_{4.1,\ep_0}^1,~ \eta \ms \eta^\dagger
\end{align*}
makes sense. This also shows $\gps'(\mathcal{B}_{4.1,\ep}^1) \subs \mathcal{U}_0$ and $\Delta$ is nothing else then the inclusion in the charts $\gph'$ and $\gph$. Now we want to show that $\Delta$ is smooth. To this point we have to show the smoothness of the corresponding map between $G_\ep^1(\C^m;\C^m|B_{4.1}^{k_i}(0))$ and $G_\ep^1(\C^m;\C^m|B_{4.1}^{k_i}(0))$. Reconsidering Lemma \ref{DeltaAst} (b) we can achieve that the map $\log^\ast_i$ is defined and complex analytic on the bigger set $\bigcup_{x\in B_{4.1,r_{\log^\ast}}^\C(0)} \set{x} \ti B_{\delta^\C(\ep)}^\C (x) \subs \C^m \ti C^m$. 
If $\eta \in \Hol_\ep^1(U_{4.1,n}^{k_i};\C^m)_\str^\R$ then $\eta^\dagger \co U_{4.1,n}^{k_i}\ra \C^m$, $x\ms \log^\ast_i(x,{\ps'}_\eta^i(x))$ makes sense and is complex analytic. We want to show the smoothness of  $\Hol^1_\ep(U^{k_i}_{n,4.1}) \ra \Hol^1_\ep(U^{k_i}_{4n,4.1})$, $\eta \ms \eta^\dagger$. We can write this map as the following composition
\begin{align*}
&\Hol_{\ep}^1(U_{4.1,n}^{k_i};\C^m)_\str^{\R}  \xrightarrow{\circled{1}} \Hol(U_{4.1,2n} ; \C^m) \cap \lfloor \ol{U}_{4,4n}; B^\C_{\delta_{\log^\ast}} (0) \rfloor +\id
\xrightarrow{\circled{2}} \Hol(U^{k_i}_{4.1,4n};\C^m)\\
&\text{with } \circled{1}(\eta)= {\ps'}^i_\eta, ~\circled{2}(f) = x\ms \exp_i^\ast(x,\bl)\inv(f(x)).
\end{align*} 
Now the smoothness of $\circled{1}$ and $\circled{2}$ follows from Lemma \ref{EssentialI}. The set $\gps'(\mathcal{B}_{4.1,\ep}^1)$ is an open $1$-neighbourhood on $\Diff(M)_0'$ and with $\gps'(\mathcal{B}_{4.1,\ep}^1) \subs \mathcal{U}_0$ we get $\Diff(M)_0' \subs \Diff(M)_0$. The inclusion $\Diff(M)_0' \hookrightarrow \Diff(M)_0$ is smooth, because $\Delta$ is smooth. In the analogous way one sees $\Diff(M)_0 \subs \Diff(M)_0'$ and that $\Diff(M)_0 \hookrightarrow \Diff(M)_0'$ is smooth.  
\end{compactenum}
\end{proof}

\begin{lemma}\label{fPullBack}
Given $f\in \Diff(M)$ the map 
\begin{align*}
P_f \co \gg_\str^\o (TM) \ra \gg_\str^\o (TM),~ \eta \ms P_f\eta:= Tf \ci \eta \ci f\inv
\end{align*}
is continuous linear.
\end{lemma}
\begin{proof}
Because we can embed $\gg^\o_\str(TM)$ into $\prod_{i=1}^n \mathcal{G}(\C^m;\C^m|\ol{B}_4^{k_i}(0))$. It is enough to that $\gg^\o_\str(TM) \ra \mathcal{G}(\C^m;\C^m|\ol{B}^{k_i}_4(0))$, $\eta \ms ((P_f\eta)^\ast)_{(i)} = ((P_f\eta)_{(i)})^\ast$ is continuous. The map $\ph_i'\co f\inv(U_{i,5}) \ra B^{k_i}_{5}(0)$, $\ph_i'= \ph_i \ci f$ is a chart of $M$. If $\zeta \in \gg^\o_\str(TM)$ we write $\zeta_{\ph_i'}:= d\ph_i'\ci \zeta \ci {\ph_i'}\inv$ for the local representative. We get
\begin{align*}
(P_f\eta)_{(i)} = d\ph_i\ci P_f\eta|_{U_{i,5}}\ci \ph_i\inv = \eta_{\ph_i'}.
\end{align*}
Hence 
\begin{align*}
((P_f\eta)_{(i)})^\ast = (\eta_{\ph_i'})^\ast = (\eta^\ast)_{{{\ph}_i^\ast}'}.
\end{align*}
But because the topology of $\gg(TM_\C|M)$ dose not depend on the choice of the atlas the map $\eta \ms (\wtil{\eta}^\ast)_{{{\ph}_i^\ast}'}$ is continuous.
\end{proof}

The following Lemma and its proof are completely analogous to \cite[5.5,5.6,5.8]{Gloeckner2}. For convenience of reader we recall Gl{\"o}ckners arguments:
\begin{lemma}\label{ConjugationGlatt}
The subgroup $\Diff(M)_0$ of $\Diff(M)$ is normal and for $f \in \Diff(M)$ the conjugation $\interi_f\co \Diff(M)_0 \ra \Diff(M)_0$, $h \ms f\ci h \ci f\inv$ is smooth.
\end{lemma}
\begin{proof}
Let $f \in \Diff(M)$. Then also the pullback metric $g'$  from $g$ over $f$ induces a Riemannian exponential function $\exp' \co \Omega' \ra M$, with $\Omega' = Tf(\Omega)$ and $\exp' = f \ci \exp \ci Tf\inv|_{\Omega'}$. Hence for $\eta \in \mathcal{V}^1_0$ we get $f\ci \gps_\eta \ci f\inv = \exp'\ci P_f \eta = \gps_{P_f\eta}'$. Now we can use Lemma \ref{UnabhaengigkeitMetrik} and find a $0$-neighbourhood $\mathcal{V}_0'\subs \gg_\str^\o(TM)$ such that $\gph' \co \mathcal{V}_0' \ra \Diff(M)_0$ is a diffeomorphism onto a $1$-neighbourhood. Because the map $P_f \co \gg^\o_\str(TM) \ra \gg^\o_\str(TM)$ is continuous linear we find a $0$-neighbourhood $W \subs \mathcal{V}_0^1 \subs \gg^\o_\str(TM)$ such that $P_f (W) \subs \mathcal{V}_0'$. Hence $\interi_f \ci \gps_\eta = \gps'_{P_f \eta} \in \Diff(M)_0$ for all $\eta \in W$. Therefore $\interi_f(\gps(W)) \subs \Diff(M)_0$. And so $\interi_f(\Diff(M)_0) = \interi_f(<\gps(W)>) \subs \Diff(M)_0$. Moreover we have $\interi_f|_{\gps(W)} = \gps'\ci P_f \ci \gph|_{\gps(W)}$. Hence $\interi_f$ is smooth.
\end{proof}

Now we get the main result of this paper. As in the case of \cite{Gloeckner2} we just have to use Lemma \ref{LokaleBeschreibung} and the results above.
\begin{theorem}
There exists a unique smooth Lie group structure on $\Diff(M)$ modelled over $\gg^\o_\str(TM)$ such that for one (and hence for all) boundary respecting Riemannian metrics on $M$ the map $\eta \ms \gps_\eta$ is a diffeomorphism from an open $0$-neighbourhood of $\gg^\o_\str(TM)$ onto an open $1$-neighbourhood of $\Diff(M)$.
\end{theorem}

\appendix

\section{Basic definitions and results for manifolds with corners}\label{SecBasic definitions and results for manifolds with corners}
In this section we fix some notation and recall well known basic definitions and results about manifolds with corners for the convenience of the reader. For more details we recommend \cite{Margalef}.

\begin{definition}
We define $\R^m_l:=[0,\8[^l\times \R^{m-l}$. Given $J\subs {1,...,l}$ we define $\pr_J\co \R^m\ra \R^{\# J}$, $v\ms (v_j)_{j\in J}$. Moreover we write $X_J:= \pr_J(X)$ for $X\subs \R^m$ and $v_J:=\pr_J(v)$ for $v\in \R^m$.
\end{definition}

\begin{definition}
We say a point $x\in \R^m_l$ has index $k\leq l$, if exactly $k$ of the first $l$ components of $x$ equal $0$. This means that the maximal index-set $J\subs \set{1,...,l}$ such that $\pr_J(x) =0$ has the cardinality $\# J = k$.
\end{definition}

\begin{lemma}\label{IndexEindeutig}
Let $M$ be a manifold with corners, $p\in M$ and $\ph\co U_\ph\ra V_\ph\subs \R^m_{i_1}$ and $\ps\co U_\ps \ra V_\ps\subs \R_{i_2}^m$ charts around $p$. Suppose $\ph(p)$ has index $k$, then $\ps(p)$ has index $k$.
\end{lemma}
\begin{proof}
Without lose of generality we assume $U_\ps\subs U_\ph$. Let $l$ be the index of $\ps(p)$. First we show $k\leq l$. Suppose $l< k$. Let $J\subs \set{1,...,i_1}$ be the maximal index-set with $(\ph(p))_J=0$ and analogously $I\subs \set{1,...,i_2}$ be the  set of components of $\ps(p)$ that equal $0$, we define the $m-l$-dimensional subspace  $E:=\pr_I\inv(\set{0})= \bigcap_{i\in I}\set{x\in \R^m:x_i=0}$ of $\R^m$. We find an open $\ps(p)$-neighbourhood $V\subs E$ with $V\subs V_\ps$. 
The map $\eta :=\ph\ci \ps\inv|_{V} \co V \ra V_\ph$ is an immersion, because $\ps\ci \ph\inv\ci \eta = \id_V$.
Therefore $F:=d\eta(\ps(p),\bl)(E)$ is a $m-l$-dimensional subspace of $\R^m$. The subspace $F$ must contain a vector such that one of its  $J$ components is not equal to $0$, because otherwise $F$ would be contained in the $m-k$-dimensional subspace $\pr_J\inv(\set{0})$ and this would contradict  $m-l > m-k$. Let $j\in J$ such that $v_j\neq 0$. Without lose of generality we can assume $v_j<0$. We choose a smooth curve $\g \co ]-\epsilon, \epsilon[\ra E$ with $\g(0)= \ps(p)$, $\im(\g)\subs V$ and $d\eta(\ps(p),\g'(0)) =v$. Let $f:=\eta\ci \g \co ]-\epsilon, \epsilon[ \ra V_\ph$. Then $\im (f) \subs V_\ph \subs \R^m_{i_1}$ and so $\im (\pr_j\ci f) \subs [0,\8[$, because $j\leq i_1$. But $\pr_j\ci f(0) = \pr_j(\ph(p)) = 0$ and $(\pr_j\ci f)'(0)=v_j<0$. Hence we find $t\in ]0,\epsilon[$ with $\pr_j\ci f(t)<0$. Hence $k\leq l$. In the analogous way one shows $l\leq k$. Hence $k=l$.
\end{proof}

\begin{definition}
Let $M$ be a manifold with corners. We say a point $p\in M$ has index $j\in \set{0,...,m}$ if we find a chart $\ph\co U_\ph\ra V_\ph\subs \R^m_{i}$ around $p$ such that $\ph(p)$ has index $j$. Because of Lemma \ref{IndexEindeutig} this definition is independent of the choice of the chart $\ph$. We write $\ind(p):=j$ and define the $j$-{\it stratum} $\partial^jM:=\set{x\in M: \ind(x)=j}$. We call the $0$-stratum $\partial^0M$ the {\it interior} of $M$.
\end{definition}

\begin{lemma}
Given an $m$-dimensional manifold with corners $M$ and $j\in \set{0,...,m}$ the $j$-stratum $\partial^j$ is a $m-j$-dimensional submanifold without boundary of $M$. Obviously we have $M= \bigcup_{j\in \set{0,...,m}} \partial^jM$.
\end{lemma}
\begin{proof}
To be a submanifold, is a local property and locally $\partial^jM$ looks like $\partial^j \R^m_k$ for a $k\geq j$.
\end{proof}

\begin{definition}
Let $M$ be a manifold with corners, $p \in M$ and $v\in T_pM$. We call $v$ an inner tangent vector, if we find a smooth curve $\eta \co [0,\epsilon[ \ra M$ with $\eta(0)=p$ and $\eta'(0)=v$.
\end{definition}

\begin{lemma}
Let $M$ be a manifold with corners, $p \in M$ and  $v\in T_pM$. The following conditions are equivalent: 
\begin{compactenum}
\item The tangent vector  $v$ is inner;
\- It exists a chart $\ph$ of $M$ around $p$, such that
\begin{align*}
\ph_j(p)=0 \Rightarrow d_p\ph_j(v)\geq 0;
\end{align*}
\- For all charts $\ph$ of $M$ around $p$ we have
\begin{align*}
\ph_j(p)=0 \Rightarrow d_p\ph_j(v)\geq 0.
\end{align*}
\end{compactenum}
\end{lemma}

\section{Proof of Theorem \ref{EnvelopTheo}}\label{ProofEnv}

As mentioned in section \ref{SecEnveloping Manifold} we can use our Lemma \ref{GottVerdammt} to proof Theorem \ref{EnvelopTheo}. The proof is completely analogous to \cite[Proposition 1]{Bruhat}.
\begin{proof}
First we show the uniqueness result. Using Lemma \ref{AusdehnenGlobal} we find a neighbourhood $U$ of $M$ in $\til{M}_1$ and a real analytic map $\til{f}\co U\ra \til{M}_2$ with $\til{f}|_M=\id_M$. For the same reason we find a neighbourhood $V$ of $M$ in $\til{M}_2$ and a real analytic map $\til{g}\co V\ra \til{M}_1$ with $\til{g}|_M=\id_M$. With Lemma \ref{Inverse} we find a neighbourhood $U_1$ of $M$ in $\til{M}_1$, and a neighbourhood $U_2$ of $M$ in $\til{M}_2$ such that $\til{f}(U_1)=U_2$ and $\til{f}|_{U_1}^{U_2}$ is a real analytic diffeomorphism.

Now we construct the enveloping manifold. For $x\in M$ let $\ph_x \co U_x^1 \ra V^1_x \subs [0,\8[^m$ be a chart of $M$ around $x$. Because $M$ is normal we find relatively compact $x$-neighbourhoods $U_x^3$ and $U_x^2$ in $M$ such that
\begin{align*}
M \sups U_x^1 \sups \ol{U_x^2} \sups U_x^2 \sups \ol{U_x^3} \sups U_x^3.
\end{align*}
Hear the closures are taken in the space $M$ and hence coincide with the closures in the topological subspaces. The family $(U_x^3)_{x\in M}$ is an open cover of the compact manifold $M$ hence we find finite subcover  $(U_{x_i}^3)_{i\in I}$. We define 
$U_i^1:=U_{x_i}^1$, $V_i^1:=V_{x_i}^1$, $\ph_i:= \ph_{x_i}$,
$U_i^2:=U_{x_i}^2$, $V_i^2:=\ph_i(U_i^2)$,
$U_i^3:=U^3_{x_i}$ and $V_i^3:=\ph_i(U_i^3)$. Hence we get 
\begin{align*}
V_i^1\sups \ol{V_i^2} \sups V^2_i \sups \ol{V_i^3} \sups V^3_i.
\end{align*}
Hear the sets $\ol{V_i^2}$ and $\ol{V_i^3}$ are compact and hence the closure in the topological subspace coincides with the closure in $[0,\8[^m$ respectively $\R^m$. Moreover we define the sets $V_{i,j}^1:=\ph_i(U_i^1\cap U_j^1)$, $V^2_{i,j}:=\ph_i(U^2_i \cap U^2_j)$ and $V^3_{i,j}:= \ph_i(U^3_i\cap U^3_j)$. Given $i,j\in I$ we  use Lemma \ref{Inverse} to find an open neighbourhood $\til{V}^1_{i,j}$ of $V^1_{i,j}$ in $\R^m$ with $\til{V}^1_{i,j}\cap [0,\8[^m = V^1_{i,j}$ and a real analytic diffeomorphism $\ps_{i,j}\co \til{V}^1_{i,j} \ra \til{V}^1_{j,i}$ with $\ps_{i,j}|_{\til{V}^1_{i,j}\cap [0,\8[^m} = \ph_j\ci \ph_i\inv|_{V^1_{i,j}}$ with inverse $\ps_{j,i}$. Because $\ol{V^2_{i,j}}\subs \ol{V^2_i}$ and $V^2_i$ is relative compact, $\ol{V^2_{i,j}}$ is compact and hence the closure in $V^2_i$ coincides with the closure in $[0,\8[^m$ respectively $\R^m$. 
Using Lemma \ref{GottVerdammt} we find an open  neighbourhood $\til{V}^2_{i,j}$ of $V^2_{i,j}$ in $\R^m$ with $\ol{V}_{i,j}^2\subs \til{V}^1_{i,j}$, $\til{V}^2_{i,j}\cap [0,\8[^m = V^2_{i,j}$ and $\ol{\til{V}^2_{i,j}}\cap [0,\8[^m = \ol{V^2_{i,j}}$. 
We can assume $\ps_{i,j}(\til{V}^2_{i,j})= \til{V}^2_{j,i}$, because $\ps_{i,j}(V^2_{i,j}) = V^2_{j,i}$. The set $\ol{V^3_i} \cap \ps_{j,i}\inv (\ol{V^3_j}\cap \ol{V^2_{j,i}})$ is compact and contained in $\til{V}^2_{i,j}$. Hence we find an open neighbourhood $\til{Z}_{i,j}$ of $\ol{V^3_i} \cap \ps_{j,i}\inv (\ol{V_j^3}\cap \ol{V^2_{j,i}})$ in $\R^m$ with $\ol{\til{Z}}_{i,j}\subs \til{V}^2_{i,j}$ and $\ps_{i,j}(\til{Z}_{i,j}) = \til{Z}_{j,i}$.
Because $\ps_{j,i}(\ol{V}^3_j \cap \ol{V^2_{j,i}}) \cap \ol{V^3_i} \subs \til{Z}_{i,j}$ we get $(\ps_{j,i}(\ol{V_j^3} \cap \ol{V^2_{j,i}}) \setminus \til{Z}_{i,j}) \cap (\ol{V^3_i} \setminus \til{Z}_{i,j}) = \emptyset$. The sets $\ps_{j,i}(\ol{V_j^3} \cap \ol{V^2_{j,i}})\setminus \til{Z}_{i,j}$ and $\ol{V_i^3} \setminus \til{Z}_{i,j}$ are closed, hence we find open disjoint sets $\til{X}_{i,j}$ and $\til{Y}_{i,j}$ with $\ps_{j,i}(\ol{V^3_j}\cap \ol{V^2_{j,i}})\setminus \til{Z}_{i,j} \subs \til{X}_{i,j}$ and $\ol{V^3_i} \setminus \til{Z}_{i,j} \subs \til{Y}_{i,j}$. Thus $\ps_{j,i}(\ol{V_j^3}\cap \ol{V^2_{j,i}}) \subs\til{Z}_{i,j} \cup  \til{X}_{i,j}$ and $\ol{V_i^3} \subs \til{Z}_{i,j} \cup \til{Y}_{i,j}$.
Because $I$ is finite the set $\bigcap_{j\in I,\ol{V^1_{i,j}}\neq \emptyset} \til{Y}_{i,j} \cup \til{Z}_{i,j}$ is open. Obviously it contains $\ol{V^3_i}$. Using Lemma \ref{GottVerdammt} we find an open set $\hat{V}^3_i$ in $\R^m$ with $\hat{V}^3_i \cap [0,\8[^m = V^3_i$, 
$\ol{\hat{V}^3_i} \cap [0,\8[^m = \ol{V^3_i}$ 
and $\hat{V}^3_i\subs \til{Y}_{i,j}\cup \til{Z}_{i,j}$ for all $j\in I$ with $\til{V}^1_{i,j}\neq \emptyset$. Now we calculate
\begin{align*}
&\ol{\ps_{i,j}(\hat{V}^3_i \cap \til{V}^2_{i,j})}\cap [0,\8[^m = \ps_{i,j}(\ol{\hat{V}^3_i \cap \til{V}^2_{i,j}})\cap [0,\8[^m = \ps_{i,j}(\ol{\hat{V}^3_i \cap \til{V}^2_{i,j}}\cap [0,\8[^m)\\
\subs&\ps_{i,j}(\ol{\hat{V}^3_i} \cap \ol{\til{V}^2_{i,j}} \cap [0,\8[^m) = \ps_{i,j}(\ol{V^3_i} \cap \ol{V^2_{i,j}} )
\end{align*}
Given $x\in V^2_i$ we find an open $x$-neighbourhood $\til{V}^2_{i,x}$ with $x\in V^2_{i,j}\Rightarrow \til{V}^2_{i,x}\subs \til{V}^2_{i,j}$. This is possible because $I$ is finite and so $\bigcap_{j\in I, x\in \til{V}^2_{i,j}}\til{V}^2_{i,j}$ is open. Now we shrink $\til{V}_{i,x}^2$ such that $x\in \ps_{j,i}(\ol{V_j^3} \cap \ol{V^2_{j,i}}) \subs \til{Z}_{i,j} \cup \til{X}_{i,j} \Rightarrow \til{V}^2_{i,x}\subs \til{Z}_{i,j} \cup \til{X}_{i,j}$ again this is possible because $I$ is finite. Suppose $\ph_i\inv (x)\notin \ol{U^3_j}$, then 
\begin{align*}
x\notin \ph_i(\ol{U^3_j} \cap U^2_i) = \ph_i(\ol{U^3_j} \cap \ol{U^2_i}) \sups \ph_i (\ph_j\inv (\ol{V^3_j} \cap \ol{V^2_{j,i}}))  = \ps_{j,i}(\ol{V^3_j} \cap \ol{V^2_{j,i}}) \sups \ol{\ps_{j,i}(\hat{V}^3_j \cap \til{V}^2_{j,i})}.
\end{align*}
Hence we can shrink $\til{V}^2_{i,x}$ such that $\ph_i\inv(x) \notin \ol{U^3_j} \Rightarrow \til{V}^2_{i,x} \cap \ps_{j,i}(\hat{V}^3_j \cap \til{V}^2_{j,i}) = \emptyset$.

We calculate
\begin{align*}
 V_{i,j}^2\cap V^2_{i,k} = \ph_i(U^2_i\cap U^2_j \cap U^2_k) = \ph_i(\ph_j\inv(\ph_j(U^2_i \cap U_j^2) \cap \ph_j (U_j^2 \cap U_k^2))).
\end{align*}
Let $S\subs \R^m$ be open with $V^2_{i,j} \cap V^2_{i,k} = S\cap [0,\8[^m$.
Given $x\in V_{i,j}^2\cap V_{i,k}^2$ we get $x\in \ps_{j,i}(\til{V}^2_{j,i}\cap \til{V}^2_{j,k})$. Analogously we get $x\in \ps_{j,i}(\til{V}^2_{j,i}\cap \til{V}^2_{j,k})$. 
Therefore we can shrink $\til{V}^2_{i,x}$ such that $x\in V^2_{i,j}\cap V^2_{i,k} \Rightarrow \til{V}^2_{i,x} \subs \ps_{j,i}(\til{V}^2_{j,i} \cap \til{V}^2_{j,k}) \cap \ps_{k,i}(\til{V}^2_{k,i} \cap \til{V}^2_{k,j})$. Moreover by replacing $\til{V}^2_{i,x}$ with $\til{V}^2_{i,x} \cap S$ we can assume $\til{V}^2_{i,x}\cap [0,\8[ \subs V^2_{i,j} \cap V^2_{i,k}$. Because $x\in S$ we get $x\in \til{V}^2_{i,x}$.  Now we shrink $\til{V}^2_{i,x}$ further by replacing $\til{V}^2_{i,x}$ with the connected component of $x$ in $\til{V}^2_{i,x}$. The maps $\ps_{i,j}|_{\til{V}^2_{i,x}}$ and $\ps_{k,j}\ci \ps_{i,k}|_{\til{V}^2_{i,x}}$ are real analytic and coincide on $\til{V}^2_{i,x}\cap [0,\8[^m \subs V^2_{i,j}\cap V^2_{i,k}$ hence they coincide on the connected set $\til{V}^2_{i,x}$. Now we define the open set $\til{V}^2_i:= \bigcup_{x\in V^2_i}V^2_{i,x}$ that is a neighbourhood of $V^2_i$ in $\R^m$ and so $\til{V}^2_i$ is also a neighbourhood of $\ol{V^3_i}$. Hence we find an open neighbourhood $\til{V}^3_i$ of $\ol{V^3_i}$  with $\til{V}^3_i \subs \til{V}^2_i \cap \hat{V}^3_i$ and $\ol{\til{V}^3_i} \subs \til{V}^2_i$. We get $\til{V}^3_i\cap [0,\8[^m = \til{V}^2_i \cap \hat{V}_i^3 \cap [0,\8[^m = \til{V}^2_i \cap V^3_i = V^3_i$ and so $\ol{V^3_i} \subs \ol{\til{V}^3_i} \cap [0,\8[^m$. We also get $\ol{\til{V}^3_i} \cap [0,\8[^m \subs \ol{\til{V}^2_i} \cap \ol{\hat{V}_i^3} \cap [0,\8[^m = \ol{\til{V}^2_i} \cap \ol{V^3_i} = \ol{V^3_i}$. Defining $\til{V}^3_{i,j}:= \til{V}^3_i \cap \ps_{j,i}(\til{V}^3_j \cap \til{V}^2_{j,i})$ for $i,j\in I$ we get $\ps_{i,j}(\til{V}^3_{i,j}) = \til{V}^3_{j,i}$.

Now we define the sets $\til{V}^3_{i,j,k}:= \til{V}^3_{i,j} \cap \til{V}^3_{i,k}$ and want to show $\ps_{i,j}(\til{V}^3_{i,j,k}) = \til{V}^3_{j,i,k}$ and $\ps_{i,j}|_{\til{V}^3_{i,j,k}} = \ps_{k,j} \ci \ps_{i,k}|_{\til{V}^3_{i,j,k}}$. If $y \in \til{V}_{i,j,k}^3$ we find $x\in V^2_i$ with $y\in \til{V}^2_{i,x}$. Hence $y \in \til{V}^2_{i,x} \cap \ps_{j,i}(\til{V}^3_j \cap \til{V}^2_{j,i}) \cap \ps_{k,i}(\til{V}^3_k \cap \til{V}^2_{k,i}) \subs \til{V}^2_{i,x}\cap \ps_{j,i}(\hat{V}^3_j \cap \til{V}^2_{j,i}) \cap \ps_{k,i}(\hat{V}^3_k \cap \til{V}^2_{k,i})$. Because of $\ph_i\inv(x) \notin \ol{U^3_j}\Rightarrow \til{V}^2_{i,x}\cap \ps_{j,i}(\hat{V}^3_j \cap \til{V}^2_{j,i}) = \emptyset$ and $\ph_i\inv(x) \notin \ol{U_k^3} \Rightarrow \til{V}^2_{i,x}\cap \ps_{k,i}(\hat{V}^3_k \cap \til{V}^2_{k,i}) = \emptyset$ we get $\ph_i\inv(x)\in \ol{U^3_j}$. Hence $x\in \ph_i(U_i^2\cap U_j^2) \cap \ph_i (U_i^2 \cap U_k^2)= V^2_{i,j}\cap V^2_{i,k}$. Therefore $\ps_{i,j}|_{\til{V}^2_{i,x}} = \ps_{k,j}\ci \ps_{i,k}|_{\til{V}^2_{i,x}}$ and $\til{V}^2_{i,x} \subs \ps_{j,i}(\til{V}^2_{j,i} \cap \til{V}^2_{j,k}) \cap \ps_{k,i}(\til{V}^2_{k,i} \cap \til{V}^2_{k,j})$. Especially $y\in \ps_{j,i}(\til{V}^2_{j,i} \cap \til{V}^2_{j,k}) \cap \ps_{k,i}(\til{V}^2_{k,i} \cap \til{V}^2_{k,j})$ and $\ps_{i,j}(y) = \ps_{k,j}\ci \ps_{i,k}(y)$.
Hence $\ps_{i,k}(y)\in \til{V}^2_{k,j}$. Because $y\in \til{V}^3_{i,j,k}= \til{V}^3_{i,j}\cap \til{V}^3_{i,k}$ we get $\ps_{i,k}(y)\in \til{V}^3_k \cap \til{V}^2_{k,j}$. Thus $\ps_{k,j}(\ps_{i,k}(y)) \in \ps_{k,j}(\til{V}^3_k\cap \til{V}^2_{k,j})$. On the other hand $\ps_{i,j}(y)\in \til{V}^3_j$, because $y\in \til{V}^3_{i,j}$. Hence $\ps_{i,j}(y)= \ps_{k,j}(\ps_{i,k}(y)) \in \til{V}^3_j \cap \ps_{k,j}(\til{V}_k^3 \cap \til{V}^2_{k,j}) = \til{V}^3_{j,k}$. Moreover we have $\ps_{i,j}(y)\in \til{V}_{j,i}^3$, because $y\in \til{V}_{i,j}^3$. Therefore $\ps_{i,j}(y)\in \til{V}^3_{j,i,k}$. We get $\ps_{i,j}(\til{V}^3_{i,j,k})\subs \til{V}_{j,i,k}^3$ and because $\ps_{i,j}\inv = \ps_{j,i}$ we see that $\ps_{i,j}(\til{V}_{i,j,k}^3) = \til{V}^3_{j,i,k}$.

Now we define the topological space  $\til{M}^1:=\coprod_{i\in I}\til{V}_i^3$ as the disjoint topological union and on $\til{M}^1$ we define an relation $\sim$: Let $x,y\in \til{M}^1$, e.g. $x\in \til{V}^3_i$ and $y\in \til{V}^3_j$. We say $x$ is equivalent to $y$ if $x\in \til{V}^3_{i,j}$, $y\in \til{V}_{j,i}^3$ and $y= \ps_{i,j}(x)$. To show that $\sim$ is an equivalence relation on $\til{M}^1$ we have to show its transitivity. Let $x\in \til{V}_i^3$, $y\in \til{V}_j^3$ and $z\in \til{V}_k^3$. Moreover let $z\in \til{V}_{k,j}^3$, $y\in \til{V}^3_{j,k}$, $y\in \til{V}_{j,i}^3$, $x\in \til{V}_{i,j}^3$, $z= \ps_{j,k}(y)$ and $y= \ps_{i,j}(x)$. Directly we get $y\in \til{V}_{j,i}^3 \cap \til{V}_{j,k}^3 = \til{V}^3_{j,i,k}$. Hence $x= \ps_{j,i}(y)\in \ps_{j,i}(\til{V}_{j,i,k}^3) \in \til{V}_{i,j,k}^3$ and $z= \ps_{j,k}(y)\in \til{V}_{k,j,i}^3$, because $y\in \til{V}_{j,i,k}^3 = \til{V}^3_{j,k,i}$. Moreover we have $\ps_{i,k}(x)= \ps_{j,k}(\ps_{i,j}(x)) =z$ and so $x$ and $z$ are equivalent. Now we define $\til{M}:=\til{M}^1/\sim$ as the topological quotient. Let $\pi \co \til{M}^1\ra \til{M},~ x\ms [x]$ be the canonical projection. 
Given $j\in I$ let $\i_j\co \til{V}_j^3 \hra \til{M}^1,~ x\ms (x,j)$ be the canonical inclusion. 
The topology on $\til{M}$ is final with respect to the maps $\pi\ci \i_i \co \til{V}_i^3 \ra \til{M}$ with $i\in I$. 
We show that the maps $\pi\ci \i_i \co \til{V}_i^3 \ra \til{M}$ are open. To this point let $U\subs \til{V}^3_i$ be open and $j\in I$. 
We calculate
\begin{align}\label{OPEN}
&\i_j\inv(\pi\inv(\pi(U))) = \i_j\inv\left( \set{(y,k)\in \til{M}^1: (\exists x\in U\subs\til{V}_i^3) ~ y\sim x }\right)\\
=&\set{y\in \til{V}^3_j: (\exists x\in U\subs \til{V}^3_i) ~ y\sim x}
=\ps_{j,i}\inv(U) \subs \til{V}^3_{j,i}.
\end{align}
Hence $\pi\ci \i_i \co \til{V}^3_i\ra \til{M}$ is continuous and open.
Now we define the maps $\ps_i \co \pi(\til{V}^3_i)\ra \til{V}^3_i,~ p \ms x$ if $\pi(x) =p$ and $x\in \til{V}_i^3$. The map $\ps_i$ is well-defined because $\ps_{i,i}=\id_{\til{V}_i^3}$. Moreover $\ps_i$ is bijective because its inverse is given by $\ps_i\inv= \pi\ci \i_i \co \til{V}^3_i \ra \pi(\til{V}^3_i),~ x\ms \p(x)$. Hence $\ps_i$ is a homeomorphism. 
To show that the maps $\ps_i$ form a real analytic atlas of $\til{M}$ we mention $\ps_i\inv(\til{V}_j^3) = \til{V}^3_{i,j}$ and calculate for $x\in \til{V}_{i,j}^3$ 
\begin{align*}
\ps_j\ci \ps_i\inv(x) = \ps_j (\pi(x)) =\ps_{j,i}(x).
\end{align*}

Now we show that $\til{M}$ is a Hausdorff space. Therefore we show $\ol{\til{V}^3_{i,j}}\subs \til{V}^2_{i,j}$. Given $y\in \til{V}^3_{i,j} \subs \til{V}^2_i$ we find $x\in V^2_i$ with $y\in \til{V}^2_{i,x}$. We want to show $x\in \ps_{j,i}(\ol{V^3_j} \cap \ol{V^2_{j,i}})$. If this not true  then $\ph_i\inv(x)\notin \ol{U^3_j}$. With $\ph\inv(x)\notin \ol{U^3_j} \Rightarrow \til{V}^2_{i,x}\cap \ps_{j,i}(\hat{V}^3_j \cap \til{V}^2_{j,i}) = \emptyset$ we get $y\notin \ps_{j,i}(\hat{V}_j^3 \cap \til{V}^2_{j,i}) \sups \ps_{j,i}(\til{V}^3_j \cap \til{V}^2_{j,i}) \sups \ps_{j,i} (\til{V}^3_{j,i})= \til{V}^3_{i,j}$. But since this is a contradiction we get $x\in \ps_{j,i}(\ol{V_j^3} \cap \ol{V^2_{j,i}})$. With $x\in \ps_{j,i}(\ol{V_j^3} \cap \ol{V^2_{j,i}}) \Rightarrow \til{V}^2_{i,x}\subs \til{Z}_{i,j}\cup \til{X}_{i,j}$ we get $y \in \til{Z}_{i,j}\cup \til{X}_{i,j}$. Moreover we have $y\in \til{V}^3_i \subs \hat{V}^3_i \subs \til{Y}_{i,j}\cup \til{Z}_{i,j}$ and with $\til{Y}_{i,j}\cap \til{X}_{i,j} = \emptyset$ we get $y\in \til{Z}_{i,j}$. Hence $\til{V}_{i,j}^3 \subs \til{Z}_{i,j}$. And therefore $\ol{\til{V}^3_{i,j}} \subs \ol{\til{Z}}_{i,j} \subs \til{V}^2_{i,j}$. Now let $p\neq q \in \til{M}$. We choose $x,y\in \til{M}^1$ with $\pi(x) = p$ and $\pi(y)=q$. Let $x\in \til{V}^3_i$ and $y\in \til{V}_j^3$. If there are an open $x$-neighbourhood $W_x\subs \til{V}_i^3$ and an open $y$-neighbourhood $W_y\subs \til{V}_j^3$ with $\pi(W_x)\cap \pi(W_y)= \emptyset$ then $\til{M}$ has to be Hausdorff, because of (\ref{OPEN}). Suppose there do not exist such neighbourhoods $W_x$ and $W_y$, then we find a sequence $(x_n)_{n\in \N}$ in $\til{V}_i^3$ and a sequence $(y_n)_{n\in \N}$ in $\til{V}_j^3$ with $x_n\sim y_n$ for all $n\in \N$. Hence $x_n \in \til{V}_{i,j}^3$ and $y_n \in \til{V}_{j,i}^3$ and so $x\in \ol{\til{V}^3_{i,j}} \subs V^2_{i,j}$ and $y \in \ol{\til{V}^3_{j,i}}\subs V^2_{j,i}$. Since $y_n = \ps_{i,j}(x_n)$ for all $n\in \N$ we get  $y= \ps_{i,j}(x)$. Therefore $y\in \til{V}^3_j \cap \ps_{i,j}(\til{V}_i^3 \cap \til{V}^2_{i,j}) = \til{V}^3_{j,i}$. With $x = \ps_{j,i}(y)$ we get $x\in \til{V}_{i,j}^3$. We conclude $x\sim y$. But this contradicts $p \neq q$.  

We define the map $\ph \co M \ra \til{M}$ by $\ph|_{U^3_i}:= \p \ci \i_i \ci \ph_i$. To see that $\ph$ is well-defined choose $p \in U^3_i\cap U^3_j$. We get $\ph_i(p)\in \til{V}^3_i$ and $\ph_j(p)\in \til{V}^3_j$ moreover we have $\ph_i(p)\in \ph_i(U^3_i \cap U^3_j) = V^3_{i,j} \subs \til{V}^3_{i,j}$, $\ph_j(p)\in \ph_j(U^3_i \cap U^3_j) = V^3_{j,i} \subs \til{V}^3_{j,i}$ and $\ps_{i,j}(\ph_i(p)) = \ph_j(p)$. Hence $\ph_i(p)\sim \ph_j(p)$. And so $\ph$ is well-defined. Now we show that $\ph$ is injective. Let $p_1,p_2\in M$ with $\ph(p_1)=\ph(p_2)$ and $p_1\in U^3_i$ and $p_2\in U^3_j$. We conclude $\ph_i(p_1)\sim \ph_j(p_2)$ and so $\ph_i(p_1)\in \til{V}_{i,j}^3 \cap [0,\8[^m$,  $\ph_j(p_2)\in \til{V}_{j,i}^3$ and $\ps_{i,j}(\ph_i(p_1)) = \ph_j(p_2)$. Hence $\ph_j(p_1)= \ph_j(p_2)$ and so $p_1=p_2$. We give $\ph(M)$ the real analytic structure such that $\ph$ becomes an real analytic diffeomorphism. If we can show that $\til{M}$ is an enveloping manifold of $\ph(M)$ we are done, because we can identify $M$ and $\ph(M)$. We have $\ph(M) = \pi(\coprod_{i\in I} V^3_i)$ with $\coprod_{i\in I} V^3_i \subs \coprod_{i\in I} \til{V}_i^3$. If $x\in V^3_i$ then $\ps_i\co \pi(\til{V}^3_i) \ra 
\til{V}^3_i$ is a chart of $\til{M}$ around $\pi(x)$. We show $\ps_i(\pi(\til{V}^3_i) \cap \ph(M))) = V^3_i$. Let $p = \pi(x)$ with $x\in \til{V}_i^3$ and $p=\p(y)$ with $y \in V_j^3$. Then $x\sim y$ and so $x\in V_i^3$ because $x= \ps_{j,i}(y) \in \ps_{j,i}(V_{j,i}^3) = \ph_i\ci \ph_j (V^3_{j,i}) \subs V_i^3$. Now let $x\in V_i^3$. Obviously $x = \ps_i (\pi(x))$ and  $\pi(x)\in \pi(\til{V}^3_i)\cap \ph(M)$. It is left to show that $\ps_i|_{\pi(\til{V}^3_i) \cap \ph(M))}^{V^3_i}$ is a chart of $\ph(M)$. To this point we show $\ps_i\ci \ph|_{\ph\inv (\pi(\til{V}^3_i)\cap \ph(M))}$ is a chart of $M$.
At first we show $\ph\inv (\pi(\til{V}^3_i) \cap \ph(M)) =  U^3_i$.  Let $p \in M$ with $\ph(p)\in \p(\til{V}^3_i) \cap \ph(M)$. We find $j \in I$ with $p \in U^3_j$  more over we find $i \in I$ and $x\in \til{V}_i^3$ with  $\ph(p)\sim x$. Hence $\ph_j(p)\sim x$. Therefore $\ps_{j,i}(\ph_j(p)) =x$ and so $\ph_i(p)=x$. We conclude $p \in U^3_i$.  Now let $p \in U^3_i$. Then $\ph(p) = \pi(\ph_i\inv(p)) \in \pi(\til{V}_i^3)$, because $\ph_i\inv(p)\in V^3_i$. Now we show $\ps_i\ci \ph|_{\ph\inv (\pi(\til{V}_i^3)\cap \ph(M))} = \ph_i$. Let $p \in \ph\inv (\pi(\til{V}_i^3)\cap \ph(M)) = U_i^3$. Then $\ps_i \ci \ph(p) = \ps_i(\pi(\ph_i(p))) = \ph_i(p)$.
\end{proof}

\end{document}